\documentclass{amsart}
\usepackage{amsaddr}

\usepackage{fullpage}
\usepackage{enumerate}
\usepackage{amssymb}
\usepackage{amsfonts}
\usepackage{amsthm}
\usepackage{amsmath}
\usepackage{mathtools}
\usepackage{bbm}
\usepackage{bussproofs}
\usepackage{mathrsfs}
\usepackage{tikz, tikz-cd}
\usepackage[all,2cell]{xy}
\usepackage{ stmaryrd } 
\usepackage{todonotes}
\usepackage{hyperref}
\usepackage[graphicx]{adjustbox}
\usepackage{graphicx}

\newtheorem{theorem}{Theorem}

\newtheorem{proposition}[theorem]{Proposition}
\newtheorem{lemma}[theorem]{Lemma}
\newtheorem{corollary}[theorem]{Corollary}

\theoremstyle{definition}
\newtheorem{example}[theorem]{Example}
\newtheorem{remark}[theorem]{Remark}
\newtheorem{definition}[theorem]{Definition}

\usepackage{amssymb}
\usepackage{amsfonts}
\usepackage{amsmath}
\usepackage{mathtools}
\usepackage{bbm}
\usepackage{bussproofs}
\usepackage{mathrsfs}
\usepackage{tikz, tikz-cd}
\usepackage[all,2cell]{xy}
\usepackage{stmaryrd} 
\usepackage{todonotes}
\usepackage{imakeidx} 
\usepackage{hyperref}

\tikzset{%
	symbol/.style={%
		draw=none,
		every to/.append style={%
			edge node={node [sloped, allow upside down, auto=false]{$#1$}}}
	}
}

\DeclareFontFamily{U}{min}{}
\DeclareFontShape{U}{min}{m}{n}{<-> udmj30}{}

\usetikzlibrary{matrix,arrows}

\DeclareMathOperator{\Ascr}{\mathscr{A}}

\DeclareMathOperator{\Cscr}{\mathscr{C}}
\DeclareMathOperator{\Dscr}{\mathscr{D}}
\DeclareMathOperator{\Escr}{\mathscr{E}}
\DeclareMathOperator{\Fscr}{\mathscr{F}}

\DeclareMathOperator{\Fbb}{\mathbb{F}}

\DeclareMathOperator{\Nbb}{\mathbb{N}}

\DeclareMathOperator{\Tbb}{\mathbb{T}}

\DeclareMathOperator{\Vbb}{\mathbb{V}}

\DeclareMathOperator{\Dcal}{\mathcal{D}}

\DeclareMathOperator{\Ocal}{\mathcal{O}}

\DeclareMathSymbol{\mathinvertedexclamationmark}{\mathord}{operators}{'074}
\DeclareMathSymbol{\mathexclamationmark}{\mathord}{operators}{'041}
\makeatletter
\newcommand{\raisedmathinvertedexclamationmark}{%
	\mathord{\mathpalette\raised@mathinvertedexclamationmark\relax}%
}
\newcommand{\raised@mathinvertedexclamationmark}[2]{%
	\raisebox{\depth}{$\m@th#1\mathinvertedexclamationmark$}%
}
\makeatother

\newcommand{\bang}{\mathexclamationmark}
\newcommand{\gnab}{\raisedmathinvertedexclamationmark}

\DeclareMathOperator{\id}{id}
\DeclareMathOperator{\op}{op}

\DeclareMathOperator{\DBun}{\mathbf{DBun}}

\DeclareMathOperator{\Spec}{Spec}
\DeclareMathOperator{\Sch}{\mathbf{Sch}}

\DeclareMathOperator{\pr}{pr} 

\DeclareMathOperator{\SMan}{\mathbf{SMan}}
\DeclareMathOperator{\R}{\mathbb{R}} 
\DeclareMathOperator{\N}{\mathbb{N}} 
\DeclareMathOperator{\Z}{\mathbb{Z}}
\DeclareMathOperator{\Sym}{Sym}
\DeclareMathOperator{\Set}{\mathbf{Set}} 
\DeclareMathOperator{\QCoh}{\mathbf{QCoh}}

\newcommand{\Mod}[1]{{#1}\textnormal{-}\mathbf{Mod}}
\newcommand{\CAlg}[1]{\mathbf{CAlg}_{#1}}
\newcommand{\Kah}[2]{\Omega^{1}_{{#1}/{#2}}}
\DeclareMathOperator{\SManPC}{\mathbf{SMan}_{pc,H}}

\usepackage{amssymb}
\usepackage{latexsym}
\usepackage{graphicx}

\title[Morphisms and the Cotangent Sequence in Tangent Categories]{Important Classes of Morphisms and the Relative Cotangent Sequence in Tangent Categories}
\author[J.-S. Lemay]{Jean-Simon Pacaud Lemay}
\address{Macquarie University, School of Mathematical and Physical Sciences}
\author[G.\@ Vooys]{Geoff Vooys}
\address{University of Calgary, Department of Mathematics and Statistics}
\email{js.lemay@mq.edu.au, gmvooys@ucalgary.ca}
\thanks{For this research, J.-S. P. Lemay was funded by an ARC DECRA award (\# DE230100303) and this material is based upon work supported by the AFOSR under award number FA9550-24-1-0008.}
\date{\today}

\numberwithin{theorem}{subsection}
\numberwithin{equation}{subsection}

\begin{document}

\begin{abstract}
In this paper we provide a deep and systematic study of what it means to be an immersion, a submersion, a local diffeomorphism, and unramified in a tangent category. We also give a systematic study of the ways in which these classes of morphisms interact, their properties, and give very explicit and concrete characterizations of how each class appears in algebraic geometry, differential geometry, algebra, and in Cartesian differential categories. Additionally, we discuss the notion of being carrable with respect to the tangent bundle projection, then use this to define the notion of horizontal descent in a tangent category, which we then use as a key tool to study the aforementioned classes of morphisms. In particular, we use this to define a de Rham relative cotangent complex in an arbitrary tangent category. 
\end{abstract}


\maketitle
\tableofcontents
\section{Introduction}

Tangent categories provide a minimal semantic setting in which we can study differential geometric reasoning. The theory of tangent categories was developed first by Rosick{\'y} in the 1980's in \cite{Rosicky} and later rediscovered in 2014 by Cockett and Cruttwell in \cite{GeoffRobinDiffStruct}. Since their rediscovery, the theory of tangent categories has been greatly expanded upon and is now a well-established field with a rich literature with connections and applications in various areas such as (syntehtic) differential geometry, algebraic geometry, commutative algebra, operads, and theoretical computer science for instance. As tangent categories are fundamentally linked to differential geometry, it is natural to study the ways in which various differential geometric concepts can be formalized to a tangent categorical framework. After extending these concepts, however, it is also both necessary and important to study how the newly extended concepts apply to various examples in order to understand what these new concepts teach us about the fundamental nature of geometry. In particular, tangent categories can thus be used to illustrate the deep connection between algebraic geometry and differential geometry in precise ways. 

In this paper, we provide an in-depth and systematic study in a tangent categorical framework of various classes of morphisms that are very important to differential geometry. From \cite{BenVectorBundles} and \cite{GeoffMarcelloTSubmersionPaper}, we already know how to define tangent categorically what it means to be both a submersion and a local diffeomorphism, and that these recover precisely submersions and local diffeomorphisms in the tangent category of smooth manifolds. What wasn't known previously is what these classes of maps were in other important tangent categories such (affine) schemes or commutative algebras. Of course, the study of properties of maps in differential geometry do not start and end at local diffeomorphisms and submersions. For instance, we also have immersions and unramified maps as well, whose formalization in a tangent category we introduce in this paper. Our investigation of these classes of maps contain a formal study of the stability and permanence properties of each class, and the many of the ways in which they relate to each other. 

Along the way we also derive some surprising results. For instance, while it may seem strange to distinguish between immersions and unramified maps; they are the same in differential geometry, we will see that in \emph{full generality} this need not be the case. So in an arbitrary tangent category, with tangent bundle functor $T$, the notion of being an immersion is indeed \emph{distinct} from being $T$-unramified. That said, if a tangent category has negatives (which is what is called a Rosick{\'y} tangent category), then being $T$-unramified is the same as being a $T$-immersion. 

Another geometric aspect of the structure of tangent categories which can be used to explain and study the various classes of maps of interest lies in what is called the horizontal descent of the horizontal bundle in the differential geometry literature. In a tangent category, the horizontal descent of a map $f$ is the unique map $\theta_f$ between the tangent bundle of the domain of $f$ and the pullback of $f$ along the tangent bundle projection (assuming said pullback exists). We spend a great deal of time in this paper studying the horizontal descent $\theta_f$ in its own right and use it to examine the structure of morphisms in a tangent category. Indeed, one can use the horizontal descent to simplify and give clean methods for checking when maps are $T$-immersions, $T$-submersions, $T$-monic, and $T$-{\'e}tale. In particular, the place where we can really see where working with the horizontal descent directly provides a powerful perspective is when working with $T$-submersions. While the general definition of a $T$-submersion is stated in terms of weak pullbacks, we can use the horizontal descent to simply the definition significantly: instead it suffices to ask that the horizontal descent admits a section. This greatly helped classify the $T$-submersions in our main examples. Moreover, arguably the most important application of the horizontal descent is that it also allows us to introduce the formalization of the relative cotangent sequence from algebraic geometry, and hence give the de Rham sequence of maps in an arbitrary tangent category. 

We expect this paper to be of interest to the wider mathematical community as it provides a set of tools with which to begin applying the machinery and technology of tangent category theory to various settings. Moreover, this paper should also be of theoretical interest to geometers, in particular category-theoretic geometers. For instance, in upcoming work \cite{JSGeoffVZariski} we will be studying and proving the existence of a Zariski topology for (certain families) of tangent categories and use this to define and study exactly what it means to be a \emph{scheme} in a tangent-category. Said project relies on a deep understanding on the classes of the maps we study in this paper (as open immersions in algebraic geometry are, in particular, $T$-monic $T$-{\'e}tale maps and so any Zariski open generally must at least be a $T$-monic $T$-{\'e}tale map).

To summarize the story in this paper, we present the table in Figure \ref{Figure: Good table is win} illustrating the tangent-categorical properties a map can have and how they are incarnated in differential geometry, algebraic geometry, commutative algebras, and Cartesian differential categories. 
\begin{figure}[ht]
\begin{center}
\begin{tabular}{|c||c|c|c|c|}
\hline Tangent Category & Diff.\@ Geo.\@ & Alg.\@ Geo.\@ & Com.\@ Alg.\@ & CDC \\
\hline\hline  $T$-monic  & Embedding & Monic & Monic & $f, \theta_f$ \\ 
\hline $T$-immersion & Immersion & Formally Unramified & Monic & Higher Order Derivatives \\
& && & linear monic in linear arg. \\
\hline $T$-unramified & Immersion & Formally Unramified & Monic & $D[f]$ with \\
& & & & trivial kernel in linear arg. \\
\hline $T$-submersion & Submersion & Exact relative & Regular & $\langle \pi_0,D[f]\rangle$ \\
 & & cotangent sequence & epimorphism & regular epi + conditions\\
\hline Split & Submersion & Relatively &  & $D[f]$ \\
$T$-submersion  & with & Formally Smooth & Linear retract &  Retract \\
 & connection & relative to $R$ & &  \\
\hline  &  Local  & Formally unramified & & $D[f]$\\
 $T$-{\'e}tale &diffeomorphism & and relatively & Isomorphism &  Isomorphism \\
 & & formally smooth & & \\
\hline 
\end{tabular}  
\caption{The corresponding classes of morphism in each given (family) of tangent categories.}\label{Figure: Good table is win}
\end{center}
\end{figure}

\textbf{Notation and Conventions}: In this paper we assume that the reader is familiar with the basics of category theory, as well as some familiarity with at least the basic language of algebraic geometry at the level of ringed spaces and differential geometry on smooth manifolds. We denote arbitrary categories using script-fonts such as $\Cscr$ or $\Dscr$. Named categories will generally be written in bold fonts, such as $\Set$ or $\mathbf{CRing}$ for example. When pertinent, for a category $\Cscr$, we will also write $\Cscr_0$ for its class of objects and $\Cscr_1$ for its class morphisms. Objects will be denoted using capital Roman letter, such as $X, Y, Z, \cdots$, while morphisms will generally be written with lower case Roman letters, such as $f, g, h, \cdots$. Identity morphisms are denoted as $\id_X$. Unlike other papers in the tangent category literature (which use diagrammatic order for composition), we write composition in the usual applicative convention using $\circ$, i.e., via:
\[
\begin{tikzcd}
X \ar[r]{}{f} \ar[dr, swap]{}{g \circ f} & Y \ar[d]{}{g} \\
 & Z
\end{tikzcd}
\]
When composing natural transformations, we denote vertical composition using $\circ$ and horizontal composition using $\ast$, i.e., by writing $\beta \circ \alpha$ for the composition:
\[
\begin{tikzcd}
\Cscr \ar[rr, bend left = 50, ""{name = U}]{}{F} \ar[rr, ""{name = M}]{}[description]{G} \ar[rr, bend right = 50, swap, ""{name= D}]{}{H} & & \Dscr
\ar[from = U, to = M, Rightarrow, shorten <= 4pt, shorten >= 4pt]{}{\alpha}
\ar[from = M, to = D, Rightarrow, shorten <= 4pt, shorten >= 4pt]{}{\beta}
\end{tikzcd}
\]
and by writing $\beta \ast \alpha$ for the composition:
\[
\begin{tikzcd}
\Cscr \ar[rr, bend left = 30, ""{name = UL}]{}{F} \ar[rr, swap, bend right = 30, ""{name = DL}]{}{G} & & \Dscr \ar[rr, bend left = 30, ""{name = UR}]{}{H} \ar[rr, swap, bend right = 30, ""{name = DR}]{}{K} & & \Escr
\ar[from = UL, to = DL, Rightarrow, shorten <= 4pt, shorten >= 4pt]{}{\alpha}
\ar[from = UR, to = DR, Rightarrow, shorten <= 4pt, shorten >= 4pt]{}{\beta}
\end{tikzcd}
\] 

\section*{Acknowledgements}
The authors would first like to thank Kristine Bauer for her insights in some of the definitions which are in this project and also for helping frame some of the terminology and perspectives employed in ways that are understandable from more myriad mathematical viewpoints. We also would like to thank Robin Cockett for his insights, his deep knowledge (and especially with regards to pointing us towards the folkloric lemma we use in Section \ref{Section: Strong TImmersions}), and his interest in this project and the material in it. We also would like to thank both Geoff Cruttwell and Marcello Lanfranchi for their helpful discussions regarding the similarities and overlaps in our studies of submersions in their paper \cite{GeoffMarcelloTSubmersionPaper} and the work in this paper. While we ended up discussing material at different levels (they focus on display systems, which is a condition we do not impose in this paper), their insight and questions helped shape much of the material regarding submersions in this paper. We'd also like to thank Deni Salja for helpful conversations. Finally, the second author would like to thank the entire Macquarie University Category Theory group, especially Richard Garner and Ross Street, for their hospitality and intellectually stimulating environment during his visit in November 2024. The questions, insights, conversations, and help that everyone provided was invaluable.

\section{Recollections on Tangent Categories}

In this section we will recall some of the basics of tangent category theory, such as the definitions of tangent categories and differential bundles, as well as provide some examples of the main tangent categories that we will focus on throughout this paper. We also will prove some elementary categorical results which allow us to conclude that for any tangent category, its tangent bundle functor reflect both limits and colimits.

\subsection{Tangent Categories and Differential Bundles}
As we have discussed above, tangent categories provide an abstract categorical framework to study the core structural aspects of differential geometry. While our general exposition regarding tangent categories and their properties will be relatively terse, we refer the reader to \cite{GeoffRobinBundle,GeoffRobinDiffStruct} for a more in-depth introduction to tangent category theory. 

\begin{definition}[{\cite[Definition 3]{BenVectorBundles}}]\label{Defn: T limit and T colimit}
Let $\Cscr$ be a category and let $T:\Cscr \to \Cscr$ be an endofunctor. If $L$ is a (co)limit of a diagram $D:I \to \Cscr$, then we say that $L$ is a \emph{$T$-(co)limit} if for all $m \in \N$, $T^m(L)$ is a (co)limit of the diagram $T^m \circ D:I \to \Cscr$.
\end{definition}

\begin{definition}[{\cite[Definition 2.3]{GeoffRobinDiffStruct}}] A \emph{tangent structure} $\Tbb$ on $\Cscr$ is a tuple $\Tbb = (T, p, 0, \operatorname{add}, \ell, c)$ consisting of: 
\begin{enumerate}[{\em (i)}]
    \item An endofunctor $T:\Cscr \to \Cscr$, called the \emph{tangent bundle functor}, 
    \item A natural transformation $p:T \Rightarrow \id_{\Cscr}$, called the \emph{projection}, such that for all objects $X \in \Cscr_0$ and for all $n \in \N$, the wide pullback
 \[
 \begin{tikzcd}
    & & T_nX \ar[dl]{}[description]{\pi_1} \ar[dll, swap]{}{\pi_0} \ar[d] \ar[dr, swap]{}[description]{\pi_{n-2}} \ar[drr]{}{\pi_{n-1}}& & \\
TX \ar[drr, swap]{}{p_X} & TX \ar[dr, swap]{}[description]{p_X} & \cdots \ar[d] & TX \ar[dl]{}[description]{p_X} & TX \ar[dll]{}{p_X} \\
 & & X
 \end{tikzcd}
 \]
    exists and is a $T$-limit. Note for each $0 \leq i \leq n-1$, we write $\pi_i:T_{n}X \to TX$ for the $i$-th projection of the pullback.
    \item A natural transformation $0:\id_{\Cscr} \Rightarrow T$, called the \emph{zero};
    \item A natural transformation $\operatorname{add}:T_2 \Rightarrow T:\Cscr \to \Cscr$, called the \emph{sum};
    \item A natural transformation $\ell:T \Rightarrow T^2$, called the \emph{vertical lift}; 
    \item A natural transformation $c:T^2 \Rightarrow T^2$, called the \emph{canonical flip};
    \end{enumerate}
    and such that:
  \begin{enumerate}[{\em (i)}]  
\item $(p_X, 0_X, \operatorname{add}_X)$ is an \emph{additive bundle} {\cite[Definition 2.1]{GeoffRobinDiffStruct}}, so an internal commutative monoid in the slice category $\Cscr_{/X}$, that is, the following diagrams commute: 
 \begin{equation*}\label{Eqn: First CMon Bundle Diagrams}
    \begin{tikzcd}
    X \ar[r]{}{0_X} \ar[dr, equals] & TX \ar[d]{}{p_X} \\
     & X
    \end{tikzcd}\quad
    \begin{tikzcd}
    (TX \times_X TX) \times_X TX \ar[d, swap]{}{\cong} \ar[rr]{}{\operatorname{add}_X \times \id} & & T_2X \ar[d]{}{\operatorname{add}_X} \\
    TX \times_X (TX \times_X TX) \ar[dr, swap]{}{\id \times \operatorname{add}_X} & & TX \\
     & T_2X \ar[ur, swap]{}{\operatorname{add}_X}
    \end{tikzcd}
 \end{equation*}
 \begin{equation*}\label{Eqn: Second CMon Bundle Diagrams}
 \begin{tikzcd}
TX \ar[r]{}{\cong} \ar[drr, equals] \ar[d, swap]{}{\cong} &  X \times_X TX \ar[r]{}{0_X \times \id} & T_2X \ar[d]{}{\operatorname{add}_X} \\
TX \times_X X \ar[r, swap]{}{\id \times 0_X} & T_2X \ar[r, swap]{}{\operatorname{add}_X} & TX
 \end{tikzcd}\quad
 \begin{tikzcd}
TX \times_X TX \ar[rr]{}{\operatorname{switch}} \ar[dr, swap]{}{\operatorname{add}_X} & & TX \times_X TX \ar[dl]{}{\operatorname{add}_X} \\
& TX
\end{tikzcd}
 \end{equation*}
    \item  $(\ell_X,0_X)$ is an \emph{additive bundle morphism} {\cite[Definition 2.2]{GeoffRobinDiffStruct}}, that is, the following diagrams commute: 
    \[
    \begin{tikzcd}
    TX \ar[r]{}{\ell_X} \ar[d, swap]{}{p_X} & T^2X \ar[d]{}{(T \ast p)_X} \\
    X \ar[r, swap]{}{0_X} & TX
    \end{tikzcd}\quad
    \begin{tikzcd}
    T_2X \ar[rr]{}{\langle \ell_X \circ \pi_0, \ell_X \circ \pi_1\rangle} \ar[d, swap]{}{\operatorname{add}_X} & & T(T_2X) \ar[d]{}{(T \ast \operatorname{add})_X} \\
    TX \ar[rr, swap]{}{\ell_X} & & T^2X
    \end{tikzcd}\quad
    \begin{tikzcd}
    X \ar[r]{}{0_X} \ar[d, swap]{}{0_X} & TX \ar[d]{}{(T \ast 0)_X} \\
    TX \ar[r, swap]{}{\ell_X} & T^2X
    \end{tikzcd}
    \]
    \item $(c_X, \id_{TX})$ is an \emph{additive bundle morphism}, that is, the following diagrams commute:
    \[
    \begin{tikzcd}
    T^2X \ar[r]{}{c_X} \ar[d, swap]{}{(T \ast p)_X} & T^2X \ar[d]{}{(p \ast T)_X} \\
    TX \ar[r, equals] & TX
    \end{tikzcd}\quad
    \begin{tikzcd}
    (T^2)_2X \ar[rr]{}{\langle c_X \circ \pi_0, c_X \circ \pi_1\rangle} \ar[d, swap]{}{(T \ast \operatorname{add})_X} & & (T^2)_2X \ar[d]{}{(\operatorname{add} \ast T)_X} \\
    T^2X \ar[rr, equals] & & T^2X
    \end{tikzcd}\quad
    \begin{tikzcd}
    TX \ar[r, equals] \ar[d, swap]{}{(T \ast 0)_X} & TX \ar[d]{}{(0 \ast T)_X} \\
    T^2X \ar[r, swap]{}{c_X} & T^2X
    \end{tikzcd}
    \]
    \item The following diagrams commute: 
    \[
    \begin{tikzcd}
    T^2 \ar[r]{}{c} \ar[dr, equals] & T^2 \ar[d]{}{c} \\
     & T^2
     \end{tikzcd}\quad
    \begin{tikzcd}
    T \ar[r]{}{\ell} \ar[dr, swap]{}{\ell} & T^2 \ar[d]{}{c} \\
     & T^2
    \end{tikzcd}
    \]
    \[
    \begin{tikzcd}
    T^3 \ar[d, swap]{}{T \ast c} \ar[r]{}{c \ast T} & T^3 \ar[r]{}{T \ast c} & T^3 \ar[d]{}{c \ast T} \\
    T^3 \ar[r, swap]{}{c \ast T} & T^3 \ar[r, swap]{}{T \ast c} & T^3
    \end{tikzcd}\quad
    \begin{tikzcd}
    T \ar[r]{}{\ell} \ar[d, swap]{}{\ell} & T^2 \ar[d]{}{T \ast \ell} \\
    T^2 \ar[r, swap]{}{\ell \ast T} & T^3
    \end{tikzcd}\quad
    \begin{tikzcd}
    T^2 \ar[d, swap]{}{c} \ar[r]{}{\ell \ast T} & T^3 \ar[r]{}{T \ast c} & T^3 \ar[d]{}{c \ast T} \\
    T^2 \ar[rr, swap]{}{T \ast \ell} & & T^3
    \end{tikzcd}
    \]
    \item For all objects $X \in \Cscr_0$, the diagram
    \[
    \begin{tikzcd}
    T_2X \ar[rrrrrr]{}{(T \ast \operatorname{add})_X \circ \langle \ell_X \circ \pi_0, (0 \ast T)_X \circ \pi_1 \rangle} & & & & & & T^2 \ar[rrr, shift left =1]{}{(T \ast p)_X} \ar[rrr, swap, shift right = 1]{}{0_X \circ p_X \circ (p \ast T)_X} & & & TX
    \end{tikzcd}
    \]
    is an equalizer, called the \emph{universality of the vertical lift}.
\end{enumerate}
A \emph{tangent category} is a pair $(\Cscr,\Tbb)$ where $\Cscr$ is a category and $\Tbb$ is a tangent structure.
\end{definition}

When referring to a tangent category, as a shorthand, we will often abuse notation slightly and refer to $\Cscr$ as a tangent category while leaving the given tangent structure implicit. This will be relevant primarily when it comes to studying specific tangent categories with known ``canonical'' tangent structures, which we will review in the next subsection. 

Let us briefly give some interpretations on tangent structures and the axioms above, which essentially formalize key properties of the tangent bundle on smooth manifolds from classical differential geometry. So an object $X$ in a tangent category can be interpreted as a base space, and $TX$ as its abstract tangent bundle. The projection $p_X$ is the analogue of the natural projection from the tangent bundle to its base space, making $TX$ an abstract fibre bundle over $X$. The wide pullbacks $T_nX$ encode the space of $n$-tangent vectors of $X$ which all anchored at the same point in $X$. In this case $T^m$ preserving $T_nX$ says that if you differentiate $m$ times and had $n$ tangent vectors anchored at the same point, then you now have precisely the space of $n$ degree $m$-vectors tangent to $T^mX$ anchored at the same base point. The sum $\operatorname{add}_X$ allows us to add tangent vectors achored at the same point, while the zero $0_X$ gives the zero tangent vector. This makes $TX$ into a generalized version of a smooth vector bundle where each fibre is a commutative monoid. The vertical lift $\ell_X$ encodes a linear differentiation operation, as it says that if you take a tangent vector and differentiate it, you get a tangent vector which is in a degree $2$ infinitesimal neighbourhood of the base point. The universal property of vertical lift says that the tangent bundle embeds into the double tangent bundle via the vertical bundle. The canonical flip $c_X$ is an analogue of the smooth involution of the same name on the double tangent bundle, which essentially encodes the symmetry of mixed partials $\frac{\partial^2f}{\partial x\,\partial y} = \frac{\partial^2f}{\partial y\,\partial x}$ of second order forms on $X$; in other words, it encodes the symmetry of the Hessian matrix for (differentials) of smooth functions. 

There is a class of tangent categories we will occassionally focus on in particular in this paper which are named after Rosick{\'y} and classify the tangent categories originally considered in \cite{Rosicky}. The difference between these and the definition of tangent structure given above lies precisely in the difference between asking for our tangent bundles to be internal Abelian groups as opposed to commutative monoids. Intuitively, this means that in each fibre of $TX$ is an Abelian group, so we can take the negation of tangent vectors. 

\begin{definition}[{\cite[Section 3.3]{GeoffRobinDiffStruct}}]
A \emph{Rosick{\'y} tangent category} (previously called a tangent category with negatives) is a tangent category $\Cscr$ together with a natural transformation $\operatorname{neg}:T \Rightarrow T$, called the \emph{negation}, such that for all $X$, $(p_X, 0_X, \operatorname{add}_X, \operatorname{neg}_X)$ is an Abelian group bundle \cite[Section 1]{Rosicky}, which is an internal Abelian group in the slice category $\Cscr_{/X}$, that is, the following diagram commutes: 
\[
\begin{tikzcd}
TX \ar[rr]{}{\langle \operatorname{neg}_X, \id_{TX}\rangle} \ar[drr]{}[description]{0_X \circ p_X} \ar[d, swap]{}{\langle \id_{TX}, \operatorname{neg}_X\rangle} & & T_2X \ar[d]{}{\operatorname{add}_X} \\
T_2X \ar[rr, swap]{}{\operatorname{add}_X} & & TX
\end{tikzcd}
\]
\end{definition}

We now recall the definition of differential bundles, which are the generalization of the notion of a smooth vector bundle in differential geometry to the abstract framework of tangent categories \cite[Theorem 1]{BenVectorBundles}. Moreover, differential bundles also recapture the abstract differential geometric analogue of modules for a commutative ri(n)g \cite[Theorem 3.13]{GeoffJSDiffBunComAlg} and quasicoherent sheaves for schemes \cite[Theorem 4.28]{GeoffJSDiffBunComAlg}.

\begin{definition}[{\cite[Definition 2.3]{GeoffRobinBundle}}]
Let $\Cscr$ be a tangent category. Then \emph{differential bundle} over an object $X \in \Cscr_0$ is a tuple $q = (q, \zeta, \sigma, \lambda)$
\begin{enumerate}[{\em (i)}]
    \item An object $E$ and a map $q:E \to X$, called the \emph{projection}, such that for all $n \in \N$, the wide pullbacks
    \[
    \begin{tikzcd}
        & & E_n \ar[dl]{}[description]{\pi_1} \ar[dll, swap]{}{\pi_0} \ar[d] \ar[dr, swap]{}[description]{\pi_{n-2}} \ar[drr]{}{\pi_{n-1}}& & \\
E \ar[drr, swap]{}{q} & E \ar[dr, swap]{}[description]{q} & \cdots \ar[d] & E \ar[dl]{}[description]{q} & TX \ar[dll]{}{q} \\
 & & X
    \end{tikzcd}
    \]
    exist and are preserved by $T^m$ for all $m \in \N$. 
\item A map $\zeta:X \to E$, called the \emph{zero};
\item A map $\sigma:E_2 \to E$, called the \emph{sum}; 
\item A map $\lambda:E \to TE$, called the \emph{lift};
    \end{enumerate}
and such that: 
\begin{enumerate}[{\em (i)}]
    \item $(q,\zeta,\sigma)$ is an additive bundle, so an internal commutative monoid in $\Cscr_{/X}$, that is the following diagrams commute:
    \begin{equation*}\label{Eqn: First Diagrams for Bundle CMon}
    \begin{tikzcd}
    X \ar[r]{}{\zeta} \ar[dr, equals] & E \ar[d]{}{q} \\
     & X
    \end{tikzcd}\quad
    \begin{tikzcd}
    (E \times_X E) \times_X E \ar[rr]{}{\sigma \times \id} \ar[d, swap]{}{\cong} & & E_2 \ar[d]{}{\sigma} \\
    E \times_X (E \times_X E) \ar[dr, swap]{}{\id \times \sigma} & & E \\
     & E \times_X E \ar[ur, swap]{}{\sigma}
    \end{tikzcd}
    \end{equation*}
    \begin{equation*}\label{Eqn: Second Diagrams for Bundle CMon}
   \begin{tikzcd}
    E \ar[r]{}{\cong} \ar[drr, equals] \ar[d, swap]{}{\cong} & X \times_X E \ar[r]{}{\zeta \times \id} & E_2 \ar[d]{}{\sigma} \\
    E \times_X X \ar[r, swap]{}{\id \times \zeta} & E_2 \ar[r, swap]{}{\sigma} & E
    \end{tikzcd}\quad
    \begin{tikzcd}
    E \times_X E \ar[rr]{}{\operatorname{switch}} \ar[dr, swap]{}{\sigma} & & E \times_X E \ar[dl]{}{\sigma} \\
     & E
    \end{tikzcd}
    \end{equation*}
    \item $(\lambda,0_X)$ is an additive bundle morphism, that is, the following diagrams commute:  
    \[
    \begin{tikzcd}
    E \ar[r]{}{\lambda} \ar[d, swap]{}{q} & TE \ar[d]{}{Tq} \\
    X \ar[r, swap]{}{0_X} & TX
    \end{tikzcd}\quad
    \begin{tikzcd}
    X \ar[d, swap]{}{\zeta} \ar[r]{}{0_X} & TX \ar[d]{}{T\zeta} \\
    E \ar[r,swap]{}{\lambda} & TE
    \end{tikzcd}\quad
    \begin{tikzcd}
    E_2 \ar[rr]{}{\langle \lambda \circ \pi_0, \lambda \circ \pi_1\rangle} \ar[d, swap]{}{\sigma} & & T(E_2) \ar[d]{}{T\sigma} \\
    E \ar[rr, swap]{}{\lambda} & & E
    \end{tikzcd}
    \]
    \item $(\lambda,\zeta)$ is an additive bundle morphism, that is, the following diagrams commute: 
     \[
    \begin{tikzcd}
    E \ar[r]{}{\lambda} \ar[d, swap]{}{q} & TE \ar[d]{}{p_E} \\
    X \ar[r, swap]{}{\zeta} & E
    \end{tikzcd}\quad
    \begin{tikzcd}
    X \ar[d, swap]{}{\zeta} \ar[r]{}{\zeta} & E \ar[d]{}{0_E} \\
    E \ar[r,swap]{}{\lambda} & TE
    \end{tikzcd}\quad
    \begin{tikzcd}
    E_2 \ar[rr]{}{\langle \lambda \circ \pi_0, \lambda \circ \pi_1\rangle} \ar[d, swap]{}{\sigma} & & (TE)_2 \ar[d]{}{\operatorname{add}_E} \\
    E \ar[rr, swap]{}{\lambda} & & E
    \end{tikzcd}
    \]
    \item The following diagram commutes: 
    \[
    \begin{tikzcd}
    E \ar[r]{}{\lambda} \ar[d, swap]{}{\lambda} & TE \ar[d]{}{\ell_E} \\
    TE \ar[r, swap]{}{T\lambda} & T^2E
    \end{tikzcd}
    \]
    \item The following diagram: 
    \[
    \begin{tikzcd}
    E_2 \ar[d, swap]{}[description]{q \circ \pi_0 = q \circ \pi_1} \ar[rrrr]{}{T(\sigma) \circ \langle \lambda \circ \pi_0, 0_E \circ \pi_1 \rangle} & & & & TE \ar[d]{}{Tq} \\
    X \ar[rrrr, swap]{}{0_X} & & & & TX
    \end{tikzcd}
    \]
    is a $T$-pullback, called the \emph{universality of the lift}.
    \end{enumerate}
\end{definition}

As a shorthand, we will often denote differential bundles by their projection $q: E \to X$. In particular, for every $X$, its tangent bundle $p_X: TX \to X$ is a differential bundle over $X$ \cite[Example 2.4.(ii)]{GeoffRobinBundle}. See \cite{BenVectorBundles,ching2024characterization,RoryConnectionsTanCAts} for other equivalent ways to define differential bundles, and various key universal properties differential bundles satisfy. An immediate simple observation one can make about differential bundles is: 

\begin{lemma}
Let $q: E \to X$ be a differential bundle in a tangent category $\Cscr$. Then $q$ is monic if and only if $q$ is an isomorphism.
\end{lemma}
\begin{proof}
Because $q:E \to X$ is a retract, $q$ is monic if and only if it is an isomorphism.
\end{proof}

What this lemma above tells us is that while it \emph{is} the case that we can largely think of differential bundles over $X$ as modules over $X$, the idea should not be to think of $E$ as directly a module but instead more as carrying a locally linear structure over $X$ parametrized by a module. We will see below that for schemes, this is exactly what happens.

There are two kinds of morphisms between differential bundles: 

\begin{definition}[{\cite[Definition 2.3]{GeoffRobinBundle}}]
Let $q:E \to X$ and $r:F \to Y$ be differential bundles in a tangent category $\Cscr$. A \emph{morphism of differential bundles} $\binom{g}{f}: q \to \mathsf{r}$ is a pair of maps $g:E \to F$ and $f:X \to Y$ such that the following diagram commutes: 
\[
\begin{tikzcd}
E \ar[r]{}{g} \ar[d, swap]{}{q} & F \ar[d]{}{r} \\
X \ar[r, swap]{}{f} & Y
\end{tikzcd}
\]
commutes. We say that $\binom{g}{f}: q \to r$ is a \emph{linear morphism of differential bundles} if it is a morphism of differential bundles such that the following diagram also commutes: 
\[
\begin{tikzcd}
E \ar[r]{}{g} \ar[d, swap]{}{\lambda_E} & F \ar[d]{}{\lambda_F} \\
TE \ar[r, swap]{}{Tg} & TF
\end{tikzcd}
\]
For differential bundles $q: E \to X$ and $q^\prime: E^\prime \to X$ over the same object $X$, we say that $f: q \to q^\prime$ is a \emph{(linear) morphism of differential bundles over $X$} if $\binom{f}{\id_X}: q \to q^\prime$ is a (linear) morphism of differential bundles. 
\end{definition}

It is worth mentioning that linear morphism of differential bundles automatically preserve the sum and zero as well \cite[Proposition 2.16]{GeoffRobinBundle}, or in other words, a linear morphism of differential bundles is also an additive bundle morphism. 

If $\Cscr$ is a tangent category, then we write $\DBun(\Cscr)$ for the category of differential bundles and linear differential bundle morphisms in $\Cscr$, and for every object we will write $\DBun(X)$ for the category of differential bundles $q:E \to X$ over $X$ with linear differential bundle morphisms $f: q \to q^\prime$ between them. We write $\DBun_\omega(X)$ for the category of differential bundles over $X$ with not-necessarily-linear differential bundle morphisms between them. 

The categories $\DBun(X)$ will be of frequent interest to us in this paper and occupy an important tool for understanding the structure of tangent categories. While this statement is generally true, it is in particular true for both Theorems \ref{Thm: ZeroCarrable makes relative bundle a kernel in DBun} and \ref{Thm: Rosckiy Timmersion iff Tunramified} below. Of particular importance for us are some of the \emph{homological} properties of $\DBun(X)$. Many of the results we present below are immediate consequences of the definition of differential bundles (the $\mathbf{CMon}$-enrichment, for instance, falls into this grouping); those results which are \emph{not} immediate are due to Luschyn-Wright as declared.

\begin{proposition}\label{Prop: STructure of DBun}
Let $X$ be an object in a tangent category $\Cscr$. The category $\DBun(X)$ satisfies the following:
\begin{enumerate}[{\em (i)}]
    \item Finite products in $\DBun(X)$ are biproducts (cf.\@ \cite[Proposition 4.1]{RoryConnectionsTanCAts}).
    \item Every object in $\DBun(X)$ admits finite powers (cf.\@ \cite[Corollary 4.5]{RoryConnectionsTanCAts}). 
    \item The trivial bundle $\id_X: X \to X$ is a zero object in $\DBun(X)$. 
    \item The category $\DBun(X)$ is $\mathbf{CMon}$-enriched.
\end{enumerate}
\end{proposition}
\begin{proof}
As mentioned, items $(i)$ and $(ii)$ are due to \cite[Proposition 4.1 \& Corollary 4.5]{RoryConnectionsTanCAts} respectively. Item $(iii)$ follows immediately from $(i)$ because the trivial bundle $\id_X: X \to X$ is a terminal object in $\DBun(X)$. Item $(iv)$ follows immediately from the fact each differential bundle is an internal commutative monoid in $\Cscr_{/X}$ and that linear differential bundle morphisms are commutative monoid morphisms \cite[Proposition 2.16]{GeoffRobinBundle}. 
\end{proof}

\subsection{Examples of Tangent Categories and Their Differential Bundles}

In this section we present our main running examples of tangent categories and their differential bundles that we will consider throughout this paper. Lists of other examples of tangent categories can be found in \cite[Example 2.2]{GeoffRobinBundle} and \cite[Example 2]{garner:embedding-theorem-tangent-cats}. 

Our first example, while trivial, highlights the important fact that tangent structure on a category is not unique. 

\begin{example}
Any category $\Cscr$ is trivially a Rosick{\'y} tangent category with tangent structure given by the identity functor and identity maps. In this case, the only differential bundle over an object $X$ is $\id_X: X \to X$. 
\end{example}

Our next example, which is expected and arguably the canonical example of a tangent category, provides a direct link between tangent category theory and differential geometry.

\begin{example}\label{Example: The tangent category of smooth manifolds} The category of $\SMan$ of (not necessarily Hausdorff nor paracompact) smooth manifolds is a Rosick{\'y} tangent category with tangent structure $\Tbb$ induced by sending a smooth manifold $M$ to its tangent bundle $T(X)$. For any smooth manifold $X$, differential bundles over $X$ correspond precisely to smooth vector bundles over $X$ \cite[Theorem 1]{BenVectorBundles}. As such, we have an equivalences of categories between the categories $\DBun(X) \simeq \mathbf{VecBun}(X)_{\operatorname{linear}}$, where $\mathbf{VecBun}(X)_{\operatorname{linear}}$ is the category of smooth vector bundles over $X$ with linear vector bundle morphisms. If we relax the condition that our bundle maps need to be linear, we also obtain an equivalence $\DBun_\omega(X) \simeq \mathbf{VecBun}(X)$, where $\mathbf{VecBun}(X)$ is the category of smooth vector bundles over $X$ with not-necessarily-linear vector bundle morphisms.
\end{example}

This next example relates tangent category theory to commutative algebra. 

\begin{example}\label{Example: The tangent category of commutative algebras}
Let $R$ be a commutative rig. The category $\CAlg{R}$ of commutative $R$-algebras is a tangent category induced by sending a commutative $R$-algebra $A$ to its algebra of dual numbers:
\[T(A) := A[\epsilon] \cong A[x]/(x^2).\] 
A straightforward extension of the arguments of \cite[Theorem 3.13]{GeoffJSDiffBunComAlg} gives that for any commutative $R$-algebra $A$, differential bundles over $A$ correspond to modules over $A$ \cite[Remark 3.15]{GeoffJSDiffBunComAlg}. So we get an equivalence of categories $\DBun(A) \simeq \Mod{A}$, where $\Mod{A}$ is the category of $A$-modules and $A$-linear maps between them. For full details of this example, see \cite[Section 3]{GeoffJSDiffBunComAlg}. 
\end{example}

Our two next examples connect tangent category theory and algebraic geometry. 

\begin{example}\label{Example: The tangent category of affine R schemes}
Let $R$ be a commutative rig and let $\CAlg{R}^{\op}$ denote the opposite category of commutative $R$-algebras. Then $\CAlg{R}^{\op}$ is a tangent category where the tangent bundle functor is defined by sending a commutative $R$-algebra $A$ to the symmetric $A$-algebra of relative K{\"a}hler differentials\footnote{If $R$ is a rig and if $A$ is a commutative $R$ algebra, an $R$-derivation on $A$ with coefficients in an $A$-module $M$ is an $R$-linear map $\partial:A \to M$ such that for all $a, b \in A$, $\partial(a+b) = \partial(a)+\partial(b)$, $\partial(ab) = \partial(a)b + a\partial(b),$ and $\partial(1) = 0$. Only in the presence of negatives can we omit the equation $\partial(1) = 0$, as the identity $\partial(1) = \partial(1^2) = \partial(1)+\partial(1)$ simply says that $\partial(1)$ is additively idempotent otherwise.}:
\[T_{A/R} := \Sym_{A}\left(\Kah{A}{R}\right).\] 
By \cite[Theorem 4.17]{GeoffJSDiffBunComAlg} for the ring-theoretic case and \cite[Remark 4.20]{GeoffJSDiffBunComAlg} for the rig-theoretic case, we have that $\DBun(A)^{\op} \simeq \Mod{A}$. An important take-away from this equivalence is that for any map $f:E \to F$ of differential bundles over $A$, $f$ is linear if and only if $f$ is isomorphic to a map of the form $\Sym_{A}(g)$ for some module map $g:N \to M$. Now when $R$ is a commutative ring, famously, $\CAlg{R}^{\op}$ is of course equivalent to the category of affine $R$-schemes, and $\CAlg{R}^{\op}$ is a Rosick{\'y} tangent category. Moreover, for any affine $R$-scheme $\Spec A$, we obtain an opposite equivalence $\mathbf{QCoh}(\Spec A) \simeq \Mod{A} \simeq \DBun(A)^{\op}$, where $\mathbf{QCoh}(\Spec A)$ is the category of quasi-coherent modules over $\Spec A$. It is also interesting to mention that in \cite[Definition 16.5.12.I]{EGA44}, Grothendieck calls $T_{A/R}$ the ``fibré tangente'' (French for tangent bundle) of $A$, while in \cite[Section 2.6]{jubin:tangent-monad-and-foliations}, Jubin calls $T_{A/R}$ the tangent algebra of $A$. For full details of this example, see \cite[Section 4]{GeoffJSDiffBunComAlg}. 
\end{example}

We now provide a brief discussion of the tangent category $\Sch_{/S}$ of $S$-schemes for a base scheme $S$. A robust and extensively detailed description may be found in \cite{VooysSchemesDeepDive}. Before we discuss the tangent category $\Sch_{/S}$, let us do a bit of setup. Note that we assume some familiarity with quasi-coherent modules on locally ringed spaces. The original introduction may be found in both \cite{EGA01} and \cite{EGA1} with some material also in \cite{EGA2}, while more modern introductions may be found in various locations such as \cite{stacks-project}. 

Recall that for any scheme $X$, there is a relative symmetric algebra functor
\[
{\Sym}_{\Ocal_X}:\QCoh(X) \to \QCoh(X,\CAlg{\Ocal_X})
\]
for $\QCoh(X,\CAlg{\Ocal_X})$ the category of quasi-coherent sheaves of commutative algebras over $\Ocal_X$. There is also an opposite equivalence of categories described in \cite[Sections 1.2, 1.3, 1.7]{EGA2}
\[
\underline{\Spec}_X:\QCoh(X,\CAlg{\Ocal_X})^{\op} \to \mathbf{AffSch}_{/X}
\]
where $\underline{\Spec}_X$ is called the relative spectrum functor and $\mathbf{AffSch}_{/X}$ is the category of schemes $q:Y \to X$ with affine structure map. The composite functor
\[
\Vbb:= \underline{\Spec}_X \circ \Sym_{\Ocal_X}
\]
is the ``vector bundle with prescribed fibres'' functor of Grothendieck \cite[Section 1.7]{EGA2} which associates to a quasi-coherent sheaf $\Fscr$ a ``vector bundle'' $\Vbb(\Fscr)$ over $X$ whose fibres are precisely those of $\Fscr$ in the same way that vector bundles over smooth manifolds $X$ may be stated, up to opposite equivalence, with locally trivializable smooth manifolds over $X$ whose fibres are vector spaces.

\begin{example}\label{Example: The tangent category of S schemes}
For any base scheme $S$, the category $\Sch_{/S}$ is a Rosick{\'y} tangent category where for a scheme $X \to S$, the tangent scheme $T_{X/S}$ is the relative spectrum of the relative symmetric algebra of the sheaf of relative K{\"a}hler differentials, 
\[
T_{X/S} := \underline{\Spec}_X\left(\underline{\Sym}_{\Ocal_X}\left(\Kah{X}{S}\right)\right) = \mathbb{V}(\Kah{X}{S});
\]
see, for instance, \cite[Line 16.5,12,1]{EGA44} for the first introduction of this scheme. For any $S$-scheme $X$, there is an equivalence of categories $\DBun(X)^{\op} \simeq \QCoh(X)$ \cite[Theorem 4.28]{GeoffJSDiffBunComAlg}. As before, the important take-away is that a morphism of differential bundles $f:E \to F$ is linear if and only if $f$ is isomorphic to $\underline{\Spec}_X(\underline{\Sym}_{\Ocal_X}(g))$ for some map $g:\Fscr \to \Escr$ of quasi-coherent sheaves on $X$. When context is clear, we will simply write $TX$ for the tangent scheme of an $S$-scheme $\nu:X \to S$ in place of the more verbose $T_{X/S}$. For full details of this example, see \cite[Sections 3.4, Definition 5.2]{VooysSchemesDeepDive}. 
\end{example}

Another important family of tangent categories we will consider in this paper are Cartesian differential categories (which we will henceforth abberviate as CDCs), which were defined originally in \cite{RickRobinRobertCDC} to capture the total derivative from multivariable calculus over Euclidean spaces. While we do not present a full definition of a CDC (and instead defer the reader to \cite{RickRobinRobertCDC}), we focus on the following observations.

\begin{example}\label{Example: CDC} Briefly, a CDC is in particular a category $\Cscr$ with finite products $\times$ equipped with a \emph{differential combinator} $D$ which sends maps $f: X \to Y$ to maps $D[f]: X \times X \to Y$, called the \emph{derivative} of $f$, which satisfies seven axioms from \cite[Definition 2.1.1]{RickRobinRobertCDC} which capture the essential properties of the total derivative, such as the chain rule. Every CDC is a tangent category where $TX = X \times X$ and $T f = \langle f \circ \pi_0, D[f]\rangle$ \cite[Proposition 4.7]{GeoffRobinDiffStruct}. In a CDC $\Cscr$, every object is a differential bundle over the terminal object $0$, so $\Cscr \simeq \DBun_\omega(0)$ by \cite[Theorem 4.11]{GeoffRobinDiffStruct}. Furthermore, a map $f: X \to Y$ is a linear differential bundle morphism if and only if $Tf = f \times f$, or equivalently if $D[f] = f \circ \pi_1$, by \cite[Section 3.4]{GeoffRobinBundle}. Some important examples of CDCs include the Lawvere theory of real smooth functions, the Lawvere theory of polynomials with coefficients in a commutative rig $R$, any category with finite biproducts, and the subcategory of \emph{differential objects} (i.e. differential bundles over the terminal object, if one exists) of a tangent category (cf.\@ \cite[Theorem 4.11]{GeoffRobinDiffStruct}).
\end{example}

\subsection{Tangent Bundle Functors and (Co)Limit Reflection}

We conclude this section by proving a relatively elementary, yet powerful, result which holds for arbitrary tangent categories. Our goal is to establish that every tangent bundle functor of a tangent category must necessarily be limit and colimit reflecting. This is a useful step in both knowing how tangent bundle functors interact with isomorphisms (namely that if $TX \cong TY$ then it was already the case that $X \cong Y$) and in determining which endofunctors have a hope of being a tangent bundle functor. We ultimately will use this to conclude that in reasonable situations, all powers of tangent bundle functors reflect the tangent-categorical analogue of local diffeomorphisms (Definition \ref{Defn: TEtale in a tangent category} below).

The technical result we present below is elementary and likely in the category theory literature. However, we could not find it upon a literature search and so present it here with a full proof for the sake of completeness. While this is technically straightforward, it illustrates the technical power that comes from having a category $\Cscr$ with an endofunctor $F:\Cscr \to \Cscr$ and a section/retraction pair in the endofunctor category. 

\begin{theorem}\label{Thm: Endofunctor which has a retract onto th eidentity in endofunctors is limit reflecting}
Let $\Cscr$ be a category with an endofunctor $F:\Cscr \to \Cscr$ and natural transformations
\[
\begin{tikzcd}
\Cscr \ar[rr, bend left = 30, ""{name = U}]{}{F} \ar[rr, bend right = 30, swap, ""{name = D}]{}{\id_{\Cscr}} & & \Cscr
\ar[from = U, to = D, Rightarrow, shorten <= 4pt, shorten >= 4pt]{}{\alpha}
\end{tikzcd}\qquad \begin{tikzcd}
\Cscr \ar[rr, bend left = 30, ""{name = U}]{}{\id_{\Cscr}} \ar[rr, bend right = 30, swap, ""{name = D}]{}{F} & & \Cscr
\ar[from = U, to = D, Rightarrow, shorten <= 4pt, shorten >= 4pt]{}{\beta}
\end{tikzcd}
\]
for which $\alpha \circ \beta = \id_{\Cscr}$. Then $F$ reflects limits and colimits.
\end{theorem}
\begin{proof}
By duality, it suffices to establish the theorem for limits. We must prove that for any diagram $D:I \to \Cscr$ and any cone $X$ over $D$ for which $FX \cong \lim(F \circ D)$ is a limit in $\Cscr$, then $X \cong \lim(D)$. To this end assume we have a diagram $D:I \to \Cscr$, a cone
\[
\begin{tikzcd}
 & X \ar[dr]{}{d_i} \ar[dl, swap]{}{d_j} \\
Di \ar[rr, swap]{}{Df} & & Dj
\end{tikzcd}
\]
over $D$ in $\Cscr$ (with structure maps $d_i$ given as above rendering all diagrams commutative for $f \in D_1$), and a limit $FX \cong \lim(F \circ D)$ in $\Cscr$. Assume that we have a cone
\[
\begin{tikzcd}
 & Z \ar[dr]{}{\varphi_j}  \ar[dl, swap]{}{\varphi_j} \\
Di \ar[rr, swap]{}{Df} & & Dj
\end{tikzcd}
\]
in $\Cscr$ and note that applying $F$ to $Z$ gives a cone over $F \circ D$. As such, the cone $FZ$ factors uniquely through $FX$ as in the diagram:
\[
\begin{tikzcd}
 & FZ \ar[ddr, bend left = 30]{}{F\varphi_j}  \ar[ddl, swap, bend right = 30]{}{F\varphi_j} \ar[d, dashed]{}{\exists!\theta} \\
 & FX \ar[dr]{}{Fd_j} \ar[dl, swap]{}{Fd_i} \\
F(Di) \ar[rr, swap]{}{F(Df)} & & F(Dj)
\end{tikzcd}
\]
Now define the map $\rho:Z \to X$ via the composition:
\[
\begin{tikzcd}
Z \ar[d, swap]{}{\rho} \ar[r]{}{\beta_Z} & FZ \ar[d]{}{\theta} \\
X & FX \ar[l]{}{\alpha_X}
\end{tikzcd}
\]
An elementary computation shows that the map $\rho$ forces the cone $Z$ to factor through $X$ and so that $X$ is a weak limit of the diagram $D$. 

To finally conclude that $X$ is a limit of $D$, assume that we have two cone maps making the diagram
\[
\begin{tikzcd}
 & Z \ar[ddr, bend left = 30]{}{\varphi_j}  \ar[ddl, swap, bend right = 30]{}{\varphi_j} \ar[d, shift left = 1]{}{\rho} \ar[d, shift right = 1, swap]{}{\sigma} \\
 & X \ar[dr]{}{d_j} \ar[dl, swap]{}{d_i} \\
Di \ar[rr, swap]{}{Df} & & Dj
\end{tikzcd}
\]
commute for all $i \in I$. But then we compute that on one hand for all $i \in I$,
\[
Fd_i \circ \beta_X \circ \rho = Fd_i \circ F\rho \circ \beta_Z = F\varphi_i \circ \beta_Z
\]
and similarly $Fd_i \circ \beta_X \circ \sigma = F\varphi_i \circ \beta_Z.$ But this implies that both $\beta_X \circ \sigma$ and $\beta_X \circ \rho$ are cone maps which factor as
\[
\begin{tikzcd}
 & Z \ar[ddr, bend left = 30]{}{F\varphi_j\circ \beta_Z}  \ar[ddl, swap, bend right = 30]{}{\varphi_j \circ \beta_Z} \ar[d, shift left = 1]{}{\rho} \ar[d, shift right = 1, swap]{}{\sigma} \\
 & FX \ar[dr]{}{Fd_j} \ar[dl, swap]{}{Fd_i} \\
F(Di) \ar[rr, swap]{}{F(Df)} & & F(Dj)
\end{tikzcd}
\]
and so must coincide by virtue of $FX$ being a limit of $F \circ D$. However, this gives that $\beta_X \circ \sigma = \beta_X \circ \rho$, so since $\beta_X$ is a section and hence a split monic, it must also be the case that $\rho = \sigma$. Thus $X$ is a limit to $D$ in $\Cscr$ as desired. 
\end{proof}

\begin{corollary}\label{Cor: tangent bundle functor reflects limits}
For any tangent category $\Cscr$ the tangent bundle functor $T$ and all its powers $T^m$ for $m \in \N$ are limit and colimit reflecting.
\end{corollary}
\begin{proof}
Simply apply Theorem \ref{Thm: Endofunctor which has a retract onto th eidentity in endofunctors is limit reflecting} to the bundle projection, bundle zero section/retraction pair $p:T \Rightarrow \id_{\Cscr}$ and $0:\id_{\Cscr} \Rightarrow T$ and then proceed via induction on $m$.
\end{proof}

We now turn our attention to reflecting isomorphisms. 

\begin{lemma}\label{Lemma: Limits of singleton functors}
Let $\Cscr$ be a category with an object $X \in \Cscr_0$ and let $\underline{X}:\mathbbm{1} \to \Cscr$ denote the functor on the terminal category which picks out the object $X$. Then $Y$ is a limit or a colimit to $\underline{X}$ if and only if $Y \cong X$.
\end{lemma}
\begin{proof}
The colimit argument is dual to the limit one and hence omitted. Observe that a cone $Z$ over $\underline{X}$ is precisely given by a morphism $f:Z \to \underline{X}(\ast) = X$ in $\Cscr$. But then on one hand if $L \cong X$ via some map $\theta:L \to X$, the diagram
\[
\begin{tikzcd}
Z \ar[r, equals] \ar[d, swap]{}{\theta^{-1} \circ f} & Z \ar[d]{}{f} \\
L \ar[r, swap]{}{\theta} & X
\end{tikzcd}
\]
is a pullback so $L$ is a limit of $\underline{X}$. On the other hand, if there is a limit $\ell:L \to \underline{X}(\ast) = X$ then the diagram
\[
\begin{tikzcd}
X \ar[dr, equals] \ar[d, swap]{}{\exists! s} \\
L \ar[r, swap]{}{\ell} & X
\end{tikzcd}
\]
admits a unique factorization through $L$. Similarly, the diagram
\[
\begin{tikzcd}
L \ar[dr]{}{\ell} \ar[d, shift right = 1, swap]{}{s \circ \ell} \ar[d, shift left = 1]{}{\id} \\
L \ar[r, swap]{}{\ell} & X
\end{tikzcd}
\]
implies that $s \circ \ell = \id$ as well and so that $\ell$ is an isomorphism.
\end{proof}
The next corollary is immediate from the lemma (as isomorphisms are limits and colimits).

\begin{corollary}\label{Cor: Limt refl is conservative}
If $F:\Cscr \to \Dscr$ is a (co)limit reflecting functor then $F$ is isomorphism reflecting.
\end{corollary}

\begin{corollary}\label{Cor: tangent bundle functors are isomorphism reflecting}
For a tangent category $\Cscr$ and any $m \in \N$, the functors $T^m$ are isomorphism reflecting.
\end{corollary}
\begin{proof}
Apply Corollary \ref{Cor: Limt refl is conservative} to Corollary \ref{Cor: tangent bundle functor reflects limits}.
\end{proof}

\section{Horizontal Descent}\label{Section: Horizontal Descent}
In differential geometry the horizontal descent of a smooth morphism $f:X \to Y$ of smooth manifolds is defined to be the map 
\[
\theta_f:\coprod_{x \in X} T_xX \to \coprod_{x \in X} T_{f(x)}Y
\]
given by $\theta_f(x,v) := (x,D[f](x)\cdot v)$, where $D[f](x)$ is the derivative of $f$ at the point $x$ applied to the tangent vector $v$. The horizontal descent of $f$ provides a system of coordinates between the tangent bundles $TX$ and $TY$ which allows us to study the properties of $f$ local to the coordinates in $X$ and transfer these local coordinates to a subbundle of $TY$. 

The importance of this map in a general tangent category, when it exists, is difficult to overstate. Not only is $\theta_f$ the pairing map of the projection $p_X:TX \to X$ and the tangent map $Tf:TX \to TY$, it also allows us to record and reflect various structural properties of $f$, as we will see throughout the paper below. In this section we will study the maps $f$ in tangent categories that have a horizontal descent and begin to get to know how the horizontal descent $\theta_f$ can be used to study said maps.

\subsection{Carrable Morphisms}

An important structural and technical question to answer in tangent categories is when given morphisms $f:X \to Y$, for which differential bundles $q:E \to Y$ is it true that the pullback of $f$ along the projection $q$, 
\[
\begin{tikzcd}
f^{\ast}(E) \ar[r]{}{\pi_1} \ar[d, swap]{}{\pi_1} & E \ar[d]{}{q} \\
X \ar[r, swap]{}{f} & Y
\end{tikzcd}
\]
exists in $\Cscr$? If it does exist, is it preserved by all powers of the tangent bundle functor? In complete generality, the question has a negative answer. However, in many examples of interest (such as in our main examples, i.e., in the categories $\SMan$, $\CAlg{R}$, $\CAlg{R}^{\op}$, $\Sch_{/S}$, and in arbitrary CDCs) if $q:E \to Y$ is the tangent bundle $p_Y:TY \to Y$, then it \emph{is} the case that the given pullbacks above exist and are preserved by all powers of the tangent bundle functor. 

In the literature it is frequent practice to study and work with the condition that the differential bundle projection $q:E \to Y$ (or more generally any map $g:Z \to Y$ in $\Cscr$) admits arbitrary pullbacks which are preserved by all powers of the tangent bundle functor. These maps are called \emph{tangent display morphisms}, see \cite[Section 4.2]{GeoffRobinBundle} for the original introduction of display morphisms into tangent category theory, \cite[Section 5.9, Paragraph 3]{JonathanThesis} for a more explicit incarnation of transverse display systems, and both \cite[Definition 6]{BenVectorBundles} and \cite[Definition 2.1]{GeoffMarcelloTSubmersionPaper} for the more modern framework of tangent display morphisms\footnote{A historical note for the interested: the core concept of being a display class of morphisms in a category is taken from \cite[Section 8.3]{TaylorPracticalFoudnations} in which Taylor introduced display classes of morphisms for use with regards to fibrations. These maps were set in the tangent-categorical framework by Cockett and Cruttwell in \cite{GeoffRobinBundle} by way of what are there called transverse display systems. These were later rephrased as transverse displayed tangent categories in \cite{JonathanThesis} and then finally to tangent display systems in \cite{BenVectorBundles}. A recent in-depth study of tangent display systems may be found in \cite{GeoffMarcelloTSubmersionPaper}.}. In this paper we are content to work with a weaker property that is sufficient for our  purposes.

\begin{definition}[{\cite[Lemma 2.7]{GeoffRobinDiffStruct}; cf.\@ also \cite[Definition 15]{GeoffJSReverse}}]\label{Defn: DBun pullbackable}
Let $\Cscr$ be a tangent category and let $f:X \to Y$ be a morphism in $\Cscr$. If $q:E \to Y$ is a differential bundle, we say that $f$ is \emph{$q$-carrable} if for every $n \in \N$ the pullback of $n$ copies of $q$ along $f$ exists and is a $T$-pullback. Moreover, if $f$ is $q$-carrable for all differential bundles $q$, we say that $f$ is \emph{$DB$-carrable} and that $\Cscr$ is \emph{$DB$-carrable} if every morphism is $DB$-carrable.
\end{definition}

The references appearing in the definition above explain our declaration that $f:X \to Y$ is $q$-carrable. First, the term ``carrable'' is French for ``squarable'' and was used by the Grothendieck school in the SGA to denote that specific maps admitted pullback squares. Second, by \cite[Lemma 2.7]{GeoffRobinDiffStruct}, the pullback
\[
\begin{tikzcd}
f^{\ast}E \ar[r]{}{\pr_1} \ar[d, swap]{}{\pr_0} & E \ar[d]{}{q} \\
X \ar[r, swap]{}{f} & Y
\end{tikzcd}
\]
makes the object $f^{\ast}E$ into a differential bundle over $X$ via the first projection $\pr_0$. The differential bundle on $f^{\ast}(E)$ is given as follows: the addition is $f^{\ast}(\sigma) = \id_X \times \sigma$, the zero is $f^{\ast}(\zeta) = \id_X \times \zeta$, and the lift is\footnote{Subject to the usual issue of potentially pre-or-post-composing with isomorphisms to match up with chosen pullbacks.} $f^{\ast}(\lambda) = \id_X \times \lambda$. Thus a map $f$ being $q$-carrable means not only does the differential bundle projection $q$ admits a pullback along $f$, but also that the full differential bundle $q$ pulls back (is squarable, in other words) along $f$. 

\begin{proposition}\label{Prop: dbun projection is always qcarrable}
Let $\Cscr$ be a tangent category and let $q$ be a differential bundle with bundle projection $q:E \to X$. Then $q$ is $q$-carrable.
\end{proposition}
\begin{proof} The $T$-pullback
\[
\begin{tikzcd}
 E_2 \ar[r]{}{\pi_1} \ar[d, swap]{}{\pi_0} & E \ar[d]{}{q} \\
 E \ar[r, swap]{}{q} & X
\end{tikzcd}
\]
defining $E_2$ precisely shows that $q$ is $q$-carrable.
\end{proof}

\begin{example}\label{Example: DB carrable}
For any commutative rig $R$ and any scheme $S$, the tangent categories $\CAlg{R}, \CAlg{R}^{\op}$, and $\Sch_{/S}$ are all $DB$-carrable. 
\end{example}

\begin{example}\label{Example: DB carrable not fincomp} 
It is tempting to think that having every morphism be $DB$-carrable implies that the tangent category $\Cscr$ is finitely complete. However, this is not the case. Indeed let us now explain how $\SMan$ is $DB$-carrable, i.e., for any vector bundle $q:E \to M$, arbitrary pullbacks along $q$ exist and are preserved by $T$. To go about this we begin with a quick observation: submersions in $\SMan$ are tangent display morphisms in the sense that arbitrary pullbacks against submersions exist and are preserved by all powers of the tangent bundle functor. This follows from \cite[Lemma 2.19, Theorem 35.13.2]{KMSDiffGeo} and the fact that submersions $s:X \to Y$ meet arbitrary maps $f:Z \to Y$ transversally. As such, it suffices to argue that for any differential bundle $q:E \to M$, the projection $q$ is a submersion. However, this is well-known because vector bundles $q:E \to M$ admit the zero section $\zeta:M \to E$, so the section/retraction pair $T\zeta:TM \to TE$ and $Tq:TE \to TM$ gives the local surjectivity of $Tq$. Thus every vector bundle projection is a submersion and so arbitrary pullbacks against $q$ exist and are preserved by all powers of $T$. In particular, this shows that $\SMan$ is $DB$-carrable. However, it is well-known that $\SMan$ does not have all pullbacks: for example, the point map $0:\lbrace \ast \rbrace \to \R, \ast \mapsto 0$ does not have a pullback against the map $f:\R^2 \to \R$ given by $f(x,y) := xy$. 
\end{example}

The discussion of when morphisms are $q$-carrable has appeared at least implicitly already in a different guise in the tangent category literature surrounding connections \cite{RoryConnectionsTanCAts}, \cite{GeoffJSElias}. In \cite[Def 3.1]{RoryConnectionsTanCAts}, Lucyshyn-Wright defined the \emph{Basic Condition} for a differential bundle $q:E \to X$ to ask that the $T$-pullbacks
\[
\begin{tikzcd}
E_m \times_X T_n(X) \ar[r] \ar[d] & T_n(X) \ar[d] \\
E_m \ar[r]{}{} & X
\end{tikzcd}
\]
exist for all $m, n \in \N$. Writing $q_m:E_m \to X$ and $p_n:T_nX \to X$ for the respective total projections, then this is equivalent to asking that each $p_n$ is $q_m$-carrable (equivalently that each $q_m$ is $p_n$-carrable) for all $m, n \in \N$. While we will not use this explicitly in this paper, we record the result for those interested in working with connections.

\begin{proposition}\label{Prop: Basic Condition and Carrability}
A differential bundle $q:E \to X$ in a tangent category $\Cscr$ satisfies the Basic Condition if and only if for all $m, n \in \N$ the map $p_n$ is $q_m$-carrable. Equivalently, $q$ satisfies the Basic Condition if and only if the map $q_m$ is $p_n$-carrable for all $m, n \in \N$.
\end{proposition}
\begin{example}
In the case that $m = n = 1$ when regarding the Basic Condition, the pullback which appears is precisely:
\[
\begin{tikzcd}
E \times_X TX \ar[r]{}{\pi_1} \ar[d, swap]{}{\pi_0} & TX \ar[d]{}{p_X} \\
E \ar[r, swap]{}{q} & X
\end{tikzcd}
\]
This is exactly the pullback which witnesses $p$-or-$q$-carrability.
\end{example}

It is worth observing that a tangent category $\Cscr$ is $DB$-carrable if and only if there is a differential bundle pseudofunctor
\[
\DBun(-):\Cscr^{\op} \to \mathfrak{Cat}
\]
where $\mathfrak{Cat}$ is the 2-category of categories, such that for all morphisms $f:X \to Y$ the transition functor $\DBun(f):\DBun(Y) \to \DBun(X)$ is given by the categorical pullback against $f$ in $\Cscr$. 

In general however, while not every morphism in a tangent category is $DB$-carrable, the $DB$-carrable morphisms form a subcategory. 

\begin{lemma}\label{Lemma: The Subcat of DPull morphisms}
For a tangent category $\Cscr$, the $DB$-carrable morphisms form a subcategory of $\Cscr$.
\end{lemma}
\begin{proof} We need explain why identity morphisms are $DB$-carrable, and why $DB$-carrable morphisms are closed under composition. That isomorphisms are $DB$-carrable is trivial to see, as if $f:X \to Y$ is an isomorphism and $q:E \to Y$ is a differential bundle then
\[
\begin{tikzcd}
E \ar[r, equals] \ar[d, swap]{}{f^{-1} \circ q} & E \ar[d]{}{q} \\
X \ar[r, swap]{}{f} & Y
\end{tikzcd}
\]
is a pullback square. So in particular identity morphisms are $DB$-carrable. To see that $DB$-carrable morphisms compose, assume that $f:X \to Y$ and $g:Y \to Z$ are $DB$-carrable and let $q:E \to Z$ be a differential bundle. Now consider the diagram
\[
\begin{tikzcd}
f^{\ast}(g^{\ast}E) \ar[r]{}{\pr_1} \ar[d, swap]{}{\pr_0} & g^{\ast}E \ar[r]{}{\pr_1} \ar[d]{}{\pr_0} & E \ar[d]{}{q} \\
X \ar[r, swap]{}{f} & Y \ar[r, swap]{}{g} & Z
\end{tikzcd}
\]
in which each square is a pullback and each vertical morphism is a differential bundle projection by virtue of $f$ and $g$ being $DB$-carrable. But then by the well-known Pullback Lemma, the diagram
\[
\begin{tikzcd}
f^{\ast}(g^{\ast}E) \ar[r]{}{\pr_1 \circ \pr_1} \ar[d, swap]{}{\pr_0} & E \ar[d]{}{q} \\
X \ar[r, swap]{}{g \circ f} & Z
\end{tikzcd}
\]
is a pullback square, so $g \circ f$ is $DB$-carrable.
\end{proof}

Based on the observations above, it becomes evident that studying the category of $DB$-carrable morphisms is essentially all about studying to what degree we can get a differential bundle pseudofunctor $\DBun(-):\Cscr^{\op} \to \mathfrak{Cat}$. While this is an interesting and important technique to provide, it is fundamentally a ``global-to-the-category $\Cscr$'' kind of approach to studying the tangent structure $\Cscr$. As is well-known in geometry, while sometimes the full category $\Cscr$ is of interest, for some problems that exist one becomes primarily interested in a single morphism $f:X \to Y$ which happens to exist in $\Cscr$. In this case, the category of $DB$-carrable morphisms, while still helpful, becomes slightly less useful because you may only care about the differential bundles $q$ over $Y$ for which $f$ is $q$-carrable; this collection of differential bundles is likely of interest to you when studying a specific morphism \emph{precisely} because it allows you to better study the properties the morphism $f$ has directly.

\begin{definition}\label{Defn: Fcar Dbuns}
    Let $\Cscr$ be a tangent category and let $f:X \to Y$ be some fixed morphism in $\Cscr$. We define the category $\DBun(Y)_{f\operatorname{-}\text{car}}$ of \emph{$f$-carrable differential bundles over $Y$} as the full subcategory of $\DBun(Y)$ whose objects are differential bundles $q:E \to Y$ over $Y$ for which $f:X \to Y$ is $q$-carrable.
\end{definition}

\begin{lemma} For a morphism $f:X \to Y$ in a tangent category $\Cscr$, $\DBun(Y)_{f\operatorname{-}\text{car}}$ is the largest full subcategory of $\DBun(Y)$ for which there is a functor $f^{\ast}:\DBun(Y)_{f\operatorname{-}\text{car}} \to \DBun(X)$ given by taking pullbacks.
\end{lemma}
\begin{proof}
Note that as the category $\DBun(Y)_{f\operatorname{-}\text{car}}$ is defined by requiring that any object $q:E \to Y$ makes $f:X \to Y$ a $q$-carrable morphism, the pullback differential bundle $\pr_0:f^{\ast}E \to X$ exists. Consequently, the functor $f^{\ast}:\DBun(Y)_{f\operatorname{-}\text{car}} \to \DBun(X)$ is defined. The fact that this is the largest such subcategory of $\DBun(Y)$ for which $f^{\ast}$ may be defined is immediate from the facts that $\DBun(Y)_{f\operatorname{-}\text{car}}$ is a full subcategory of $\DBun(Y)$ and that any differential bundle $q$ over $Y$ for which $f$ is $q$-carrable is necessarily an object of $\DBun(Y)_{f\operatorname{-}\text{car}}$.
\end{proof}

As such, we have an inclusion functor $\operatorname{incl}:\DBun(Y)_{f\operatorname{-}\text{car}} \to \DBun(Y)$ as well as a pullback functor $f^{\ast}:\DBun(Y)_{f\operatorname{-}\text{car}} \to \DBun(X)$.

\subsection{Horizontal Descent}\label{Subsection: Horizontal Descent}

Recall that for every object $X$ in a tangent category $\Cscr$, its tangent bundle $p_X: TX \to X$ is a differential bundle over $X$. As a shorthand, if a morphism $f: X \to Y$ is $p_X$-carrable, we will write that $f$ is a $p$-carrable morphism. In particular, the projection of the tangent bundle is $p$-carrable.

\begin{proposition}\label{Prop: Every bundle projection is p-carrable}
Let $X$ be an object in a tangent category $\Cscr$. Then $p_X:TX \to X$ is $p$-carrable.
\end{proposition}
\begin{proof}
Specializing Proposition \ref{Prop: dbun projection is always qcarrable} to the differential bundle $p_X:TX \to X$ shows that $p_X$ is always $p$-carrable with pullback $p_X^{\ast}(TX) \cong T_2X$.
\end{proof}

Now if $f:X \to Y$ is a $p$-carrable morphism, the naturality square
\[
\begin{tikzcd}
TX \ar[r]{}{Tf} \ar[d, swap]{}{p_X} & TY \ar[d]{}{p_Y} \\
X \ar[r, swap]{}{f} & Y
\end{tikzcd}
\]
admits a factorization:
\[
\begin{tikzcd}
TX \ar[drr, bend left = 20]{}{Tf} \ar[ddr, swap, bend right = 20]{}{p_X} \ar[dr, dashed]{}[description]{\exists!\theta_f} & & \\
 & f^{\ast}(TY) \ar[r]{}{\pr_1} \ar[d, swap]{}{\pr_0} & TY \ar[d]{}{p_Y} \\
 & X \ar[r, swap]{}{f} & Y
\end{tikzcd}
\]
In differential geometry literature, $f^{\ast}(TY)$ is often called the \emph{horizontal bundle of $f$} and $\theta_f$ is often called the \emph{horizontal descent of $f$}. A large portion of this paper can be rephrased as studying the tangent-categorical structure of various morphisms $f$ in terms of said horizontal descent of $f$, so we formally define said terms now.

\begin{definition}\label{Defn: Thetaf}
Let $\Cscr$ be a tangent category and let $f:X \to Y$ be a $p$-carrable morphism. We call the pullback $f^{\ast}(TY) \cong X \times_Y TY$ the \emph{horizontal bundle of $f$}. If $\theta_f:TX \to f^{\ast}(TY)$ is the unique morphism rendering the diagram
\[
\begin{tikzcd}
TX \ar[drr, bend left = 20]{}{Tf} \ar[ddr, swap, bend right = 20]{}{p_X} \ar[dr, dashed]{}[description]{\exists!\theta_f} & & \\
 & f^{\ast}(TY) \ar[r]{}{\pr_1} \ar[d, swap]{}{\pr_0} & TY \ar[d]{}{p_Y} \\
 & X \ar[r, swap]{}{f} & Y
\end{tikzcd}
\]
commutative, then we call $\theta_f = \langle p_X, Tf\rangle$ the \emph{horizontal descent of $f$}.
\end{definition}

\begin{example}\label{Example: What  is thetaf in SMan} In $\SMan$, for a smooth morphism $f:X \to Y$, its horizontal bundle can be understood in local coordinates as pairs $(x,\overrightarrow{w})$ of a point $x \in X$ and a tangent vector $\overrightarrow{w}$ of $Y$ at the point $f(x)$. Then its horizontal descent $\theta_f: TX \to f^{\ast}(TY)$, acts on a tangent vector $(x,\overrightarrow{v}) \in TX$ by $\theta_f\left(x,\overrightarrow{v}\right) = \left(x,D[f](x)\cdot \overrightarrow{v}\right)$. 
\end{example}

\begin{example}\label{Example: What is thetaf in CAlg}
Let $R$ be a commutative rig. In the tangent category $\CAlg{R}$, the horizontal bundle of a map $f:A \to B$ is the rig which is the ``infinitesimal extension of $A$ by $B$.'' That is, $A \times_B B[\epsilon]$ is isomorphic to the rig: 
\[A \times_B B[\epsilon] \cong \lbrace (a,b) \; | \; a \in A, b \in B \rbrace =: A \ltimes B,\] 
where $A \ltimes B$ denotes the semidirect product. In particular, the multiplication of $A \ltimes B$ is given $(a,b)(x,y) = (ax, f(a)y + bf(x))$, so $A$ acts on $B$ with the dual numbers multiplication on the $B$-component. Then its horizontal descent $\theta_f: A[\epsilon] \to A \ltimes B$ is defined as $\theta_f\left( a + a^{\prime}\epsilon \right)= (a, f(a^{\prime}))$. 
\end{example}

\begin{example}\label{Example: What is thetaf in CAlgop}
Let $R$ be a commutative rig. In the tangent category $\CAlg{R}^{\op}$ given a map $f^{\op}:B \to A$ in $\CAlg{R}^{\op}$, which we use as the representation of the map $f:A \to B$ in $\CAlg{R}$, the horizontal bundle of $f^{\op}$ is represented by the algebra $\Sym_A(\Kah{A}{R}) \otimes_A B \cong T_{A/R} \times_A B$; this is seen via the calculation:
\begin{prooftree}
    \AxiomC{$f^{\op}:B \to A$ in $\CAlg{R}^{\op}$}\doubleLine
    \UnaryInfC{$f:A \to B$ in $\CAlg{R}$}
    \UnaryInfC{$\Kah{A}{R} \otimes_A B \xrightarrow{\mathrm{d}a \otimes b \mapsto b\mathrm{d}(f(a))} \Kah{B}{R}$ in $\Mod{B}$}
    \UnaryInfC{$\Sym_{B}(\Kah{A}{R} \otimes_A B) \to \Sym_B(\Kah{B}{R})$ in $\CAlg{B}$}
    \AxiomC{$\Sym_B(\Kah{A}{R} \otimes_A B) \cong \Sym_A(\Kah{A}{R}) \otimes_A B$ in $\CAlg{B}$}
    \BinaryInfC{$\Sym_A(\Kah{A}{R}) \otimes_A B \xrightarrow{\mathrm{d}a \otimes b \mapsto b\mathrm{d}(f(a))} \Sym_B(\Kah{B}{R})$ in $\CAlg{R}$}
    \UnaryInfC{$\theta_f:\Sym_B(\Kah{B}{R}) \to \Sym_A(\Kah{A}{R}) \otimes_A B$ in $\CAlg{R}^{\op}$}\doubleLine
    \UnaryInfC{$\theta_f:T_{B/R} \to T_{A/R} \times_A B$}
\end{prooftree}
where the elments $\mathrm{d}a$ are the generating elements of the K{\"a}hler differentials. In this case we also see that the horizontal descent $\theta_f$ is the (opposite) of the map induced by the assignment $\mathrm{d}a \otimes b \mapsto b\mathrm{d}\big(f(a)\big).$
\end{example}

\begin{example}\label{Example: What is thetaf in Schemes}
Let $S$ be a scheme. The construction of the horizontal descent $\theta_f$ for $\Sch_{/S}$ is much the same as the construction of $\theta_f$ for $\CAlg{R}^{\op}$ above, save we depend on some slight differences in structure. Recall that the relative cotangent sequence for schemes \cite[Proposition II.8.11]{Hartshorne} says that for any morphisms $f:X \to Y$, $g:Y \to S$ the sequence of quasi-coherent sheaves on $X$
\[
\begin{tikzcd}
f^{\ast}(\Kah{Y}{S}) \ar[r]{v_{X/Y/S}} & \Kah{X}{S} \ar[r] &  \Kah{X}{Y} \ar[r] & 0
\end{tikzcd}
\]
is exact. In this case the tangent scheme of $X$ is $T_{X/S} := \underline{\Spec}_X(\underline{\Sym}_{\Ocal_X}(\Kah{X}{S}))$ and the horizontal bundle for $f:X \to Y$ is the scheme
\begin{align*}
T_{Y/S} \times_Y X &\cong \underline{\Spec}_Y\left(\underline{\Sym}_{\Ocal_Y}(\Kah{Y}{S})\right) \times_Y X \cong \underline{\Spec}_X\left(f^{-1}\underline{\Sym}_{\Ocal_Y}(\Kah{Y}{S}) \otimes_{f^{-1}\Ocal_Y} \Ocal_X\right) \\
&\cong\underline{\Spec}_X\left(\underline{\Sym}_{\Ocal_X}\left(f^{-1}\Kah{Y}{S} \otimes_{f^{-1}\Ocal_Y} \Ocal_X\right)\right) = \underline{\Spec}_X\left(\underline{\Sym}_{\Ocal_X}\left(f^{\ast}\Kah{Y}{S}\right)\right).
\end{align*}
The horizontal descent $\theta_f:T_{X/S} \to T_{Y/S} \times_Y X$ is then the map
\[
\underline{\Spec}_X\left(\underline{\Sym}_{\Ocal_X}\left(\Kah{X}{S}\right)\right) \xrightarrow{\underline{\Spec}_X(\underline{\Sym}_{\Ocal_X}(v_{X/Y/S}))} \underline{\Spec}_{X}\left(\underline{\Sym}_{\Ocal_X}\left(f^{\ast}\Kah{Y}{S}\right)\right).
\]
\end{example}
\begin{example}\label{Example: What is thetaf in a CDC}
Let $\Cscr$ be a CDC. Then every map $f:X \to Y$ is $p$-carrable because the diagram
\[
\begin{tikzcd}
X \times Y \ar[r]{}{f \times \id} \ar[d, swap]{}{\pi_0} & Y \times Y \ar[d]{}{\pi_0} \\
X \ar[r, swap]{}{f} & Y
\end{tikzcd}
\]
is always a pullback (where recall that $TY = Y \times Y$). In this case the horizontal descent $\theta_f:X \times X \to X \times Y$ is $\theta_f \colon= \langle \pi_0, D[f]\rangle$. In particular, if $f:X \to Y$ is linear in the CDC sense, that is, if $D[f] = f \circ \pi_1$, then $\theta_f = \id_X \times f$.
\end{example}

Let us now prove some structural properties of the horizontal descent maps with regards to how they interact with compositions, isomorphisms, and the tangent bundle functor itself. For the remainder of this section we work in a tangent category $\Cscr$. 

\begin{lemma}\label{Lemma: Thetaf and isos}
If $f:X \to Y$ is an isomorphism then it is $p$-carrable and $\theta_f:TX \to f^{\ast}(TY)$ is an isomorphism. Moreover, in such a situation we can take $\theta_f$ to be the identity.
\end{lemma}
\begin{proof}
If $f:X \to Y$ is an isomorphism then the diagram
\[
\begin{tikzcd}
TX \ar[r]{}{Tf} \ar[d, swap]{}{p_X} & TY \ar[d]{}{p_Y} \\
X \ar[r, swap]{}{f} & Y
\end{tikzcd}
\]
is a pullback because $Tf$ is an isomorphism as well. Thus in this case we have that $TX \cong f^{\ast}(TY)$. The final claim of the lemma follows upon simply making $f^{\ast}(TY) := TX$, which may be done freely.
\end{proof}
\begin{lemma}\label{Lemma: Thetaf and composition}
Let $f:X \to Y$ and $g:Y \to Z$ be $p$-carrable morphisms such that $g \circ f$ is also $p$-carrable. Then the following equality holds: 
\[
\theta_{g \circ f} = \gamma \circ (\id_X \times \theta_g) \circ \theta_f
\]
where $\gamma:X \times_Y (Y \times_Z TZ) \xrightarrow{\cong} X \times_Z TZ$ is the canonical (unique) natural isomorphism. 
\end{lemma}
\begin{proof}
Note that $\theta_{g \circ f}$ is the unique morphism making the diagram
\[
\begin{tikzcd}
TX \ar[dr, dashed]{}{\exists!\theta_{g \circ f}} \ar[drr, bend left = 20]{}{Tg \circ Tf} \ar[ddr, swap, bend right = 20]{}{p_X} & & \\
 & X \times_Z TZ \ar[r]{}{\pr_1} \ar[d, swap]{}{\pr_0} & TZ \ar[d]{}{p_Z}  \\
 & X \ar[r, swap]{}{g \circ f} & Z
\end{tikzcd}
\]
commute. A routine check shows that the composite
\[
TX \xrightarrow{\theta_f} X \times_Y TY \xrightarrow{\id_X \times \theta_g} X \times_Y (Y \times_Z TZ) \xrightarrow[\cong]{\gamma} X \times_Z TZ
\]
does precisely this.
\end{proof}

When examining how the tangent bundle functor interacts with the horizontal descent of a $p$-carrable morphism, we are obliged to compare $T\theta_f$ with $\theta_{Tf}$. As such, we must confront the fact that there is some subtlety in the way that $T\theta_f$ and $\theta_{Tf}$ interact. On one hand, we have that $T\theta_f:T^2X \to T(f^{\ast}(TY)) \cong T(X \times_Y TZ)$, while $\theta_{Tf}: T^2X \to T(f)^{\ast}(T^2Y).$ When $T$ preserves the pullback
\[
\begin{tikzcd}
f^{\ast}(TY) \ar[r]{}{\pr_1} \ar[d, swap]{}{\pr_0} & TY \ar[d]{}{p_Y} \\
X \ar[r, swap]{}{f} & Y
\end{tikzcd}
\]
this amounts to the difference between the unique morphisms making the following diagrams commute: 
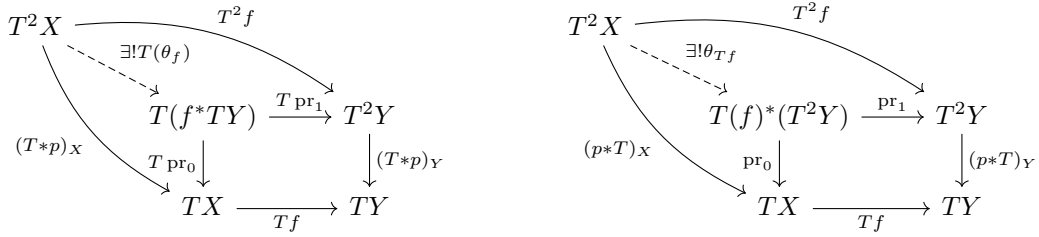
\[\begin{tikzcd}
T^2X \ar[drr, bend left = 20]{}{T^2f}  \ar[ddr, swap, bend right = 20]{}{(T \ast p)_X} \ar[dr, dashed]{}{\exists!T \theta_f} \\
 & T(f^{\ast}TY) \ar[r]{}{T\pr_1} \ar[d, swap]{}{T\pr_0} & T^2Y \ar[d]{}{(T \ast p)_Y} \\
 & TX \ar[r, swap]{}{Tf} & TY
\end{tikzcd} \qquad \qquad \begin{tikzcd}
T^2X \ar[drr, bend left = 20]{}{T^2f}  \ar[ddr, swap, bend right = 20]{}{(p \ast T)_X} \ar[dr, dashed]{}{\exists!\theta_{Tf}} \\
 & T(f)^{\ast}(T^2Y) \ar[r]{}{\pr_1} \ar[d, swap]{}{\pr_0} & T^2Y \ar[d]{}{(p \ast T)_Y} \\
 & TX \ar[r, swap]{}{Tf} & TY
\end{tikzcd}\]
In particular, note that the morphisms which make $T^2Y$ into a differential $TY$-bundle are \emph{different} in each case. As such, while the underlying objects $T(f)^{\ast}(T^2Y)$ and $T(f^{\ast}TY)$ are certainly isomorphic, the isomorphism between them which translates between $T \theta_f$ and $\theta_{Tf}$ must also mediate between the different differential bundle structures on $T^2Y$ based on bundle projections $(T \ast p)_Y$ or $(p \ast T)_Y$, respectively.
\begin{proposition}\label{Proposition: Tangent of thetaf}
Let $f:X \to Y$ be a $p$-carrable morphism. Then there is an isomorphism $\tilde{c}:T(f^{\ast}TY) \to T(f)^{\ast}(T^2Y)$ for which the following equality holds:
\[
\theta_{Tf} \circ c_X = \tilde{c} \circ T\theta_f
\]
\end{proposition}
\begin{proof}
Recall that the canonical flip $c:T^2 \Rightarrow T^2$ is a natural involution (and hence isomorphism) which makes the following diagram commute: 
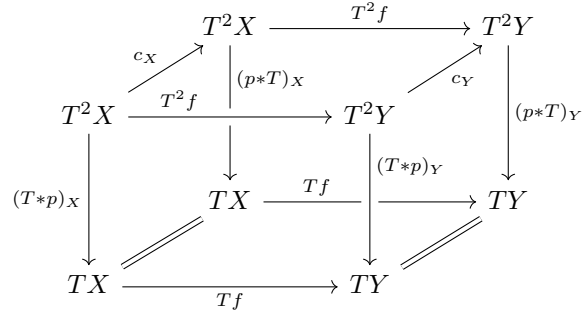
\[
\begin{tikzcd}
 & T^2X \ar[rr]{}{T^2f} \ar[dd, near start]{}{(p \ast T)_X} & & T^2Y \ar[dd]{}{(p \ast T)_Y} \\
T^2X \ar[dd, swap]{}{(T \ast p)_X} \ar[ur]{}{c_X} \ar[rr, near start, crossing over]{}{T^2f} & & T^2Y \ar[ur, swap]{}{c_Y}  \\
 & TX \ar[rr, near start]{}{Tf} & & TY \\
TX \ar[ur, equals] \ar[rr, swap]{}{Tf} & & TY \ar[ur, equals]
\ar[from = 2-3, to = 4-3, crossing over, near start]{}{(T \ast p)_Y}
\end{tikzcd}
\]
This then implies that there is an isomorphism $\tilde{c}:T(f^{\ast}TY) \to T(f)^{\ast}(T^2Y)$ which make the following diagrams commute: 
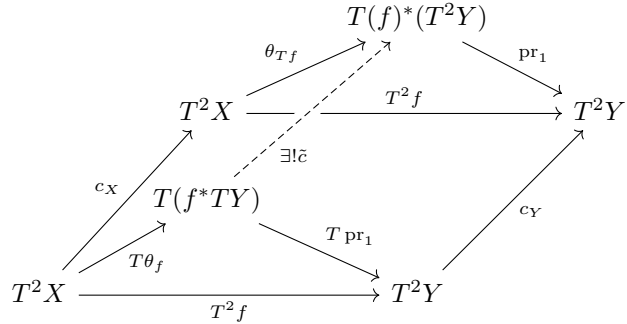
\[
\begin{tikzcd}
& & T(f)^{\ast}(T^2Y) \ar[dr]{}{\pr_1} \\
& T^2X \ar[ur]{}{\theta_{Tf}} \ar[rr]{}{T^2f} & & T^2Y \\
& T(f^{\ast}TY) \ar[dr]{}{T\pr_1} \ar[uur, near start, crossing over, swap, dashed]{}{\exists!\tilde{c}} \\
T^2X \ar[uur]{}{c_X} \ar[ur, swap]{}{T\theta_f} \ar[rr, swap]{}{T^2f} & & T^2Y \ar[uur, swap]{}{c_Y}
\end{tikzcd}\] 
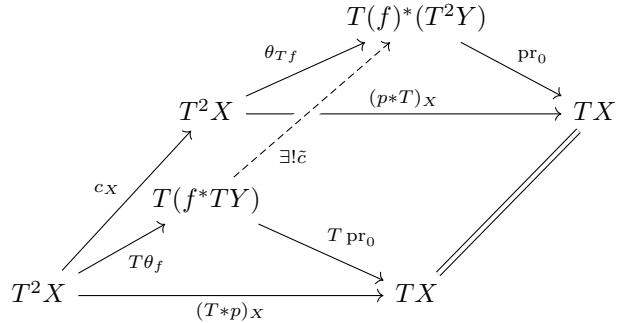
\[ \begin{tikzcd}
& & T(f)^{\ast}(T^2Y) \ar[dr]{}{\pr_0} \\
& T^2X \ar[ur]{}{\theta_{Tf}} \ar[rr]{}{(p \ast T)_X} & & TX \\
& T(f^{\ast}TY) \ar[dr]{}{T\pr_0} \ar[uur, near start, crossing over, swap, dashed]{}{\exists!\tilde{c}} \\
T^2X \ar[uur]{}{c_X} \ar[ur, swap]{}{T\theta_f} \ar[rr, swap]{}{(T \ast p)_X} & & TX \ar[uur, equals]
\end{tikzcd}
\]
Picking out the common face of both triangular prisms gives the equation claimed in the proposition.
\end{proof}

Some immediate but useful consequences of this proposition extended to arbitrary powers of the tangent bundle functor is given below.
\begin{corollary}\label{Cor: pcarrable is power of Tm stable}
For any $m \in \N$ if $f:X \to Y$ is $p$-carrable then so too is $T^mf:T^mX \to T^mY$.
\end{corollary}
\begin{proof} Begin by observing that the $m = 0$ case is a tautology while the $m = 1$ case follows immediately from Proposition \ref{Proposition: Tangent of thetaf}. Assume $m \geq 2$ and note that we induce natural isomorphisms $\tilde{c}:T^{m+1} \Rightarrow T^{m+1}$ by pre-and-post-whiskering the canonical flip $c$ with the various powers of $T$ in order to witness that the cospan
\[
\begin{tikzcd}
 & T^{m+1}Y \ar[d]{}{(T^m \ast p)_X} \\
T^mX \ar[r, swap]{}{T^mf} & T^mY
\end{tikzcd}
\]
is isomorphic to the cospan:
\[
\begin{tikzcd}
 & T^{m+1}Y \ar[d]{}{(p \ast T^m)_X} \\
T^mX \ar[r, swap]{}{T^mf} & T^mY
\end{tikzcd}
\]
Since $f:X\to Y$ is $p$-carrable, the pullback
\[
\begin{tikzcd}
T^{m+1}(X) \ar[dr, dashed]{}{\exists!\theta_{T^mf}} \ar[drr, bend left = 20]{}{T^mf} \ar[ddr, swap, bend right = 20]{}{(p \ast T^m)_X} \\
 & T^m(f)^{\ast}(T^{m+1}Y) \ar[r]{}{\pr_1} \ar[d, swap]{}{\pr_0} & T^{m+1}Y \ar[d]{}{(p \ast T^m)_Y} \\
 & T^mX \ar[r, swap]{}{T^mf} & T^mY
\end{tikzcd}
\]
exists for all $m \in \N$ because $T^mf$ admits a pullback against $(T^m \ast p)_Y$ for all $m \in \N$ and so, by passing through the cospan isomorphisms, $T^mf$ also admits a pullback against $(p \ast T^m)_Y$. Furthermore, all powers $T^k$ preserve said pullback because $T^{m+k}f = T^k(T^mf)$ and all powers of $T$ preserve the pullback of $f$ along $p_Y$. Thus $T^mf$ is $p$-carrable.
\end{proof}

\begin{corollary}\label{Cor: The combinatorial swicheroo fo rhigher power s of T on thetaf}
Let $f:X \to Y$ be a $p$-carrable morphism. Then for all $m \in \N$ there exists an isomorphism $\tilde{c}_m:T\big(T^m(f)^{\ast}(T^{m+1}Y)\big) \to T^{m+1}(f)^{\ast}(T^{m+2}Y)$ for which the following equality holds: 
\[
 \tilde{c}_{m}\circ T\theta_{T^mf} = \theta_{T^{m+1}f} \circ (c \ast T^m)_X .
\]
\end{corollary}
\begin{proof}
Begin by observing that since $f:X\to Y$ is $p$-carrable, the pullback
\[
\begin{tikzcd}
T^{m+1}(X) \ar[dr, dashed]{}{\exists!\theta_{T^{m+1}f}} \ar[drr, bend left = 20]{}{T^mf} \ar[ddr, swap, bend right = 20]{}{(p \ast T^m)_X} \\
 & T^m(f)^{\ast}(T^{m+1}Y) \ar[r]{}{\pr_1} \ar[d, swap]{}{\pr_0} & T^{m+1}Y \ar[d]{}{(p \ast T^m)_Y} \\
 & T^mX \ar[r, swap]{}{T^mf} & T^mY
\end{tikzcd}
\]
exists for all $m \in \N$ by Corollary \ref{Cor: pcarrable is power of Tm stable}. Moreover, $p$-carrability also implies that all powers of the tangent bundle functor $T^k$ preserve said pullback squares for all $k \in \N$. Thus on one hand we have that $T\theta_{T^mf}$ arises as the unique map rendering the diagram
\[
\begin{tikzcd}
T^{m+2}X \ar[drr, bend left = 20]{}{T^{m+2}f} \ar[ddr, swap, bend right = 20]{}{(T \ast p \ast T^m)_X} \ar[dr, dashed]{}{\exists!T\theta_{T^mf}} & & \\
 & T\big(T^m(f)^{\ast}(T^{m+1}Y)\big) \ar[r]{}{T\pr_1} \ar[d, swap]{}{T\pr_0} & T^{m+2}Y \ar[d]{}{(T \ast p \ast T^m)_Y} \\
 & T^{m+1}X \ar[r, swap]{}{T^mf} & T^{m+1}Y
\end{tikzcd}
\]
commutative. However, a routine check shows that we can induce commuting triangular prisms as in the proof of Proposition \ref{Proposition: Tangent of thetaf} and hence produce a unique isomorphism $\tilde{c}_m:T\big(T^m(f)^{\ast}(T^{m+1}Y)\big) \to T^{m+1}(f)^{\ast}(T^{m+2}Y)$ which renders the diagram
\[
\begin{tikzcd}
T^{m+2}X \ar[rr]{}{(c \ast T^m)_X} \ar[d, swap]{}{T\theta_{T^mf}} & & T^{m+2}X \ar[d]{}{\theta_{T^{m+1}f}} \\
T\big(T^m(f)^{\ast}(T^{m+1}Y\big) \ar[rr, dashed, swap]{}{\exists!\tilde{c}_m} & & T^{m+1}(f)^{\ast}(T^{m+2}Y)
\end{tikzcd}
\]
commutative. The identity encoded by said commuting diagram is precisely the claimed identity $\theta_{T^{m+1}f} \circ (c \ast T^m)_X = \tilde{c}_m \circ T\theta_{T^mf}$. 
\end{proof}
\begin{corollary}\label{Cor: The other higher powers of T to thetaf in one line}
Let $f:X \to Y$ be a $p$-carrable morphism. Then for all $m \in \N$ there are isomorphisms $C$ and $\tilde{C}$ for which
\[
\tilde{C}_m \circ T^m\theta_f \circ C_m = \theta_{T^{m}f}.
\]
\end{corollary}
\begin{proof}
Observe that for all $k, \ell \in \N$, applying Corollary \ref{Cor: The combinatorial swicheroo fo rhigher power s of T on thetaf} and using the fact that $c_{T^kX} = (c \ast T^k)_X$ is an involution gives
\begin{align*}
T^{\ell}(\theta_{T^{k+1}f}) &= T^{\ell}\left(\tilde{c}_k \circ T\theta_{T^{k}f} \circ (c \ast T^k)_X\right) = T^{\ell}(\tilde{c}_k) \circ T^{\ell+1}(\theta_{T^{k}f}) \circ T^{\ell}(c \ast T^k)_X \\
&= T^{\ell}(\tilde{c}_k) \circ T^{\ell+1}(\theta_{T^kf}) \circ (T^{\ell} \ast c \ast T^k)_X.
\end{align*}
Consequently, for any $0 \leq k \leq m$,
\[
T^k(\theta_{T^{m+1-k}f}) = T^k\tilde{c}_{m-k} \circ T^{k+1}(\theta_{T^{m-k}f}) \circ (T^k \ast c \ast T^{m-k})_X.
\]
Proceeding recursively in $m$ and terminating when $m = 0$ allows us to deduce that
\begin{align*}
\theta_{T^{m}f} &= \tilde{c}_m \circ T\theta_{T^mf} \circ (c \ast T^m)_X = \tilde{c}_m \circ T\tilde{c}_{m-1} \circ T^2\theta_{T^{m-1}f} \circ (T \ast c \ast T^{m-1})_X \circ (c \ast T^m)_X = \\
&= \tilde{c}_m \circ T\tilde{c}_{m-1} \circ \cdots T^m\tilde{c}_0 \circ T^{m}\theta_f \circ (T^m \ast c)_X \circ (T^{m-1} \ast c \ast T)_X \circ \cdots \circ (c \ast T^m)_X.
\end{align*}
Setting $\tilde{C} := \tilde{c}_m \circ \cdots \circ T^m\tilde{c}_0$ and $C := (T^m \ast c)_X \circ \cdots \circ (c \ast T^m)_X$ thus gives that
\[
\theta_{T^{m}f} = \tilde{C} \circ T^{m}\theta_f \circ C
\]
as desired. 
\end{proof}

We now study how the horizontal descents interact with pullbacks.

\begin{proposition}\label{Prop: How thetaf pulls back}
Let $f:X \to Z$ be a $p$-carrable morphism and assume we have a $T$-pullback: 
\[
\begin{tikzcd}
X \times_Z Y \ar[r]{}{\pr_1} \ar[d, swap]{}{\pr_0} & Y \ar[d]{}{g} \\
X \ar[r, swap]{}{f} & Z
\end{tikzcd}
\]
If $\pr_1$ is $p$-carrable, then there is a morphism $\hat{\pi}:\pr_1^{\ast}(TY) \to f^{\ast}(TZ)$ which makes
\[
\begin{tikzcd}
\pr_1^{\ast}(TY) \ar[r]{}{\pi_1} \ar[d, swap]{}{\hat{\pi}} & TY \ar[d]{}{Tg} \\
f^{\ast}(TZ) \ar[r,swap]{}{\pr_1} & TZ
\end{tikzcd}
\]
into a pullback square and the following equality holds:
\[
\theta_{\pr_1} = \rho \circ (\theta_f \times \id_{TY}) \circ \varphi.
\]
where $\rho:f^{\ast}(TZ) \times_{TZ} TY \to \pr_1^{\ast}(TZ)$ and $\varphi:T(X \times_Y Z) \to TX \times_{TY} TZ$ are the unique induced isomorphisms. 
\end{proposition}
\begin{proof}
The final statement of the proposition is completely formal, so it suffices to first construct the morphism $\hat{\pi}:\pr_1^{\ast}(TY) \to f^{\ast}(TZ)$ and then prove that the square 
\[
\begin{tikzcd}
\pr_1^{\ast}(TY) \ar[r]{}{\pi_1} \ar[d, swap]{}{\hat{\pi}} & TY \ar[d]{}{Tg} \\
f^{\ast}(TZ) \ar[r,swap]{}{\pr_1} & TZ
\end{tikzcd}
\]
is a pullback. To construct $\hat{\pi}$, we note first that the diagram
\[
\begin{tikzcd}
\pr_1^{\ast}(TY) \ar[d, swap]{}{\pr_0} \ar[r]{}{\pr_1} & TY \ar[d]{}{Tg} \\
X \times_Z Y \ar[d, swap]{}{\pr_0} & TZ \ar[d]{}{p_Z} \\
X \ar[r, swap]{}{f} & Z
\end{tikzcd}
\]
commutes, making $(\pr_1^{\ast}(TY), \pr_0\circ \pr_0, Tg \circ \pr_1)$ into a cone over $X \xrightarrow{f} Z \xleftarrow{p} TZ$. As such, we define $\hat{\pi}:\pr_1^{\ast}(TY) \to f^{\ast}(TZ)$ as the induced unique cone map, that is $\hat{\pi} = \langle \pr_0 \circ \pr_0, Tg \circ \pr_1\rangle$. Now assume that we have a commuting diagram:
\[
\begin{tikzcd}
W \ar[r]{}{\alpha} \ar[d, swap]{}{\beta} & TY \ar[d]{}{Tg} \\
f^{\ast}(TZ) \ar[r, swap]{}{\pr_1} & TZ
\end{tikzcd}
\]
Since 
\[
f \circ \pr_0 \circ \beta = p_Z \circ \pr_1 \circ \beta = p_Z \circ Tg \circ \alpha = g \circ p_Y \circ \alpha
\]
we get a unique morphism $\gamma:W \to X \times_Z Y$ making the diagram
\[
\begin{tikzcd}
W \ar[drr, bend left = 20]{}{p_Y \circ \alpha} \ar[ddr, swap, bend right = 20]{}{\pr_0 \circ \beta} \ar[dr, dashed]{}{\exists!\,\gamma} \\
 & X \times_Z Y \ar[r]{}{\pr_1} \ar[d, swap]{}{\pr_0} & Y \ar[d]{}{g} \\
 & X \ar[r, swap]{}{f} & Z
\end{tikzcd}
\]
commute. But then since
\[
\begin{tikzcd}
W \ar[r]{}{\alpha} \ar[d, swap]{}{\gamma} & TY \ar[d]{}{p_Y} \\
X \times_Z Y \ar[r, swap]{}{\pr_1} & Y
\end{tikzcd}
\]
commutes, we get a unique morphism $\delta$ making the diagram
\[
\begin{tikzcd}
W \ar[drr, bend left = 20]{}{\alpha} \ar[ddr, swap, bend right = 20]{}{\gamma} \ar[dr]{}{\delta} \\
 &\pr_1^{\ast}(TY) \ar[r]{}{\pr_1} \ar[d, swap]{}{\pr_0} & TY \ar[d]{}{p_Y} \\
 & X \times_Z Y \ar[r, swap]{}{\pr_1} & Z
\end{tikzcd}
\]
commute. We now show that $\delta$ makes the following diagram commute:
\[
\begin{tikzcd}
W \ar[drr, bend left = 20]{}{\alpha} \ar[ddr, swap, bend right = 20]{}{\beta} \ar[dr, dashed]{}{\exists!\,\delta} \\
 &\pr_1^{\ast}(TY) \ar[r]{}{\pr_1} \ar[d, swap]{}{\pr_0} & TY \ar[d]{}{Tg} \\
 & f^{\ast}(TZ) \ar[r, swap]{}{\pr_1} & TZ
\end{tikzcd}
\]
Since we already know that $\alpha = \pr_1 \circ \delta$, it suffices to prove the identity regarding $\beta$. Then we compute that
\begin{align*}
\hat{\pi} \circ \delta =\left\langle \pr_0\circ \pr_0, Tg \circ \pr_1 \right\rangle \circ \delta = \left\langle \pr_0 \circ \pr_0 \circ \delta, Tg \circ \pr_1 \circ \delta\right\rangle = \left\langle \pr_0 \circ \gamma, Tg \circ \alpha \right\rangle = \langle \pr_0 \circ \beta, \pr_1 \circ \beta\rangle = \beta.
\end{align*}
Thus the diagram commutes. It is also routine to check that $\delta$ is the unique such cone map, as it was constructed stage-by-stage from other unique cone morphisms. Thus the diagram
\[
\begin{tikzcd}
\pr_1^{\ast}(TY) \ar[r]{}{\pi_1} \ar[d, swap]{}{\hat{\pi}} & TY \ar[d]{}{Tg} \\
f^{\ast}(TZ) \ar[r,swap]{}{\pr_1} & TZ
\end{tikzcd}
\]
is a pullback and hence completes the proof of the proposition.
\end{proof}

\section{The Relative Cotangent Sequence}

One of the most important conceptual and technical tools in both differential geometry and algebraic geometry is the relative cotangent sequence. In algebraic geometry, for a morphism $f:X \to Y$ of schemes there is an exact sequence
\[
\begin{tikzcd}
f^{\ast}\Kah{Y}{S} \ar[r] & \Kah{X}{S} \ar[r] & \Kah{X}{Y} \ar[r] & 0
\end{tikzcd}
\]
of quasi-coherent sheaves on $X$. One important application of this sequence is that it can be used to probe for when $f$ is formally unramified (this occurs if and only if $\Kah{X}{Y} \cong 0$). On the other hand, by passing through the opposite equivalence $\DBun(X) \simeq \QCoh(X)^{\op}$, we find that the relative cotangent sequence reflects to give a sequence of differential bundles
\[
\begin{tikzcd}
X \ar[r] & \underline{\Spec}_X\left(\underline{\Sym}_{\Ocal_X}\left(\Kah{X}{Y}\right)\right) \ar[r] & T_{X/S} \ar[r]{}{\theta_f} & X \times_Y T_{Y/S} \ar[equals, r] & f^{\ast}T_{Y/S}
\end{tikzcd}
\]
where the map to the last bundle $T_{Y/S} \times_X Y$ follows from the following chain of isomorphisms
\[
\underline{\Spec}_X\left(\underline{\Sym}_{\Ocal_X}\left(f^{\ast}\Kah{Y}{S}\right)\right) \cong \underline{\Spec}_{X}\left(f^{\ast}\underline{\Sym}_{\Ocal_Y}\left(\Kah{Y}{S}\right)\right) \cong X \times_Y \underline{\Spec}_{Y}\left(\underline{\Sym}_{\Ocal_Y}\left(\Kah{Y}{S}\right)\right) = f^{\ast}(T_{Y/S}).
\]
In differential geometry, for any map $f:X \to Y$ of smooth manifolds, we can always study the sub-bundle $T_{X/Y}$ of $TX$ defined as $T_{X/Y} := \left\lbrace \left(x, \overrightarrow{v}\right) \in TX \; : \; D[f](x) \cdot \overrightarrow{v} = \overrightarrow{0}\right\rbrace$. This ``vertical bundle'' not only gives a relative tangent bundle of $X$ over $Y$ as an object of $\SMan_{/Y}$, but it also provides an exact sequence
\[
\begin{tikzcd}
X \ar[r] & T_{X/Y} \ar[r] & TX \ar[r]{}{\theta_f} & f^{\ast}(TY)
\end{tikzcd}
\]
in the Abelian category $\mathbf{Vec}(X) \simeq \DBun(X).$

In both the algebraic geometry and differential geometry settings, we see that once we are able to produce a ``relative tangent bundle'' $T_{X/Y}$ (this is the bundle $\underline{\Spec}_X(\underline{\Sym}_{\Ocal_X}(\Kah{X}{Y}))$ in the algebraic-geometric setting) we are able to produce a sequence which shows that this bundle $T_{X/Y}$ acts as a kernel to the horizontal descent $\theta_f:TX \to f^{\ast}(TY)$. In this section we prove that this is no accident and show that this hold in a general tangent category. Indeed, we will show that provided a map $f:X \to Y$ is both $p$-carrable \emph{and} $0$-carrable (Definition \ref{Defn: Zero Carrable}), then not only does the bundle $T_{X/Y}$ exist, but it \emph{also} is a kernel to $\theta_f$ in the sense that $T_{X/Y}$ is the limit of the following cospan in $\DBun(X)$:
\[
TX \xrightarrow{\theta_f} f^{\ast}(TY) \xleftarrow{f^\ast(0_Y)} X
\]
We also show that this structure provides the framework for defining morphisms we call $T$-unramified, which are the tangent-categorical analogue of a map $f:X \to Y$ whose tangent $Tf$ has trivial kernel, i.e., no tangent vectors ``ramify'' through $Tf$ by mapping both to zero. While we are primarily focused on the general theory of such maps in this section, they become especially important in Section \ref{Section: Strong TImmersions} when we introduce the tangent-categorical extension of immersions: $T$-immersions. There we will prove that in a Rosick{\'y} tangent category, a map $f$ is $T$-unramified if and only if it is a $T$-immersion (Theorems \ref{Thm: Rosckiy Timmersion iff Tunramified}, \ref{Thm: If C Rosicky then immersion iff unrmaified}). 

\subsection{Zero Carrable Morphisms}
In this section we examine the morphisms $f$ for which there is a $T$-pullback of the zero along the tangent of $f$. These morphisms are intimately related to formally unramified morphisms in algebraic geometry and to the vertical bundle of the tangent bundle in differential geometry. They also allow us to produce a new tangent bundle functor on the slice category $\Cscr_{/X}$ which is \emph{not} necessarily isomorphic to tangent bundle functor on $\Cscr$ ``restricted'' to the slice category $\Cscr_{/X}$. We will see later in this section that these mutual overlaps allow us to define a differential bundle analogue of the relative cotangent sequence of algebraic geometry in an arbitrary tangent category whenever $f:X \to Y$ is also $p$-carrable.

\begin{definition}\label{Defn: Zero Carrable}
Let $f:X \to Y$ be a morphism in a tangent category $\Cscr$. We say that $f$ is \emph{$0$-carrable} if the $T$-pullback
\[
\begin{tikzcd}
T_{X/Y} \ar[r]{}{\iota^f} \ar[d, swap]{}{\iota_f} & Y \ar[d]{}{0_Y} \\
TX \ar[r, swap]{}{Tf} & TY
\end{tikzcd}
\]
exists. We call $T_{X/Y}$ the \emph{relative tangent bundle of $X$ over $Y$}.
\end{definition}

\begin{remark}
    In \cite{Rosicky} and \cite{GeoffRobinDiffStruct} the relative tangent bundle $T_{X/Y}$ for a map $f:X \to Y$ is defined via a \emph{different} universal property than that presented in Definition \ref{Defn: Zero Carrable}. There $T_{X/Y}$ was defined to be the equalizer
    \[
    \begin{tikzcd}
        T_{X/Y} \ar[r]{}{\iota_f} & TX \ar[rrr, shift left = 1]{}{Tf} \ar[rrr, shift right = 1, swap]{}{0_Y \circ p_Y \circ T(f)} & & & TY
    \end{tikzcd}
    \]
    provided that it exists. We take the pullback perspective of Definition \ref{Defn: Zero Carrable} because it corresponds to the perspective taken in \cite{GeoffJSDiffBunComAlg}.
\end{remark}

\begin{remark}
In some settings the tangent relative tangent bundle $T_{X/Y}$ is called the vertical tangent (the idea being that as the pullback of $Tf$ against $0_Y$, all the tangent vectors in $T_{X/Y}$ have to sit ``vertically'' over $X$). In \cite{PronkVooys} it is even denoted as $V_f(X)$; however, in this paper we use the notation $T_{X/Y}$ to line up with the notation Grothendieck used in \cite{EGA44} when describing $T_{X/S}$ as ``fibr{\'e} tangent de $X$ relativement {\`a} $S$'' for an $S$-scheme $\nu:X \to S$ \cite[Below Line 16.5.12.1]{EGA44}.
\end{remark}

By \cite[Proposition 2.4]{GeoffJSDiffBunComAlg} and \cite[Pages 4 -- 5]{Rosicky}, if every map in the slice category $\Cscr_{/Y}$ is $0$-carrable, then the relative tangent bundle $T_{X/Y}$ gives a tangent bundle functor $T_{(-)/Y}:\Cscr_{/Y} \to \Cscr_{/Y}$ which makes $\Cscr_{/Y}$ into a tangent category. These tangent structures are generally different from those which appear on the slice categories in \cite[Proposition 2.5]{GeoffRobinDiffStruct}, as those are given by just restricting the global tangent structure to the slice category. Instead, these give a relative version which aligns with the tangent structures used on $\mathbf{CAlg}_{R}^{\op}$ and $\Sch_{/S}$ in algebraic geometry (as opposed to asserting that the only tangents you work with there are induced by the global K{\"a}hler differentials $\Kah{(-)}{\mathbb{Z}})$.

\begin{proposition}\label{Prop: Bundle projection is zero carrable}
    In any tangent category $\Cscr$, for any object $X$ the projection $p_X:TX \to X$ is $0$-carrable.
\end{proposition}
\begin{proof}
    To show this observe that from the expression of the universality of the vertical lift a l{\`a} \cite[Definition 2.1]{GeoffRobinBundle} shows that the diagram
    \[
    \begin{tikzcd}
        T_2X \ar[rrr]{}{p_X\circ \pi_0 = p_X \circ \pi_1} \ar[d, swap]{}{v} & & & X \ar[d]{}{0_X} \\
        T^2X \ar[rrr, swap]{}{(T \ast p)_X} & & & TX
    \end{tikzcd}
    \]
    is a pullback square with $v = (T \ast \operatorname{add})_X \circ \langle (0 \ast T) \circ \pi_0, \ell_X \circ \pi_1\rangle$. However, this pullback square is exactly a limit of the cospan
    \[
    \begin{tikzcd}
        T^2X \ar[rr]{}{(T \ast p)_X} & & TX & & \ar[ll, swap]{}{0_X} X
    \end{tikzcd}
    \]
    and so illustrates that $p_X:TX \to X$ is $0$-carrable.
\end{proof}

Let us now see how the vertical tangent bundles $T_{X/Y}$ are incarnated in our four main tangent categories of interest. It is worth observing that we can, in complete generality, only say that the vertical bundles exist whenever our tangent category of interest $\Cscr$ admits all pullbacks and that the tangent bundle functor preserves pullback\footnote{The minimal assumptsion, of course, is that every map is $0$-carrable.}. As such, when we consider $\SMan$ we also give a short argument as to \emph{why} the vertical tangent bundles exist.

\begin{example}\label{Example: Relative Tangent bundle in SMan} In $\SMan$ we claim that every map $f:X \to Y$ is $0$-carrable despite the fact that $\SMan$ is not finitely complete. To see this let
\[
T_{X/Y} := \left\lbrace \left(x,\overrightarrow{v}\,\right) \; : \; x \in X, \overrightarrow{v}\,\in T_xX.\,D[f](x)\cdot\overrightarrow{v}=\overrightarrow{0}  \right\rbrace
\]
be the smooth submanifold of $TX$ generated stalkwise by the kernel of the linear maps\footnote{Because the linear maps $D[f](x):T_xX \to T_{f(x)}Y$ are all necessarily smooth maps of vector spaces, their kernels are closed linear subspaces of $T_xX$ and so the product space $\lbrace x \rbrace \times \operatorname{Ker}D[f](x)$ is a closed linear subspace of $\lbrace x \rbrace \times T_xX$. These closed linear subspaces may be patched locally as in $X$ because the transition maps in the atlas defining $X$ are smooth and hence assemble to give a (not-necessarily of full rank) smooth submanifold of $TX$.} $D[f](x)$. The maps $\iota_f:T_{X/Y} \to TX$ and  $\iota^f:T_{X/Y} \to Y$ are then defined by taking $\iota_f$ to be the (smooth) submanifold inclusion and $\iota^f$ to be given by $\iota^f(x,\overrightarrow{v}\,) := f(x)$. The diagram
\[
\begin{tikzcd}
T_{X/Y} \ar[r]{}{\iota^f} \ar[d, swap]{}{\iota_f} & Y \ar[d]{}{0_Y} \\
TX \ar[r, swap]{}{Tf} & TY
\end{tikzcd}
\]
then commutes precisely because each tangent vector $(x,\overrightarrow{v}\,)$ in $T_{X/Y}$ has been chosen so that the vector $\overrightarrow{v}$ lies in the kernel of $D[f](x)$. To see that it is a pullback assume that $M$ is a smooth manifold with maps $g:M \to Y$ and $h:M \to TX$ for which
\[
\begin{tikzcd}
M \ar[r]{}{g} \ar[d, swap]{}{h} & Y \ar[d]{}{0_Y} \\
TX \ar[r, swap]{}{Tf} & TY
\end{tikzcd}
\]
commutes. Now observe that for all $m \in M$ we can write 
\[
h(m) = \left(h_0(m), \overrightarrow{h_1^{h_0(m)}(m)}\right)
\]
where $h_0 = p_X \circ h$ and where $h_1^{h_0(-)}(-)$ is the map which assigns to each $m$ the tangent vector component associated to $h(m)$ anchored at $h_0(m)$. We claim that this factors through $T_{X/Y}$; because $\iota_f$ is monic (it is injective on the level of underlying sets), it suffices to provide such a factorization because it is consequently always unique if it exists. From the fact that $0_Y \circ g = Tf \circ h$ we see that for all $m \in M$ we have
\[
\left(g(m), \overrightarrow{0}\,\right) = (0_Y \circ g)(m) = (Tf \circ h)(m) = \left( f\big(h_0(m)\big), D[f]\bigg(f\big(h_0(m)\big)\bigg)\overrightarrow{h_1^{h_0(m)}\big(m\big)}\,\right).
\]
This informs us simultaneously that $g(m) = (f \circ h_0)(m)$ for all $m \in M$ and also that for all $m \in M$,
\[
D[f]\bigg(f\big(h_0(m)\big)\bigg)\overrightarrow{h_1^{h_0(m)}\big(m\big)} = \overrightarrow{0}
\]
in $T_{f(h_0(m))}Y$. Thus every tangent vector $h_1^{h_0(m)}(m)$ lies in the kernel of $D[f]\big(f(h_0(m))\big)$ and so $h$ factors through the inclusion of $T_{X/Y}$ into $TX$. This implies that the diagram
\[
\begin{tikzcd}
M \ar[ddr, swap, bend right = 20]{}{h} \ar[dr, dashed]{}{\exists!\,h} \ar[drr, bend left = 20]{}{g} \\
& T_{X/Y} \ar[r]{}{\iota^f} \ar[d, swap]{}{\iota_f} & Y \ar[d]{}{0_Y} \\
& TX \ar[r, swap]{}{Tf} & TY
\end{tikzcd}
\]
commutes. Lastly, it is not difficult to see that the tangent bundle functor preserves this pullback, and so this gives us that $f$ is $0$-carrable.
\end{example}

\begin{example}\label{Example: Relative tangent bundle in CAlg}
Let $R$ be a commuutative rig. Because $\CAlg{R}$ is complete and that the tangent bundle functor preserves limits, we find that for every map $f:A \to B$ of commutative $R$-algebras, the pullback
\[
\begin{tikzcd}
A[\epsilon] \times_{B[\epsilon]} B \ar[r]{}{\iota^f} \ar[d, swap]{}{\iota_f} & B \ar[d]{}{0_B} \\
A[\epsilon] \ar[r, swap]{}{f[\epsilon]} & B[\epsilon]
\end{tikzcd}
\]
always exists and so every morphism is $0$-carrable. A quick computation allows us to describe the relative tangent bundle explicitly as: 
\[
T_{A/B} \cong \lbrace (a+x\epsilon, b) \in A[\epsilon] \times B \; | \; f(a)+f(x)\epsilon = b\rbrace \cong \lbrace (a + x\epsilon, f(a)) \; | \; x \in \operatorname{Ker}(f) \rbrace \cong A \ltimes \operatorname{Ker}(f).
\]
That is, $T_{A/B}$ is isomorphic to the infinitesimal extension of $A$ by the module $\operatorname{Ker}(f)$.
\end{example}

\begin{example}\label{Example: Relative Tangent Bundle in CAlgop}
Let $R$ be a commutative rig. In $\CAlg{R}^{\op}$, for any map $f^{\op}:B \to A$ in $\CAlg{R}^{\op}$ we compute that the relative tangent bundle is given by:
\begin{prooftree}
    \AxiomC{$f^{\op}:B \to A$ in $\CAlg{R}^{\op}$}
    \UnaryInfC{$f:A \to B$ in $\CAlg{R}$}
    \UnaryInfC{The $B$-module of relative K{\"a}hler differentials $\Kah{B}{A}$ exists}
    \UnaryInfC{Relative tangent bundle $T_{B/A} \cong \Sym_B(\Kah{B}{A})$}
\end{prooftree}
In particular, every map in $\CAlg{R}^{\op}$ is $0$-carrable because $\CAlg{R}^{\op}$ is complete and the tangent bundle functor preserves limits.
\end{example}
\begin{example}\label{Example: Relative tangent bundle in Sch}

Let $S$ be a base scheme. Because $\Sch_{/S}$ is finitely complete and the tangent bundle functor preserves limits, every map $f:X \to Y$ is $0$-carrable. In this case we also compute that from the relative cotangent sequence
\[
\begin{tikzcd}
f^{\ast}(\Kah{Y}{S}) \ar[r] & \Kah{X}{S} \ar[r] & \Kah{X}{Y} \ar[r] & 0
\end{tikzcd}
\]
in $\QCoh(X)$, the relative tangent bundle in $\Sch_{/S}$ is $T_{X/Y}\cong\underline{\Spec}_X(\underline{\Sym}_{\Ocal_X}(\Kah{X}{Y}))$.
\end{example}

As a prelude towards defining unramified morphisms in tangent categories, we now argue that whenever $f:X \to Y$ is $0$-carrable, the unique map $X \to T_{X/Y}$ is monic. So if $f: X \to Y$ is $0$-carrable, let $0_{X/Y}: X \to T_{X/Y}$ be the unique map which makes the following diagram commute:
\[
\begin{tikzcd}
X \ar[drr, bend left =20]{}{f} \ar[ddr, bend right = 20, swap]{}{0_X} \ar[dr, dashed]{}{\exists!0_{X/Y}} \\
 & T_{X/Y} \ar[r]{}{\iota^f} \ar[d, swap]{}{\iota_f} & Y \ar[d]{}{0_Y} \\
 & TX \ar[r, swap]{}{Tf} & TY
\end{tikzcd}
\]
While it is realitvely immediate that this map $0_{X/Y}$ is a section, we record it explicitly for the sake of completeness.

\begin{lemma}\label{Lemma: Unique map from X to relative Tangent Bundle}
For any $0$-carrable map $f:X \to Y$ in a tangent category $\Cscr$, the unique morphism $0_{X/Y}:X \to T_{X/Y}$ is a  section of $p_{X/Y}:T_{X/Y} \to X$.
\end{lemma}
\begin{proof}
The map $p_{X/Y}:T_{X/Y} \to X$ is defined via the composite:
\[
\begin{tikzcd}
T_{X/Y} \ar[r]{}{\iota_f} & TX \ar[r]{}{p_X} & X
\end{tikzcd}
\]
It then follows from the definition of $0_{X/Y}$ that the equation $p_{X/Y} \circ 0_{X/Y} = p_X \circ \iota_f \circ 0_{X/Y} = p_X \circ 0_X = \id_X$ holds.
\end{proof}

Given two $0$-carrable morphisms $f:X \to Y$ and $g:Y \to Z$, it is not at all immediate that the composite $g \circ f:X \to Z$ should be $0$-carrable. Instead let us now give a necessary and sufficient condition for the $0$-carrability of $g \circ f$.

\begin{proposition}\label{Prop: Zero carrable nec and suf}
Let $f:X \to Y$ and $g:Y \to Z$ be morphisms in a tangent category $\Cscr$ and assume that $g$ is $0$-carrable. Then $g \circ f$ is $0$-carrable if and only if the $T$-pullback
\[
\begin{tikzcd}
TX \times_{TY} T_{Y/Z} \ar[r]{}{\pi_1} \ar[d, swap]{}{\pi_0} & T_{Y/Z} \ar[d]{}{\iota_g} \\
TX \ar[r, swap]{}{Tf} & TY
\end{tikzcd}
\]
exists.
\end{proposition}
\begin{proof}
We will show that if $TX \times_{TY} T_{Y/Z}$ exists then it serves as a limit of the cone 
\[
TX \xrightarrow{Tg \circ Tf} TZ \xleftarrow{0_Z} Z
\]
and then that if $T_{X/Z}$ exists, it serves as a  pullback of the cospan 
\[
T_{Y/Z} \xrightarrow{\iota_{g}} TY \xleftarrow{Tf} TX.
\]

$\implies:$ If the $T$-pullback given in the statement of the proposition exists, then stacking the pullback squares
\[
\begin{tikzcd}
TX \times_{TY} T_{Y/Z} \ar[r]{}{\pi_1} \ar[d, swap]{}{\pi_0} & T_{Y/Z} \ar[d]{}{\iota_g} \ar[r]{}{\iota^g} & Z \ar[d]{}{0_Z} \\
TX \ar[r, swap]{}{Tf} & TY \ar[r, swap]{}{Tg} & TZ
\end{tikzcd}
\]
and applying the Pullback Lemma yields that $TX \times_{TY} T_{Y/Z}$ is a $T$-pullback of the cospan 
\[
TX \xrightarrow{Tg \circ Tf} TZ \xleftarrow{0_Z} Z.
\]

$\impliedby:$ Assume that $T_{X/Z}$ exists. To see that $T_{X/Y}$ is the apex of a cone over $TX \xrightarrow{Tf} TY \xleftarrow{\iota_{g}}T_{Y/Z}$, note that the fact that $T(g \circ f) = T(g) \circ T(f)$ and the resulting commutativity of the square
\[
\begin{tikzcd}
T_{X/Y} \ar[r]{}{\iota^{g \circ f}} \ar[d, swap]{}{Tf \circ \iota_{g \circ f}} & Z \ar[d]{}{0_Z} \\
TY \ar[r, swap]{}{Tg} & TZ
\end{tikzcd}
\]
implies that there is a unique morphism $\gamma_{g \circ f}$ which renders
\[
\begin{tikzcd}
T_{X/Z} \ar[dr, dashed]{}{\exists!\gamma_{g \circ f}} \ar[drr, bend left = 20]{}{\iota^{g \circ f}} \ar[ddr, swap, bend right = 20]{}{Tf \circ \iota_{g \circ f}} \\
 & T_{Y/Z} \ar[r]{}{\iota^g} \ar[d, swap]{}{\iota_g} & Z \ar[d]{}{0_Z} \\
 & TY \ar[r, swap]{}{Tg} & TZ
\end{tikzcd}
\]
commutative. From here we compute
\[
0_Z \circ \iota^g \circ \gamma_{g \circ f} = 0_Z \circ \iota^{g \circ f} = Tg \circ Tf \circ \iota_{g \circ f} = Tg \circ \iota_g \circ \gamma_{g \circ f} = 0_Z \circ \iota^g \circ \gamma_{g \circ f}.
\]
Since $0_Z$ is monic, we have $\iota_g \circ \gamma_{g \circ f} = Tf \circ \iota_{g \circ f}$ and hence that the square
\[
\begin{tikzcd}
T_{X/Z} \ar[r]{}{\gamma_{g \circ f}} \ar[d, swap]{}{\iota_{g \circ f}} & T_{Y/Z} \ar[d]{}{\iota_g} \\
TX \ar[r, swap]{}{Tf} & TY
\end{tikzcd}
\]
commutes. To prove that this square is a limiting square, assume that we have a commuting diagram
\[
\begin{tikzcd}
W \ar[r]{}{\varphi} \ar[d, swap]{}{\psi} & T_{Y/Z} \ar[d]{}{\iota_g} \\
TX \ar[r, swap]{}{Tf} & TY
\end{tikzcd}
\]
and paste on the commuting square below:
\begin{equation}\label{Eqn: The cool pasting thing for zero carrable composition}
\begin{tikzcd}
W \ar[r]{}{\varphi} \ar[d, swap]{}{\psi} & T_{Y/Z} \ar[d]{}{\iota_g} \ar[r]{}{\iota^g} & Z \ar[d]{}{0_Z} \\
TX \ar[r, swap]{}{Tf} & TY \ar[r, swap]{}{Tg} & TZ
\end{tikzcd}
\end{equation}
Thus, from Diagram (\ref{Eqn: The cool pasting thing for zero carrable composition}) gives that there is a unique morphism $\rho:W \to T_{X/Z}$ which renders the diagram
\[
\begin{tikzcd}
W \ar[dr, dashed]{}{\exists!\rho} \ar[drr, bend left = 20]{}{\iota^g \circ \varphi} \ar[ddr, swap, bend right = 20]{}{\psi} & & \\
 & T_{X/Z} \ar[r]{}{\iota^{g \circ f}} \ar[d, swap]{}{\iota_{g \circ f}} & Z \ar[d]{}{0_Z} \\
 & TX \ar[r, swap]{}{T(g \circ f)} & TZ
\end{tikzcd}
\]
commutative. We then compute
\[
\iota_g \circ \varphi = Tf \circ \psi = Tf \circ \iota_{g \circ f} \circ \rho = \iota_g \circ \gamma_{g \circ f} \circ \rho,
\]
so from the fact that $\iota_g$ is monic we derive that $\gamma_{g \circ f} \circ \rho = \varphi$. Thus the diagram
\[
\begin{tikzcd}
W \ar[dr, dashed]{}{\exists!\rho} \ar[drr, bend left = 20]{}{\varphi} \ar[ddr, swap, bend right = 20]{}{\psi} \\
 & T_{X/Z} \ar[r]{}{\gamma_{g \circ f}} \ar[d, swap]{}{\iota_{g \circ f}} & T_{Y/Z} \ar[d]{}{\iota_g} \\
 & TX \ar[r, swap]{}{Tf} & TY
\end{tikzcd}
\]
commutes; additionally, the fact that $\rho$ is unique is immediate from the fact that $\iota_{g \circ f}$ is monic. This indicates that $T_{X/Z}$ is a pullback; that it is a $T$-pullback follows from a straightforward check using the fact that $T_{X/Z}$, $T_{Y/Z}$ are .
\end{proof}
\begin{corollary}
Let $\Cscr$ be a tangent category and let $f:X \to Y$ and $g:Y \to Z$ be morphisms in $\Cscr$. Then $g \circ f$ is $0$-carrable if and only if $g$ is $0$-carrable and $TX \times_{TY} T_{Y/Z}$ exists.
\end{corollary}

We now discuss two pullback-stability results for $0$-carrable morphisms. The first is an immediate result implied by pullback stacking, while the second is a more subtle. Namely, for the more subtle result, we will prove that if we have a pullback diagram
\[
\begin{tikzcd}
X \times_Y Z \ar[r]{}{\pi_1} \ar[d, swap]{}{\pi_0} & Z \ar[d]{}{g} \\
X \ar[r, swap]{}{f} & Y
\end{tikzcd}
\]
which is preserved by $T$ and if both maps $f$ and $g$ are $0$-and-$p$-carrable, then there is an isomorphism $T_{(X \times_Y Z)/Z} \cong T_{X/Y} \times_Y Z$. This gives a conceptual proof of Grothendieck's observation in \cite[Line 16.5.12.2]{EGA44} that for any cospan $X \rightarrow Y \leftarrow Z$ of schemes there is an isomorphism $T_{(X \times_Y Z)/Z} \cong T_{X/Y} \times_Y Z$ (and its corresponding tangent-categorical generalization). Let us first, however, prove the straightforward pullback-stacking result before establishing the more technical result.

\begin{lemma}
Let $\Cscr$ be a tangent category and assume that we have a $T$-pullback square:
\[
\begin{tikzcd}
X \times_Y Z \ar[r]{}{\pi_1} \ar[d, swap]{}{\pi_0} & Z \ar[d]{}{g} \\
X \ar[r, swap]{}{f} & Y
\end{tikzcd}
\]
If $\pi_1$ is both $p$-carrable and $0$-carrable then the diagram below is a $T$-pullback:
\[
\begin{tikzcd}
T_{(X \times_Y Z)/Z} \ar[r]{}{\iota^{\pr_1}} \ar[d, swap]{}{T\pi_0 \circ \iota_{\pi_1}} & Z \ar[d]{}{Tg \circ 0_Z} \\
TX \ar[r, swap]{}{Tf} & TY
\end{tikzcd}
\]
\end{lemma}
\begin{proof}
We simply stack the pullback squares
\[
\begin{tikzcd}
T_{(X \times_Y Z)/Z} \ar[r]{}{\iota_{\pi_1}} \ar[d, swap]{}{\iota^{\pi_1}} & T(X \times_Y Z) \ar[r]{}{T\pi_0} \ar[d]{}{T\pi_1} & TX \ar[d]{}{Tf} \\
Z \ar[r, swap]{}{0_Z} & TZ \ar[r, swap]{}{Tg} & TY
\end{tikzcd}
\]
and apply the Pullback Lemma.
\end{proof}
This lemma immediately allows us to deduce the following stronger result: that $T_{(X \times_Y Z)/Z}$ is a pullback of $Z \rightarrow TZ \leftarrow T(X \times_Y Z)$ if and only if it is a pullback of $Z \rightarrow TY \leftarrow TX$.
\begin{corollary}
Let $\Cscr$ be a tangent category and assume that we have a $T$-pullback square
\[
\begin{tikzcd}
X \times_Y Z \ar[r]{}{\pi_1} \ar[d, swap]{}{\pi_0} & Z \ar[d]{}{g} \\
X \ar[r, swap]{}{f} & Y
\end{tikzcd}
\]
which is preserved by the tangent bundle functor $T$. If $\pi_1$ is $p$-carrable then it is $0$-carrable if and only if there is a $T$-pullback square:
\[
\begin{tikzcd}
P \ar[r]{}{} \ar[d]{}{} & Z \ar[d]{}{Tg \circ 0_Z} \\
TX \ar[r, swap]{}{Tf} & TY
\end{tikzcd}
\]
\end{corollary}

\begin{proposition}\label{Prop: Pullback stability for zero and p carrable maps}
Let $\Cscr$ be a tangent category and assume that we have a $T$-pullback square:
\[
\begin{tikzcd}
X \times_Y Z \ar[r]{}{\pi_1} \ar[d, swap]{}{\pi_0} & Z \ar[d]{}{g} \\
X \ar[r, swap]{}{f} & Y
\end{tikzcd}
\]
If $f$ and $\pi_1$ are both $p$-carrable and $0$-carrable then the unique morphism $\gamma:T_{(X \times_Y Z)/Z} \to T_{X/Y}$ which renders the diagram
\[
\begin{tikzcd}
T_{(X \times_Y Z)/Z} \ar[drr, bend left = 20]{}{g \circ \pr_1} \ar[dr, dashed]{}{\exists!\, \gamma} \ar[ddr, swap, bend right = 20]{}{T\pi_0 \circ \pr_0} \\
 & T_{X/Y} \ar[r]{}{\pr_1} \ar[d, swap]{}{\pr_0} & Z \ar[d]{}{g} \\
 & TX \ar[r, swap]{}{Tf} & TY
\end{tikzcd}
\]
commutative makes
\[
\begin{tikzcd}
T_{(X \times_Y Z)/Z} \ar[r]{}{\pr_1} \ar[d, swap]{}{\gamma} & Z \ar[d]{}{g} \\
T_{X/Y} \ar[r, swap]{}{\iota^f} & Y
\end{tikzcd}
\]
into a $T$-pullback square. In particular, $T_{(X \times_Y Z)/Z} \cong T_{X/Y} \times_Y Z$.
\end{proposition}
\begin{proof}
We prove that the desired square is a pullback by showing that it uniquely factors arbitrary cones over $T_{X/Y} \rightarrow Y \leftarrow Z$ by way of making use of the universal properties of each of $T(X \times_Y Z),$ $T_{X/Y}$, and $T_{(X \times_Y Z)/Z}$.. To this end, assume that we have a commuting diagram
\[
\begin{tikzcd}
W \ar[r]{}{\varphi} \ar[d, swap]{}{\psi} & Z \ar[d]{}{g} \\
T_{X/Y} \ar[r, swap]{}{\iota^f} & Y
\end{tikzcd}
\]
and note that by making use of the universal property of $T_{X/Y}$ and the naturality of the transformation $0$, we can extend it to a commuting diagram of the form:
\[
\begin{tikzcd}
W \ar[r]{}{\varphi} \ar[d, swap]{}{\psi} & Z \ar[d]{}{g} \ar[dr]{}{0_Z} \\
T_{X/Y} \ar[r, swap]{}{\iota^f} \ar[d, swap]{}{\iota_f} & Y \ar[d]{}{0_Y} & TZ \ar[dl]{}{Tg} \\
TX \ar[r, swap]{}{Tf} & TY
\end{tikzcd}
\]
Because the tangent bundle functor $T$ preserves the pullback $X \times_Y Z$ and this diagram is ultimately a cone over $TX \xrightarrow{Tf} TY \xleftarrow{Tg} TZ$, there is a unique morphism $\alpha$ rendering the diagram
\[
\begin{tikzcd}
V \ar[dr, dashed]{}{\exists! \alpha} \ar[drr, bend left = 20]{}{0_Z \circ \varphi} \ar[ddr, swap, bend right = 20]{}{\iota_f \circ \psi} & & \\
 & T(X \times_Y Z) \ar[r]{}{T\pi_1} \ar[d, swap]{}{T\pi_0} & TZ \ar[d]{}{Tg} \\
& TX \ar[r, swap]{}{Tf} & TY
\end{tikzcd}
\]
commutative.

We now claim that the span $T(X \times_Y Z) \xleftarrow{\alpha} V \xrightarrow{\varphi} Z$ forms a cone over $T(X \times_Y Z) \xrightarrow{T\pi_1} TZ \xleftarrow{0_Z} Z$ and hence factors through $T_{(X \times_Y Z)/Z}$. However, this is an immediate consequence of the definition of $\alpha$ (and in fact is encoded in the top-most commuting triangle of the diagram directly above). Thus there is a unique morphism $\rho:V \to T_{(X \times_Y Z)/Z}$ which renders
\[
\begin{tikzcd}
V \ar[dr, dashed]{}{\exists! \rho} \ar[drr, bend left = 20]{}{\varphi} \ar[ddr, swap, bend right = 20]{}{\alpha} & & \\
 & T_{(X \times_Y Z)/Z} \ar[r]{}{\iota^{\pi_1}} \ar[d, swap]{}{\iota_{\pi_1}} & Z \ar[d]{}{0_Z} \\
 & T(X \times_Y Z) \ar[r, swap]{}{T\pi_1} & TZ
\end{tikzcd}
\]
commutative. 

We now claim that $\rho:V \to T_{(X \times_Y Z)/Z}$ produces a factorization
\[
\begin{tikzcd}
V \ar[drr, bend left = 20]{}{\varphi} \ar[ddr, swap, bend right = 20]{}{\psi} \ar[dr]{}{\rho} & & \\
 & T_{(X \times_Y Z)/Z} \ar[d, swap]{}{\gamma} \ar[r]{}{\iota^{\pi_1}} & Z \ar[d]{}{g} \\
 & T_{X/Y} \ar[r, swap]{}{\iota^f} & Y
\end{tikzcd}
\]
To see this note that by the definiton and construction of $\rho,$ the identity $\iota^{\pi_1} \circ \rho = \varphi$ holds. As such, it suffices to show that $\gamma \circ \rho = \psi$. Now observe that since $\iota_{\pi_1} \circ \gamma = T\pi_0 \circ \pr_0$ holds, we can derive from the equalities
\[
\iota_{\pi_1} \circ \gamma \circ \rho = T\pi_0 \circ \pr_0 \circ \rho = T\pi_0 \circ \alpha = \iota_{\pi_1} \circ \psi
\]
and the fact that $\iota_{\pi_1}:T_{X/Y} \to TX$ is monic, $\gamma \circ \rho = \psi$ holds. Thus the diagram
\[
\begin{tikzcd}
V \ar[drr, bend left = 20]{}{\varphi} \ar[ddr, swap, bend right = 20]{}{\psi} \ar[dr]{}{\rho} & & \\
 & T_{(X \times_Y Z)/Z} \ar[d, swap]{}{\gamma} \ar[r]{}{\pr_1} & Z \ar[d]{}{g} \\
 & T_{X/Y} \ar[r, swap]{}{\iota^f} & Y
\end{tikzcd}
\]
commutes. Furthermore, we see that $\rho$ is also necessarily unique, as if there are any two maps making
\[
\begin{tikzcd}
V \ar[drr, bend left = 20]{}{\varphi} \ar[ddr, swap, bend right = 20]{}{\psi} \ar[dr, shift left = 1]{}{\rho} \ar[dr, shift right = 1, swap]{}{\sigma} & & \\
 & T_{(X \times_Y Z)/Z} \ar[d, swap]{}{\gamma} \ar[r]{}{\iota^{\pi_1}} & Z \ar[d]{}{g} \\
 & T_{X/Y} \ar[r, swap]{}{\iota^f} & Y
\end{tikzcd}
\]
commute, then they also make the diagram
\[
\begin{tikzcd}
V \ar[drr, bend left = 20]{}{\varphi} \ar[dddr, swap, bend right = 20]{}{\iota_{\pi_1} \circ \psi} \ar[dr, shift left = 1]{}{\rho} \ar[dr, shift right = 1, swap]{}{\sigma} & & \\
 & T_{(X \times_Y Z)/Z} \ar[d, swap]{}{\gamma} \ar[r]{}{\iota^{\pi_1}} & Z \ar[d]{}{0_Z} \\
 & T(X \times_Y Z) \ar[r]{}{T\pi_1} \ar[d, swap]{}{T\pi_0} & TZ \ar[d]{}{Tg} \\
 & TX \ar[r, swap]{}{Tf} & TY
\end{tikzcd}
\]
commute. Thus, since $T_{(X \times_Y Z)/Z}$ is a pullback, $\rho = \sigma$ and so
\[
\begin{tikzcd}
T_{(X \times_Y Z)/Z} \ar[r]{}{\iota^{\pi_1}} \ar[d, swap]{}{\gamma} & Z \ar[d]{}{g} \\
T_{X/Y} \ar[r, swap]{}{\iota_{f}} & Y
\end{tikzcd}
\]
is a pullback square.
\end{proof}

Before moving on, it is worth discussing some of the structure maps that the relative tangent bundle $T_{X/Y}$ carries, in particular its differential bundle structure. So for any map $f:X \to Y$ (seen as an object in the slice category), the projection $p_{X/Y}:T_{X/Y} \to X$ is defined via
\[
\begin{tikzcd}
T_{X/Y} \ar[r]{}{\iota_f} \ar[dr, swap]{}{p_{X/Y}} & TX \ar[d]{}{p_X} \\
 & X
\end{tikzcd}
\]
which induces a differential bundle $p_{X/Y}: T_{X/Y} \to X$ in $\Cscr$. The zero $0_{X/Y}: X \to T_{X/Y}$ is defined as above. The addition $\operatorname{add}_{X/Y}: (T_{X/Y})_2 \to T_{X/Y}$ is relatively straightforward to define by pulling back against two copies of $\iota_f:T_{X/Y} \to TX$ and using the addition on $T_2X$. To define the lift $\lambda_{X/Y}: T_{X/Y} \to T(T_{X/Y})$, note that by the assumption of $f$ being $0$-carrable, the diagram
\[
\begin{tikzcd}
T(T_{X/Y}) \ar[r]{}{T(\iota^f)} \ar[d, swap]{}{T\iota_f} & TY \ar[d]{}{(T \ast 0)_Y} \\
T^2X \ar[r, swap]{}{T^2f} & T^2Y
\end{tikzcd}
\]
is a pullback square. But then we induce the commuting cube
\begin{equation}\label{Eqn: Diagram defining relative lift}
\begin{tikzcd}
 & T_{X/Y} \ar[dl, swap, dashed]{}{\exists!\lambda_{X/Y}} \ar[dd, near start]{}{\iota_f} \ar[rr]{}{\iota^f} & & Y \ar[dl]{}{0_Y} \ar[dd]{}{0_Y} \\
T(T_{X/Y}) \ar[rr, near start, crossing over]{}{T(\iota^f)} & & TY \\
& TX \ar[dl, swap]{}{\ell_X} \ar[rr, near start]{}{Tf} & & TY \ar[dl]{}{\ell_Y} \\
T^2X \ar[rr, swap]{}{T^2f} & & T^2Y
\ar[from = 2-3, to = 4-3, crossing over, near start]{}{(T \ast 0)_Y}
\ar[from = 2-1, to = 4-1, crossing over, swap]{}{T\iota_f}
\end{tikzcd}
\end{equation}
with $\lambda_{X/Y}$ given by the universal property carried by $T(T_{X/Y})$. We also have a flip $c_{X/Y}$ on the relative tangent structure, which we obtain by following roughly the same process as defining the lift $\ell_{X/Y}$. Once again we observe that upon applying $T$ to the diagram exhibiting the universal property of $T_{X/Y}$ and appending the naturality square for $c_Y$ we obtain the commuting diagram:
\[
\begin{tikzcd}
T(T_{X/Y}) \ar[r]{}{T(\iota^f)} \ar[d, swap]{}{T\iota_f} & TY \ar[d]{}{(T \ast 0)_Y} \ar[r, equals] & T^2Y \ar[d]{}{(0 \ast T)_Y} \\
T^2X \ar[r, swap]{}{T^2f} & T^2Y \ar[r, swap]{}{c_Y} & T^2Y
\end{tikzcd}
\]
Note that both squares are pullback squares, so the overall diagram is a pullback square as well. Thus
\[
\begin{tikzcd}
T(T_{X/Y}) \ar[r]{}{T\pi_1} \ar[d, swap]{}{T\pi_0} & TY \ar[d]{}{(0 \ast T)_Y} \\
T^2X \ar[r, swap]{}{c_Y \circ T^2f} & T^2Y
\end{tikzcd}
\]
is a pullback. This implies that, as $c_Y \circ T^2f$ and $T^2f$ differ only by an application of $c$, there is a pullback square
\[
\begin{tikzcd}
T^2_{X/Y} \ar[r]{}{\iota^{Tf}} \ar[d, swap]{}{\iota_{Tf}} & TY \ar[d]{}{(0 \ast T)_Y} \\
T^2X \ar[r, swap]{}{T^2f} & T^2Y
\end{tikzcd}
\]
in $\Cscr$. This incidentally shows that $Tf$ is $0$-carrable. Moreover, this also shows that there is a unique map $c_{X/Y}$ which renders the diagram
\[
\begin{tikzcd}
 & T^2_{X/Y} \ar[rr]{}{\iota^{Tf}} \ar[dd, near end]{}{\iota_{Tf}} & & TY \ar[dd]{}{(0 \ast T)_Y} \\
T(T_{X/Y})  \ar[dd, swap]{}{T\iota_{f}} \ar[ur, dashed]{}{\exists!\,c_{X/Y}} \ar[rr, crossing over, near start]{}{T(\iota^f)} & & TY \ar[ur, equals]   \\
 & T^2X \ar[rr, swap, near start]{}{T^2f} & & T^2Y \\
T^2X \ar[ur]{}{c_X} \ar[rr, swap]{}{T^2f} & & T^2Y \ar[ur, swap]{}{c_Y}
\ar[from = 2-3, to = 4-3, crossing over, near start]{}{(T \ast 0)_Y}
\end{tikzcd}
\]
commutative. Putting these observations together allows us to deduce the proposition below.

\begin{proposition} In a tangent category $\Cscr$, if $f:X \to Y$ is $0$-carrable then $Tf$ is $0$-carrable. Furthermore, $T(T_{X/Y})$ may be used as a model for the pullback of $T^2f$ against $(0 \ast T)_Y$.
\end{proposition}

Moreover, putting this together with the fact that tangent bundle functors reflect limits (Corollary \ref{Cor: tangent bundle functor reflects limits}), we obtain the following characterization.

\begin{corollary}
In a tangent category $\Cscr$, a map $f$ is $0$-carrable if and only if $Tf$ is $0$-carrable.
\end{corollary}

\subsection{The Relative Cotangent Sequence}\label{Subsection: Relative Cotangent sequence}
We now prove the most important result in this section. The theorem below is somewhat innocuous, as on one hand it simply shows that it is not a coincidence that in each of our main examples, $T_{X/Y}$ had the flavour of a kernel of the horizontal descent $\theta_f$ (in fact, in Example \ref{Example: Relative Tangent bundle in SMan} we literally defined $T_{X/Y}$ as the kernel of $\theta_f$ in $\DBun(X)$). However, on the other hand what is remarkable about the following result is that it gives us a way of phrasing the relative cotangent sequence in tangent-categorical terms, allowing us to study said cotangent sequence in any tangent category. 

\begin{theorem}\label{Thm: ZeroCarrable makes relative bundle a kernel in DBun}
Let $f:X \to Y$ be a $0$-carrable and $p$-carrable morphism in a tangent category $\Cscr$. Then the map $\iota_f:T_{X/Y} \to TX$ makes the diagram
\[
\begin{tikzcd}
T_{X/Y} \ar[r]{}{\iota_f} & TX \ar[rr, shift left = 1]{}{\theta_f} \ar[rr, shift right = 1, swap]{}{f^{\ast}(0_Y) \circ p_X} & & f^{\ast}(TY)
\end{tikzcd}
\]
into an equalizer diagram in $\DBun(X)$.
\end{theorem}
\begin{proof}
We first check that $\iota_f$ does indeed equalize $\theta_f$ and $f^{\ast}(0_Y) \circ p_X$. To see this note that since $f^{\ast}(TY)$ is a pullback, it suffices to show that, for pullback projections
\[
\begin{tikzcd}
f^{\ast}(TY) \ar[r]{}{\pr_1} \ar[d, swap]{}{\pr_0} & TY \ar[d]{}{p_Y} \\
X \ar[r, swap]{}{f} & Y
\end{tikzcd}
\]
the identities $\pr_0 \circ \theta_f \circ \iota_f = \pr_0 \circ f^{\ast}(0_Y) \circ p_X \circ \iota_f$ and $\pr_1 \circ \theta_f \circ \iota_f = \pr_1 \circ f^{\ast}(0_Y) \circ p_X \circ \iota_f$ hold. To do so, first observe that $\iota_f:T_{X/Y} \to TX$ is a linear morphism of differential bundles, which implies that the diagrams
\[
\begin{tikzcd}
T_{X/Y} \ar[r]{}{\iota_f} \ar[dr, swap]{}{\iota_f \circ p_X} & TX \ar[r]{}{\theta_f} \ar[d]{}[description]{p_X} & f^{\ast}(TY) \ar[dl]{}{\pr_0} \\
 & X
\end{tikzcd}\quad
\begin{tikzcd}
T_{X/Y} \ar[r]{}{\iota_f} & TX \ar[r]{}{\theta_f} & f^{\ast}(TY) \\
 & X \ar[ul]{}{0_{X/Y}} \ar[u]{}[description]{0_X} \ar[ur, swap]{}{f^{\ast}(0_Y)}
\end{tikzcd}
\]
commute. The diagram involving only projections implies that $\pr_0 \circ \theta_f \circ \iota_f = p_X \circ \iota_f$ while the fact that the differential bundle $f^{\ast}(TY)$ has projection $\pr_0$ and zero $f^{\ast}(0_Y)$ implies that $\pr_0 \circ f^{\ast}(0_Y) \circ p_X \circ \iota_f = p_X \circ \iota_f$. So $\pr_0 \circ \theta_f \circ \iota_f = \pr_0 \circ f^{\ast}(0_Y) \circ p_X \circ \iota_f$. In a similar vein, we also get that $\pr_1 \circ \theta_f \circ \iota_f = Tf \circ \iota_f = 0_Y \circ \iota^f$ and
\[
\pr_1 \circ f^{\ast}(0_Y) \circ p_X \circ \iota_f = 0_Y \circ f \circ p_X \circ \iota_f = 0_Y \circ p_Y \circ Tf \circ \iota_f = 0_Y \circ p_Y \circ 0_Y \circ \iota^f = 0_Y \circ \iota^f.
\]
Thus $\pr_1 \circ \theta_f \circ \iota_f = \pr_1 \circ f^{\ast}(0_Y) \circ p_X \circ \iota_f$. Therefore, by the universal property of $f^{\ast}(TY)$, it follows that $\theta_f \circ \iota_f = f^{\ast}(0_Y) \circ p_X \circ \iota_f$ and so $\iota_f$ equalizes $\theta_f$ and $f^{\ast}(0_Y) \circ p_X$.

To see that $\iota_f:T_{X/Y} \to TX$ is the equalizer claimed, let $q: Z \to X$ be a differential bundle, and let $k: q \to p_X$ be a a linear differential bundle morphism which equalizes $\theta_f$ and $f^{\ast}(0_Y) \circ p_X$, so $k: Z \to TX$. Using that $k$ commutes with bundle projections, we have that 
\begin{align*}
Tf \circ k &= \iota^f \circ \theta_f \circ k = \iota^f \circ f^{\ast}(0_Y) \circ p_X \circ k = \iota^f \circ f^{\ast}(0_Y) \circ q = 0_Y \circ f \circ q
\end{align*}
This gives a unique map $\rho: Z \to T_{X/Y}$ which makes the following diagram commute: 
\[
\begin{tikzcd}
Z \ar[drr, bend left = 20]{}{k} \ar[dr, dashed]{}{\exists!\rho} \ar[ddr, swap, bend right = 20]{}{f \circ q} \\
 & T_{X/Y} \ar[d, swap]{}{\iota^f} \ar[r]{}{\iota_f} & TX \ar[d]{}{Tf} \\
 & Y \ar[r, swap]{}{0_Y} & TY
\end{tikzcd}
\]

We now need to verify that $\rho$ is a linear differential bundle morphism. To see this we first note that the cell factoring $k$ shows that the diagram
\[
\begin{tikzcd}
Z \ar[r]{}{\rho} \ar[rr, bend left = 30]{}{k} \ar[d, swap]{}{q}  & T_{X/Y} \ar[r]{}{\iota_f} & TX \ar[d]{}{p_X} \\
X \ar[rr, equals] & & X
\end{tikzcd}
\]
commutes. However, the upper edge of the rectangle satisfies $q = p_X \circ k = p_X \circ \iota_f \circ \rho = p_{X/Y} \circ \rho$ so $(\rho, \id_X)$ is a bundle map. To see that it is linear it suffices to prove that the diagram
\[
\begin{tikzcd}
Z \ar[d, swap]{}{\lambda} \ar[r]{}{\rho} & T_{X/Y} \ar[d]{}{\lambda_{X/Y}} \\
TZ \ar[r, swap]{}{T\rho} & T(T_{X/Y})
\end{tikzcd}
\]
commutes; for the construction of $\lambda_{X/Y}$, see Diagram \ref{Eqn: Diagram defining relative lift}. However, as $T^2X \xleftarrow{T\iota_f} T(T_{X/Y}) \xrightarrow{T\iota^f} TY$ is the pullback of $T^2X \xrightarrow{T^2f} T^2Y \xleftarrow{(T \ast 0)_Y} TY$, it carries the universal property that $T\rho$ is the unique map which renders the diagram
\[
\begin{tikzcd}
TZ \ar[drr, bend left = 20]{}{Tk} \ar[dr, dashed]{}{\exists!T\rho} \ar[ddr, swap, bend right = 20]{}{Tf \circ Tq} \\
 & T(T_{X/Y}) \ar[d, swap]{}{T\iota^f} \ar[r]{}{T\iota_f} & T^2X \ar[d]{}{T^2f} \\
 & TY \ar[r, swap]{}{(T \ast 0)_Y} & T^2Y
\end{tikzcd}
\]
commutative. Note also that the diagram
\[
\begin{tikzcd}
Z \ar[rr]{}{Tf \circ Tq \circ \lambda} \ar[d, swap]{}{Tk \circ \lambda} & & TY \ar[d]{}{(T \ast 0)_Y} \\
T^2X \ar[rr, swap]{}{T^2f} & & T^2Y 
\end{tikzcd}
\]
commutes with the outer edges composing into a way to witness $Z$ as a cone over $T^2X \xrightarrow{T^2f} T^2Y \xleftarrow{(T \ast 0)_Y} TY$. As such, if we can prove now that both $T\rho \circ \lambda$ and $\lambda_{X/Y} \circ \rho$ factor the diagram above through $T(T_{X/Y})$, we will be finished with verifying the linearity of $(\rho,\id_X)$ by way of the universal property of the pullback.  

We first show that $T\rho \circ \lambda$ factors the above square through $T(T_{X/Y})$. To this end we compute now that
\[
T\iota_f \circ T\rho \circ \lambda  = T(\iota_f \circ \rho) \circ \lambda = Tk \circ \lambda
\]
while
\[
T\iota^f \circ T\rho \circ \lambda = T(\iota^f \circ \rho) \circ \lambda = T(f \circ q) \circ \lambda = Tf \circ Tq \circ \lambda.
\]
Thus the diagram
\[
\begin{tikzcd}
Z \ar[dr]{}{T\rho \circ \lambda} \ar[drr, bend left = 20]{}{Tq \circ Tf \circ \lambda} \ar[ddr, swap, bend right = 20]{}{Tk \circ \lambda} \\
 & T(T_{X/Y}) \ar[r]{}{T\iota^f} \ar[d, swap]{}{T\iota_f} & TY \ar[d]{}{(T \ast 0)_{Y}} \\
 & T^2X \ar[r, swap]{}{T^2f} & T^2Y
\end{tikzcd}
\]
commutes. Next we show that $T^2f \circ Tk \circ \lambda = (T \ast 0)_Y \circ Tf \circ Tq \circ \lambda$ factors through $\lambda_{X/Y} \circ \rho$. To this end we first compute that
\[
T\iota_f \circ \lambda_{X/Y} \circ \rho = \ell_X \circ \iota_f \circ \rho = \ell_X \circ k = Tk \circ \lambda
\]
by the linearity of $k$. Similarly,
\begin{align*}
T\iota^f \circ \lambda_{X/Y} \circ \rho &= 0_Y \circ \iota^f \circ \rho = 0_Y \circ f \circ q = q \circ 0_X \circ Tf = Tf \circ Tq \circ \lambda;
\end{align*}
note that the last identity, namely $q \circ 0_X = Tq \circ \lambda$, used that $(\lambda, 0_X)$ is a morphism of additive bundles in $\Cscr$. However, this also proves that the diagram
\[
\begin{tikzcd}
Z \ar[dr]{}{\lambda_{X/Y} \circ \rho} \ar[drr, bend left = 20]{}{Tq \circ Tf \circ \lambda} \ar[ddr, swap, bend right = 20]{}{Tk \circ \lambda} \\
 & T(T_{X/Y}) \ar[r]{}{T\iota^f} \ar[d, swap]{}{T\iota_f} & TY \ar[d]{}{(T \ast 0)_{Y}} \\
 & T^2X \ar[r, swap]{}{T^2f} & T^2Y
\end{tikzcd}
\]
commutes. However, as $T(T_{X/Y})$ is a limiting cone, this in turn implies that $T\rho \circ \lambda = \lambda_{X/Y} \circ \rho$. So we get that $(\rho,\id_X)$ is a linear morphism of differential bundles.

Lastly, since $\iota_f: T_{X/Y} \to TX$ is monic (since recall it is the pullback against the monic $0_Y$), we know that  $\rho$ is the unique linear differential bundle morphism which renders the diagram
\[
\begin{tikzcd}
T_{X/Y} \ar[r]{}{\iota_f} & TX \ar[rr, shift left = 1]{}{\theta_f} \ar[rr, shift right = 1, swap]{}{f^{\ast}(0_Y) \circ p_X} & & f^{\ast}(TY) \\
Z \ar[ur, swap]{}{k} \ar[u, dashed]{}{\exists\,\rho}
\end{tikzcd}
\]
commutative. As such, we finally conclude that $T_{X/Y} \xrightarrow{\iota_f} TX$ is the equalizer of $\theta_f$ and $f^{\ast}(0_Y) \circ p_X$ in $\DBun(X)$ as desired. 
\end{proof}

We may now properly define the notion of a relative cotangent sequence in a tangent category. 

\begin{definition}\label{Defn: Relative Cotangent Sequence}
Let $f:X \to Y$ be a $0$-carrable and $p$-carrable morphism in a tangent category $\Cscr$. The \emph{relative cotangent sequence for $f$} is the diagram
\[
\begin{tikzcd}
X \ar[r]{}{0_{X/Y}} & T_{X/Y} \ar[r]{}{\iota_f} & TX \ar[r]{}{\theta_f} & f^{\ast}(TY)
\end{tikzcd}
\]
\end{definition}

As products in $\DBun(X)$ are biproducts by Proposition \ref{Prop: STructure of DBun} and $\DBun(X)$ is $\mathbf{CMon}$-enriched, the sequence which appears in Definition \ref{Defn: Relative Cotangent Sequence} is the differential bundle analogue of an exact sequence. Let us explore this idea further. Because the trivial differential bundle over $X$ (i.e. the identity $\id_X$) is the zero object in $\DBun(X)$, the diagram
\[
\begin{tikzcd}
X \ar[r] & T_{X/Y} \ar[r]{}{\iota_f} & TX \ar[r]{}{\theta_f} & f^{\ast}(TY)
\end{tikzcd}
\]
is exact in the sense that it exhibits $T_{X/Y}$ as the kernel of $\theta_f$ as, by Theorem \ref{Thm: ZeroCarrable makes relative bundle a kernel in DBun}, the diagram
\[
\begin{tikzcd}
T_{X/Y} \ar[r]{}{\iota_f} \ar[d, swap]{}{p_{X/Y}} & TX \ar[d]{}{\theta_f} \\
X \ar[r, swap]{}{f^{\ast}(0_Y)} & f^{\ast}(TY)
\end{tikzcd}
\]
is a $T$-pullback and hence a kernel preserved by $T^m$ for all $m \in \Nbb$. It also recovers the relative cotangent sequences of smooth manifolds and of schemes in the following sense.

\begin{example}\label{Example: Relative Cotangent Sequence for SMan}
In $\SMan$, for any smooth morphism $f:X \to Y$, the relative cotangent sequence of $f$ is the sequence
\[
\begin{tikzcd}
X \ar[r] & T_{X/Y} \ar[r] & TX \ar[rr]{}{\langle p_X, Tf \rangle} & & f^{\ast}(TY)
\end{tikzcd}
\]
in $\DBun(X)$. Along stalks $x \in X$ the sequence above induces the exact sequence
\[
\begin{tikzcd}
0 \ar[r] & \operatorname{Ker}\big(D[f](x)\big) \ar[r]{}{\operatorname{incl}} & T_xX \ar[rr]{}{D[f](x)} & & T_{f(x)}Y
\end{tikzcd}
\]
of real vector spaces.
\end{example}
\begin{example}\label{EXample: Relative Cotangent Sequence for Schemes}
For a morphism of schemes $f:X \to Y$, the relative cotangent sequence for $f$ in $\DBun(X)$ (in the sense of Definition \ref{Defn: Relative Cotangent Sequence}) is determined by applying the relative spectrum $\underline{\Spec}_X$ of the relative symmetric algebra functor $\underline{\Sym}_{\Ocal_X}$ to the cotangent sequence:
\[
\begin{tikzcd}
f^{\ast}\Kah{Y}{S} \ar[r] & \Kah{X}{S} \ar[r] & \Kah{X}{Y} \ar[r] & 0
\end{tikzcd}
\]
Because the functor $\underline{\Spec}_{X} \circ \underline{\Sym}_{\Ocal_X}:\QCoh(X)^{\op} \to \DBun(X)$ is an equivalence of categories by \cite[Theorem 4.28]{GeoffJSDiffBunComAlg} we see that the relative cotangent sequence of \cite{EGA44} coincides with the relative cotangent sequence of Definition \ref{Defn: Relative Cotangent Sequence}.
\end{example}

\subsection{Unramified Morphisms}
There is a special situation which is worth our focus. We have seen already that the map $0_{X/Y}:X \to T_{X/Y}$ is the zero section for the relative tangent bundle of $X$ over $Y$. A natural question is then to ask when the map $0_{X/Y}$ is an isomorphism. In the case of $\Sch_{/S}$ this is equivalent to asking:
\[\underline{\Spec}_X\left(\underline{\Sym}_{\Ocal_X}(0)\right) \cong X \cong T_{X/Y} = \underline{\Spec}_{X}\left(\underline{\Sym}_{\Ocal_X}\left(\Kah{X}{Y}\right)\right),\] 
which happens if and only if $\Kah{X}{Y} \cong 0$ (and similarly for $\CAlg{R}^{\op}$ for a commutative rig $R$). In the case of $\SMan,$ this is equivalent to asking that the diagram
\[
\begin{tikzcd}
X \ar[r]{}{f} \ar[d, swap]{}{0_X} & Y \ar[d]{}{0_Y} \\
TX \ar[r, swap]{}{Tf} & TY
\end{tikzcd}
\]
is a pullback. By making use of Example \ref{Example: Relative Tangent bundle in SMan} this amounts to asking that for all $x \in X$, $\operatorname{Ker}\big(D[f](x)\big) \cong 0.$ In turn, this indicates that there is no ramificiation within the map $f$ in the sense that nonzero tangent vectors get sent to zero through $Tf$. We make this a formal definition and briefly study these maps below.

\begin{definition}\label{Defn: TUnramified Morphisms}
In a tangent category $\Cscr$, a morphism $f:X \to Y$is \emph{$T$-unramified} if the diagram
\[
\begin{tikzcd}
    X \ar[r]{}{f} \ar[d, swap]{}{0_X} & Y \ar[d]{}{0_Y} \\
    TX \ar[r, swap]{}{Tf} & TY
\end{tikzcd}
\]
is a $T$-pullback.
\end{definition}

Immediate from the definition is the following characterization of $T$-unramified morphisms as the $0$-carrable maps which have trivial relative tangent bundle.
\begin{lemma}
If $f$ is $T$-unramified then $T^mf$ is $T$-unramified for all $m \in \N$.
\end{lemma}
\begin{proof}
As the $0$-naturality square for $f$
\[
\begin{tikzcd}
X \ar[r]{}{f} \ar[d, swap]{}{0_X} & Y \ar[d, swap]{}{0_Y} \\
TX \ar[r, swap]{}{Tf} & TY
\end{tikzcd}
\]
is a $T$-pullback, for any $m \in \Nbb$ the diagram
\[
\begin{tikzcd}
T^mX \ar[rr]{}{T^mf} \ar[d, swap]{}{((T^m \ast 0)_X} & & T^mY \ar[d]{}{(T^m \ast 0)_Y} \\
T^{m+1}X \ar[rr, swap]{}{T^{m+1}f} & & T^{m+1}Y
\end{tikzcd}
\]
is a $T$-pullback as well. After whiskering by some number of natural transformations of the form $T^{k} \ast c \ast T^{\ell}$ for indices satisfying $k + \ell = m-1$, we arrive at the fact that the square above is naturally isomorphic to the square:
\[
\begin{tikzcd}
T^mX \ar[rr]{}{T^mf} \ar[d, swap]{}{(0 \ast T^m)_X} & & T^mY \ar[d]{}{(0 \ast T^m)_Y} \\
T^{m+1}X \ar[rr, swap]{}{T^{m+1}f} & & T^{m+1}Y
\end{tikzcd}
\]
Because being a $T$-pullback is stable under isomorphism, it follows that $T^mf$ is $T$-unramified.
\end{proof}
\begin{lemma}\label{Lemma: Tunramified is zerocarrable}
In a tangent category $\Cscr$, any $T$-unramified morphism is $0$-carrable and $T_{X/Y} \cong X$. 
\end{lemma}
\begin{proof}
Because the naturality square of $0$ is a pullback, this says that $f$ is $0$-carrable and $T_{X/Y} \cong X$.
\end{proof}

As mentioned at the start of this subsection, we will see in Theorem \ref{Thm: If C Rosicky then immersion iff unrmaified}, in a Roscisk{\'y} tangent category the converse of Lemma \ref{Lemma: Tunramified is zerocarrable} is true. That is, in a Roscisk{\'y} tangent category, a $p$-carrable morphism $f$ is $T$-unramified if and only if it is $0$-carrable. For now, we simply use Theorem \ref{Thm: If C Rosicky then immersion iff unrmaified} as the tool which allow us to compute the $T$-unramified morphisms in our main classes of examples. 

\begin{example}\label{Example: Tunramified in SMan}
In $\SMan$ a $T$-unramified morphism is an immersion. This can be seen from Theorem \ref{Thm: Rosckiy Timmersion iff Tunramified} below and the fact that $\SMan$ is a Rosick{\'y} tangent category.
\end{example}

\begin{example}\label{Example: Tunramified in CAlgR}
Let $R$ be a commutative rig. In $\CAlg{R}$ a map $f:A \to B$ is $T$-unramified if and only if $\operatorname{Ker}(f) = 0$. To see this recall from Example \ref{Example: Relative tangent bundle in CAlg} that given a map $f:A \to B$ the relative tangent bundle $T_{A/B} \cong A \ltimes \operatorname{Ker}(f)$ with natural map $0_{B/A}:A \to A \ltimes \operatorname{Ker}(f)$ given by $a \mapsto (a,0)$. This map is an isomorphism if and only if $\operatorname{Ker}(f) = \lbrace a \; | \; f(a) = 0\rbrace = 0$.
\end{example}

\begin{example}\label{Example: Tunramified in CAlgRop}
Let $R$ be a commutative rig. Then in $\CAlg{R}^{\op}$, the $T$-unramified morphisms correspond precisely to the $R$-algebra morphisms $f: A \to B$ which are formally unramified, that is, those for which $\Kah{B}{A} \cong 0$ \cite[D{\'e}finition 19.10.2 \& Proposition 20.7.4]{EGA04}. This follows from the opposite equivalence $\DBun(B) \simeq \Mod{B}^{\op}$ for all commutative $R$-algebras $A$ and the fact that the relative tangent bundle $T_{B/A}$ of $f^{\op}:B \to A$ is represented by the rig $\Sym_{B}(\Kah{B}{A})$.
\end{example}

\begin{example}\label{Example: Tunramified in Sch}
Let $S$ be a scheme. In $\Sch_{/S}$, the $T$-unramified morphisms are precisely the scheme morphisms $f: X \to Y$ which are formally unramified, i.e., those for which $\Kah{X}{Y} \cong 0$. This follows from Theorem \ref{Thm: Rosckiy Timmersion iff Tunramified} below and the fact that $\Sch_{/S}$ is a Roscisk{\'y} tangent category. 
\end{example}

\begin{example}\label{Example: Tunramified in CDC} In a CDC, a map $f: X \to Y$ is $T$-unramified if and only if the diagram
\[
\begin{tikzcd}
X \ar[d, swap]{}{\iota_0} \ar[r] & 0 \ar[d]{}{0} \\
X \times X \ar[r, swap]{}{D[f]} & Y
\end{tikzcd}
\]
is a pullback square, where $0$ is the terminal object and $0: 0 \to Y$ is the zero map (since homsets in CDCs are commutative monoids). This follows from translating between the tangent bundle functor $T$ to the differential combinator $D$. Thus saying $f$ is T-unramified essentially says that $D[f]$ has trivial kernel in its second argument, or in other words, has trivial kernel in the slice category. In particular, if $f$ is linear (that is, if $D[f] = f \circ \pi_1$) then $f$ is $T$-unramified if and only if the diagram
\[
\begin{tikzcd}
0 \ar[r, equals] \ar[d, swap]{}{0} & 0 \ar[d]{}{0} \\
X \ar[r, swap]{}{f} & Y
\end{tikzcd}
\]
is a pullback square, or in other words, if $\operatorname{Ker}(f) \cong 0$. 
\end{example}

We conclude this section by showing that $T$-unramified maps compose, are stable under application of (powers of) $T$, and are also stable under pullback.

\begin{proposition}\label{Prop: Compositions of zero carrable}
Let $f:X \to Y$ and $g:Y \to Z$ be morphisms in a tangent category $\Cscr$. Then:
 \begin{enumerate}[{\em (i)}] 
    \item If $g:Y \to Z$ is $T$-unramified then $f$ is $0$-carrable if and only if $g \circ f$ is $0$-carrable. In this case, $T_{X/Y} \cong T_{X/Z}$.
    \item If $g:Y \to Z$ is $T$-unramified then $g \circ f$ is $T$-unramified if and only if $f$ is $T$-unramified.
\end{enumerate}
\end{proposition}
\begin{proof}
We first prove Statement (i). $\implies$: Assume that $f$ is $0$-carrable and consider the diagram:
\[
\begin{tikzcd}
T_{X/Y} \ar[r]{}{\iota^f} \ar[d, swap]{}{\iota_f} & Y \ar[r]{}{g} \ar[d]{}{0_Y} & Z \ar[d]{}{0_Z} \\
TX \ar[r, swap]{}{Tf} & TY \ar[r, swap]{}{Tg} & TZ
\end{tikzcd}
\]
Since $f$ is $0$-carrable and $g$ is $T$-unramified, both squares are $T$-pullbacks. Applying the Pullback Lemma implies that the outer square 
\[
\begin{tikzcd}
T_{X/Y} \ar[rr]{}{g \circ \iota^f} \ar[d, swap]{}{\iota_f} & & Z \ar[d]{0_Z} \\
TX \ar[rr, swap]{}{Tg \circ Tf} & & TZ
\end{tikzcd}
\]
is a $T$-pullback as well. Thus $T_{X/Y} \cong T_{X/Z}$ and $g \circ f$ is $0$-carrable.

$\impliedby:$ Assume that $g \circ f$ is $0$-carrable and consider the diagram:
\[
\begin{tikzcd}
T_{X/Z} \ar[r]{}{\iota^{g \circ f}} \ar[d, swap]{}{\iota_{g \circ f}} & Z \ar[d]{}{0_Z} \\
TX \ar[r, swap]{}{Tg \circ Tf} & TZ
\end{tikzcd}
\]
Because, by \cite[Pages 4, 5]{Rosicky}, $T_{X/Z}$ is the equalizer of $Tg \circ Tf$ and the map
\[
0_Z \circ p_Z \circ Tg \circ Tf = Tg \circ 0_Y \circ p_Y \circ Tf = Tg \circ Tf \circ 0_X \circ p_X,
\]
and because $Y \cong T_{Y/Z}$ is the equalizer of $Tg$ and $0_Z \circ p_Z \circ Tg = Tg \circ 0_Y \circ p_Y$ (also by \cite[Pages 4, 5]{Rosicky}), the diagram above admits a factorization:
\[
\begin{tikzcd}
T_{X/Z} \ar[rrr]{}{\iota^{g \circ f}} \ar[dd, swap]{}{\iota_{g \circ f}} \ar[drr, swap]{}{f \circ p_{X/Y}} & & & Z \ar[dd]{}{0_Z} \\
 & & Y \ar[ur, swap]{}{g} \\
TX \ar[drr, swap]{}{Tf} \ar[rrr,  near start]{}{Tg \circ Tf} & & & TZ \\
 & & TY \ar[ur, swap]{}{Tg} 
 \ar[from = 2-3, to = 4-3, crossing over, near start]{}{0_Y}
\end{tikzcd}
\]
However, as the back face of the triangular prism is a $T$-pullback and the rightmost face is a $T$-pullback, we may apply the Pullback Lemma to deduce that the leftmost square is a $T$-pullback as well. This shows exactly that $f$ is $0$-carrable with corresponding limit cone:
\[
\begin{tikzcd}
T_{X/Z} \ar[r]{}{f \circ p_{X/Y}} \ar[d, swap]{}{\iota_{g \circ f}} & Y \ar[d]{}{0_Y} \\
TX \ar[r, swap]{}{Tf} & TY
\end{tikzcd}
\]

To prove Statement $(ii)$ we can either defer to Statement $(i)$ as a special case or give a more specialized proof by noting that in the diagram
\[
\begin{tikzcd}
X \ar[r]{}{f} \ar[d, swap]{}{0_X} & Y \ar[d]{}{0_Y} \ar[r]{}{g} & Z \ar[d]{}{0_Z} \\
TX \ar[r, swap]{}{Tf} & TY \ar[r, swap]{}{Tg} & TZ
\end{tikzcd}
\]
with the rightmost square a $T$-pullback, the total square is a $T$-pullback if and only if the leftmost square is a $T$-pullback.
\end{proof}

\begin{lemma}\label{Lemma: Tm preserves unramified}
In a tangent category $\Cscr$, if $f:X \to Y$ is $T$-unramified, then $T^mf$ is $T$-unramified for all $m \in \mathbb{N}$.
\end{lemma}
\begin{proof}
This is immediate from the definition, as being $T$-unramified asks that $T^m$ preserve the $0$-naturality square pullback for all $m \in \N$. Because $(T^m \ast 0) \cong (0 \ast T^m)$ for all $m \in \N$, the lemma follows.
\end{proof}

\begin{proposition}\label{Prop: Tunramified pullback stable}
In a tangent category $\Cscr$, assume that $f:X \to Y$ is $T$-unramified and that
\[
\begin{tikzcd}
W \ar[r]{}{\pi_1} \ar[d, swap]{}{\pi_0} & Z \ar[d]{}{g} \\
X \ar[r, swap]{}{f} & Y
\end{tikzcd}
\]
is a $T$-pullback. Then $\pi_1$ is $T$-unramified as well.
\end{proposition}
\begin{proof}
Because $T$-unramified maps are preserved by $T^m$ for all $m \in \N$, setting $m = 1$ gives the pullback square
\[
\begin{tikzcd}
TW \ar[r]{}{T\pi_1} \ar[d, swap]{}{T\pi_0} & TZ \ar[d]{}{Tg} \\
TX \ar[r, swap]{}{Tf} & TY
\end{tikzcd}
\]
which fits into a commuting cube:
\[
\begin{tikzcd}
 & TW \ar[rr]{}{T\pi_1} \ar[dd, near start]{}{T\pi_0} & & TZ \ar[dd]{}{Tg} \\
W \ar[ur]{}{0_W} \ar[dd, swap]{}{\pi_0} & & Z \ar[ur, swap]{}{0_Z} \\
 & TX \ar[rr, swap, near start]{}{Tf} & & TY \\
X \ar[ur]{}{0_X} \ar[rr, swap]{}{f} & & Y \ar[ur, swap]{}{0_Y}
\ar[from = 2-1, to = 2-3, crossing over, near end]{}{\pi_1}
\ar[from = 2-3, to = 4-3, crossing over, near start]{}{g}
\end{tikzcd}
\]
The commutativity of the cube implies that the diagram
\[
\begin{tikzcd}
W \ar[r]{}{0_W} \ar[d, swap]{}{\pi_1} & TW \ar[r]{}{T\pi_0} \ar[d]{}{T\pi_1} & TX \ar[d]{}{Tf} \\
Z \ar[r, swap]{}{f} & TZ \ar[r, swap]{}{Tg} & TY
\end{tikzcd}
\]
commutes with right-most square a pullback. The Pullback Lemma thus tells us that the left-most square is a pullback square if and only if the total square is a pullback. However, using the commutativity of the cube we see that the rectangle above is equivalent to the diagram:
\[
\begin{tikzcd}
W \ar[r]{}{\pi_0} \ar[d, swap]{}{\pi_1} & X \ar[r]{}{0_X} \ar[d]{}{f} & TX \ar[d]{}{Tf} \\
Z \ar[r, swap]{}{g} & Y \ar[r, swap]{}{0_Y} & TY
\end{tikzcd}
\]
The leftmost square above is a $T$-pullback by assumption and the rightmost square is a $T$-pullback by virtue of $f$ being $T$-unramified. Thus by the Pullback Lemma the total square is a $T$-pullback; walking this argument back shows that the diagram
\[
\begin{tikzcd}
W \ar[r]{}{0_W} \ar[d, swap]{}{\pi_1} & TW \ar[d]{}{T\pi_1} \\
Z \ar[r, swap]{}{0_Z} & TZ
\end{tikzcd}
\]
is a $T$-pullback as well. Hence, $\pi_1$ is $T$-unramified.
\end{proof}

\section{Tangent Monics and Tangent Epimorphisms}

In this section we study what we call $T$-monic and $T$-epic morphisms which, as the name suggest, are monic and epic morphisms such all tangent powers are also monic and epic respectively. 

\subsection{Tangent Monics}

Tangent monics are exactly the $T$-monics which appear in $\Cscr$. They have been briefly studied in the recent paper \cite[Definition 3.39]{GeoffMarcelloTSubmersionPaper} where they were used there to study display monic $T$-{\'e}tale maps (which were in turn used as a tangent-categorical analogue of injective local diffeomorphisms in a tangent category). We use them here as technical tools for both studying the various morphisms defined in this paper and also as a way to characterize $T$-{\'e}tale maps and monic $T$-{\'e}tale maps later in terms of the horizontal descent $\theta_f$ (Propositions \ref{Prop: Classify of Tetales} and \ref{Prop: Classify of monic Tetales} below).

\begin{definition}[{\cite[Definition 3.39]{GeoffMarcelloTSubmersionPaper}}]\label{Defn: TMonic maps}
A morphism $f:X \to Y$ is $T$-monic if and only if $T^mf$ is monic for all $m \in \N$.
\end{definition}
Because a morphism $f:X \to Y$ is monic if and only if the diagram
\[
\begin{tikzcd}
X \ar[d, equals] \ar[r, equals] & X \ar[d]{}{f} \\
X \ar[r, swap]{}{f} & Y
\end{tikzcd}
\]
is a pullback, we can make the following immediate observation regarding $T$-monics.
\begin{lemma}
In a tangent category $\Cscr$, a map $f: X \to Y$ is $T$-monic if and only if the square
\[
\begin{tikzcd}
X \ar[r, equals] \ar[d, equals] & X \ar[d]{}{f} \\
X \ar[r, swap]{}{f} & Y
\end{tikzcd}
\]
is a $T$-pullback.
\end{lemma}

\begin{example}\label{Example: T monics for SMan}
In $\SMan$, the $T$-monic morphisms are precisely the monic immersions, i.e., the embeddings. A clean proof of this may be given via Proposition \ref{Prop: f monic and thetaf monic iff Tmonic} below.
\end{example}

\begin{example}\label{Example: T monics for calg}
Let $R$ be a commutative rig. Then in $\CAlg{R}$ the $T$-monics are precisely the monomorphisms. This can be seen by observing that the tangent bundle functor $T(A) := A[\epsilon]$ preserves limits, so it in particular preserves all monics. By appealing to Corollary \ref{Cor: tangent bundle functor reflects limits} we see that $T^mf$ for $m \geq 1$ is monic if and only if $f$ is monic.
\end{example}

\begin{example}\label{Example: T monics for affine schemes}
Let $R$ be a commutative rig. Then in $\CAlg{R}^{\op}$ the $T$-monics are precisely the monomorphisms as well. This can be seen from much the same way as the example prior. 
\end{example}

\begin{example}\label{Example: T monics for Sch}
In $\Sch_{/S}$ the $T$-monic morphisms are precisely the monomorphisms. While we can once again use the same argument as in the previous example, an alternative proof of this may be given as an immediate consequence of Proposition \ref{Prop: f monic and thetaf monic iff Tmonic} below.
\end{example}

\begin{example}\label{Example: T monics for CDC}
Let $\Cscr$ be a CDC. Then $f:X \to Y$ is $T$-monic if and only if $D^n[f]$ is monic for all $m \in \mathbb{N}$, that is, $f$ is $T$-monic if and only if all of its higher-derivatives are monic. In particular, when $f$ is linear, then $f$ is $T$-monic if and only if $f$ is monic. 
\end{example}

\begin{example} Here is an example which shows, concretely, that not all monics in a tangent category are $T$-monics. Consider the CDC $\mathbf{Euc}$, the category of Euclidean spaces $\R^n$ and smooth functions between them. Then the monic $x^3:\R \to \R$ is \emph{not} a $T$-monic since its derivative is $D[x^3]:\R \times \R \to \R$ is $D[x^3](p,v) = \left(3x^2|_{x=p}\right)\overrightarrow{v} = 3p^2\overrightarrow{v}$ for all $(p,v) \in \R \times \R$. In particular, if $p = 0$ then $D[x^3](0,v) = D[x^3](0,w)$ for all $v, w \in \R$ and so $D[x^3]$ is not monic. So $x^3$ is monic but not $T$-monic. 
\end{example}

We now record that in any tangent category, the vertical lift and the canonical flip are $T$-monic.

\begin{proposition}\label{Prop: Vertical Lift is TMonic}
In a tangent category $\Cscr$, 
 \begin{enumerate}[{\em (i)}] 
    \item The vertical lift $\ell_X:TX \to T^2X$ is $T$-monic for every object $X$.
    \item The canonical flip $c_X:T^2X \to T^2X$ is a $T$-monic for every object $X$.
\end{enumerate}
\end{proposition}
\begin{proof}
For the vertical lift, by \cite[Lemma 2.13]{GeoffRobinDiffStruct}, the universality of the vertical lift implies that $\ell$ is the equalizer
\[
\begin{tikzcd}
TX \ar[r]{}{\ell_X} & T^2X \ar[rrr, bend right = 30, swap]{}{(T \ast p)_X} \ar[rrr,bend left = 30]{}{(p \ast T)_X} \ar[r]{}{(p \ast T)_X} & TX \ar[r]{}{p_X} & X \ar[r]{}{0_X} & TX
\end{tikzcd}
\]
Additionally, all powers $T^m$ of the tangent bundle functor $T$ preserve said equalizer. Putting this all together implies that $\ell_X$ is $T$ monic. On the other hand, for the canonical flip, this is immediate from $c_X$ being an isomorphism.
\end{proof}

We immediately get that $T$-monic morphisms compose, are stable under application of the tangent bundle functor, and are stable under pullbacks which are preserved by the tangent bundle functor. 

\begin{proposition}\label{Prop: T monics compose and are stable under T preserved pullbacks}
Let $f:X \to Y$ and $g:Y \to Z$ be $T$-monic maps in a tangent category $\Cscr$ and assume also that we have a pullback square
\[
\begin{tikzcd}
X \times_Y W \ar[r]{}{\pr_0} \ar[d, swap]{}{\pr_1} & X \ar[d]{}{f} \\
W \ar[r, swap]{}{h} & Y
\end{tikzcd}
\]
which is preserved by $T^m$ for all $m \in \N$. Then:
\begin{enumerate}[{\em (i)}] 
    \item The composite $g \circ f:X \to Z$ is $T$-monic.
    \item For all $m \in \N$, $T^mf$ is $T$-monic.
    \item The map $\pr_1:X \times_Y W \to W$ is $T$-monic.
\end{enumerate}
\end{proposition}
\begin{proof}
Claim $(i)$ is immediate from the observations that $T^m(g \circ f) = T^m(g) \circ T^m(f)$ for all $m \in \N$ and that monics compose. Claim $(ii)$ is immediate from the fact that for all $k, m \in \N$ if $T^{m}f$ is monic then $T^{m+k}f = T^k(T^mf)$ is monic. For claim $(iii)$, note that upon fixing an $m \in \N$, applying the functor $T^m$ to the given pullback diagram, and using that $T^m$ preserves this pullback gives rise to the pullback square:
\[
\begin{tikzcd}
T^m(X \times_Y W) \ar[rr]{}{T^m\pr_0} \ar[d, swap]{}{T^m\pr_1} & & T^mX \ar[d]{}{T^mf} \\
T^mW \ar[rr, swap]{}{T^mh} & & T^mY
\end{tikzcd}
\]
Since $T^mf$ is monic and monics are stable under pullback, we see that $T^m\pr_1$ is monic.
\end{proof}

We now show that for any $0$-carrable morphism the unique map $0_{X/Y}:X \to T_{X/Y}$ is $T$-monic.

\begin{lemma}\label{Lemma: 0carrable is Tmonic}
In a tangent category $\Cscr$, for any $0$-carrable morphism $f:X \to Y$, $0_{X/Y}$ is $T$-monic.
\end{lemma}

\begin{proof}
Because $0_{X/Y}:X \to T_{X/Y}$ is a section of $p_{X/Y}:T_{X/Y} \to X$, $T^{m}0_{X/Y}$ is monic for all $m \in \N$ because sections are absolute monomorphisms.
\end{proof}

We now show that a $p$-carrable map is a $T$-monic if it is monic and that the horizontal descent is $T$-monic.

\begin{proposition}\label{Prop: f monic and thetaf monic iff Tmonic}
Let $f:X \to Y$ be a $p$-carrable map in a tangent category $\Cscr$. Then $f$ is $T$-monic if and only if $f$ is monic and $\theta_f$ is $T$-monic.
\end{proposition}
\begin{proof} For the $\implies$ direction, since $f$ is $p$-carrable, then for all $m \in \N$ we have pullback diagrams
\[
\begin{tikzcd}
T^{m+1}X \ar[drr, bend left = 20]{}{T^{m+1}f} \ar[ddr, swap, bend right = 20]{}{(T^m \ast p)_X} \ar[dr, dashed]{}{\exists!T^m\theta_f} \\
 & T^m(f^{\ast}TY) \ar[r]{}{T^m\pi_1} \ar[d, swap]{}{T^m\pi_0} & T^{m+1}Y \ar[d]{}{(T^m \ast p)_Y} \\
 & T^mX \ar[r, swap]{}{T^mf} & T^mY
\end{tikzcd}
\]
induced by applying $T^m$ to the pullback diagram defining $f^{\ast}(TY)$. In each case, since $T^mf$ is monic for all $m \in \N$, $T^m\pi_0$ is monic. Using also that $T^{m+1}f$ is monic for all $m \in \N$ allows us to deduce that $T^m\theta_f$ is monic for all $m \in \N$ as well because $T^{m+1}f = T^mf \circ T^m\theta_f$. 

For the $\impliedby$ direction, in order to prove that $f$ is $T$-monic, we must show that $T^m(f)$ is monic for all $m \in \N$. We show this by induction on $m$. For the base case $m = 0$ we note that $T^0(f) = \id(f) = f$ is monic by assumption. When $m = 1$ we note that since $Tf = \pr_1 \circ \theta_f$, $Tf$ is monic precisely because $\theta_f$ and $\pr_1$ are both monic since $\theta_f$ is by assumption and $\pr_1$ because it is the pullback of $p_Y:TY \to Y$ against the monic $f:X \to Y$. We now proceed with the induction on $m$. Let $m \in \N$ be given with $m \geq 1$ and assume that $T^m(f)$ is monic. We will show that $T^{m+1}(f)$ is monic. Now since we have a pullback
\[
\begin{tikzcd}
T^m(f)^{\ast}(T^{m+1}(Y)) \ar[r]{}{\pr_1} \ar[d, swap]{}{\pr_0} & T(T^m(Y)) \ar[d]{}{(p \ast T^m)_Y} \\
T^m(X) \ar[r, swap]{}{T^m(f)} & T^m(Y)
\end{tikzcd}
\]
Then $\pr_1$ is monic by the Inductive Hypothesis on $T^m(f)$. Additionally, applying Proposition \ref{Proposition: Tangent of thetaf} to $T^{m-1}f$ and expanding out the induced algebraic equations shows that there are isomorphisms $\tilde{c}_m:T(T^{m-1}(f)^{\ast}(T^mY)) \to T^m(f)^{\ast}(T^{m+1}Y)$ and $\hat{c}_m:T^{m+1}X \to T^{m+1}X$ which make the equation $\tilde{c}_m \circ T^m\theta_{f} = \theta_{T^mf} \circ \hat{c}_m$ hold. But as $\tilde{c}_m$ and $\hat{c}_m$ are isomorphisms and $T^m\theta_f$ is monic, $\theta_{T^mf}$ is monic. This then gives rise to the commuting diagram
\[
\begin{tikzcd}
T^{m+1}X \ar[drr, bend left = 20]{}{T^{m+1}f} \ar[dr, dashed]{}{\exists!\theta_{T^mf}} \ar[ddr, swap, bend right = 20]{}{(p \ast T^m)} \\
 & T^{m}(f)^{\ast}(T^{m+1}(Y)) \ar[d, swap]{}{\pr_0} \ar[r]{}{\pr_1} & T^{m+1}Y \ar[d]{}{(p \ast T^m)_Y} \\
 & T^mX \ar[r, swap]{}{T^mf} & TY
\end{tikzcd}
\]
and hence allows us to deduce that $T^{m+1}(f)$ is monic. 
\end{proof}

\begin{corollary}
In a tangent category $\Cscr$, if $f:X \to Y$ is $p$-carrable and $T$-monic, then $T^m\theta_f$ and $\theta_{T^mf}$ are $T$-monic for all $m \in \N$.
\end{corollary}
\begin{proof} By Proposition \ref{Prop: f monic and thetaf monic iff Tmonic}, $\theta_f$ is $T$-monic. By Proposition \ref{Prop: T monics compose and are stable under T preserved pullbacks}.(ii), we get that $T^m\theta_f$ is also $T$-monic. Again by Proposition \ref{Prop: T monics compose and are stable under T preserved pullbacks}.(ii), we get that $T^m f$ is $T$-monic, hence by Proposition \ref{Prop: f monic and thetaf monic iff Tmonic}, $\theta_{T^mf}$ is $T$-monic. 
\end{proof}
\begin{lemma}\label{Lemma: Split monics are tmonic}
Let $\Cscr$ be a tangent category and let $s$ be a split monic. Then $s$ is $T$-monic.
\end{lemma}
\begin{proof}
Because split monic are absolute monomorphisms, $T^ms$ is a monic for all $m \in \N$.
\end{proof}
\begin{proposition}
    In any tangent category $\Cscr$, for any differential bundle $q:E \to X$, the zero $\zeta:X \to E$ is $T$-monic. In particular, $0_X:X \to TX$ is $T$-monic.
\end{proposition}
\begin{proof}
By Lemma \ref{Lemma: Split monics are tmonic} in any tangent category $\Cscr$, for any differential bundle $q:E \to X$, $\zeta:X\to E$ is $T$-monic because $\zeta$ is split monic. The final statement follows from specializing to the differential bundle $p_X:TX \to X$.
\end{proof}

\subsection{Tangent Epics}

In this short subsection we introduce and give some basic results regarding the dual to tangent monics: tangent epics. While we will not study these morphisms seriously in this paper, we introduce them to lay some groundwork for the study of $T$-submersions in Section \ref{Section: TSubmersions} and \ref{Section: Split Submersions}. We are indebted to Geoff Cruttwell for suggesting studying $T$-epics and introducing them here.

\begin{definition}\label{Defn: Tangent epimorphisms}
A morphism $f:X \to Y$ is a \emph{$T$-epimorphism} (\emph{$T$-epic} for short) if and only if $T^mf$ is epic for all $m \in \N$.
\end{definition}
Because a morphism $f:X \to Y$ is epic if and only if the diagram 
\[
\begin{tikzcd}
X\ar[r]{}{f} \ar[d, swap]{}{f}  & Y \ar[d, equals] \\
Y \ar[r, equals] & Y
\end{tikzcd}
\]
is a pushout, we can make the following immediate observation regarding $T$-epimorphisms.
\begin{lemma}
Let $f:X \to Y$ be a morphism in a tangent category $\Cscr$. We say that $f$ is \emph{$T$-epic} if the diagram
\[
\begin{tikzcd}
X \ar[r]{}{f} \ar[d, swap]{}{f} & Y \ar[d, equals] \\
Y \ar[r, equals] & Y
\end{tikzcd}
\]
is a $T$-pushout.
\end{lemma}

While we do not provide examples of $T$-epimorphisms explicitly in this paper, we will provide a stability result below. Afterwards, we will give some quick calculations that allow one to deduce some classes of maps as $T$-epimorphisms.
\begin{proposition}
Let $f:X \to Y$ and $g:Y \to Z$ be morphisms in a tangent category $\Cscr$. Additionally, assume also that we have a $T$-pushout square:
\[
\begin{tikzcd}
X \ar[r]{}{f} \ar[d, swap]{}{h} & Y \ar[d]{}{\rotatebox[origin = c]{180}{$\pi$}_0} \\
W \ar[r, swap]{}{\rotatebox[origin = c]{180}{$\pi$}_1} & P
\end{tikzcd}
\]
Then:
\begin{enumerate}[{\em (i)}] 
    \item If $f$ and $g$ are $T$-epic, the composite $g \circ f:X \to Z$ is $T$-epic. In particular, if $f$ is $T$-epic then $g \circ f$ is $T$-epic if and only if $f$ is $T$-epic.
    \item For all $m \in \N$, $T^mf$ is $T$-epic.
    \item The map $\rotatebox[origin = c]{180}{$\pi$}_1:W \to P$ is $T$-epic.
\end{enumerate}
\end{proposition}
\begin{proof}
We prove the proposition item-by-item. For $(i)$ we note that this is immediate from the fact that if $T^mf$ and $T^mg$ are epic for all $m \in \N$ then $T^m(g \circ f) = T^mg \circ T^mf$ is epic as well. The last statement is immediate from the elementary category theory fact that if there is a composable pair of morphisms $\gamma:A \to  B$ and $\rho:B \to C$ for which $\gamma$ is epic, then their composite $\rho \circ \gamma$ is epic if and only if $\rho$ is epic.

For $(ii)$ assume that the diagram
\[
\begin{tikzcd}
X \ar[r]{}{f} \ar[d, swap]{}{f} & Y \ar[d, equals] \\
Y \ar[r, equals] & Y
\end{tikzcd}
\]
is a $T$-pushout. But then for all $m \in \N$ the diagram
\[
\begin{tikzcd}
T^mX \ar[r]{}{T^mf} \ar[d, swap]{}{T^mf} & T^mY \ar[d, equals] \\
T^mY \ar[r, equals] & T^mY
\end{tikzcd}
\]
is a $T$-pushout as well.

For $(iii)$ note that since the diagram given is a $T$-pushout, for all $m \in \N$ the diagram
\[
\begin{tikzcd}
T^mX \ar[r]{}{T^mf} \ar[d, swap]{}{T^mh} & T^mY \ar[d]{}{T^m\rotatebox[origin = c]{180}{$\pi$}_0} \\
T^mW \ar[r, swap]{}{T^m\rotatebox[origin = c]{180}{$\pi$}_1} & T^mP
\end{tikzcd}
\]
is a pushout. But then since epimorphisms are stable under pushout, $T^m\rotatebox[origin = c]{180}{$\pi$}_1$ is epic for all $m \in \N$.
\end{proof}

\begin{lemma}\label{Lemma: Split epimorphisms are tepimorphisms}
Let $\Cscr$ be a tangent category and let $r$ be a split epimorphism. Then $r$ is a $T$-epimorphism.
\end{lemma}
\begin{proof}
Because split epimorphisms are absolute epimorphisms, $T^mr$ is an epimorphism for all $m \in \N$.
\end{proof}
\begin{proposition}
    In any tangent category $\Cscr$, for any differential bundle $q:E \to X$, $q$ is a $T$-epimorphism. In particular, $p_X:TX \to X$ is a $T$-epimorphism.
\end{proposition}
\begin{proof}
By Lemma \ref{Lemma: Split epimorphisms are tepimorphisms} in any tangent category $\Cscr$, for any differential bundle $q:E \to X$, $q$ is a $T$-epimorphism because $q$ is a split epimorphism. The final statement follows from specializing to the differential bundle $p_X:TX \to X$.
\end{proof}

\section{Strong Tangent-Immersions}\label{Section: Strong TImmersions}
In differential geometry, immersions are the morphisms $f:X \to Y$ of smooth manifolds which preserve differential data in the sense that at every point $x \in X$, $T_xf:T_xX \to T_{f(x)}Y$ is injection. In this section we introduce a notion of immersions which is well-suited for tangent-categorical settings because defining and studying such objects involves \emph{only} data encoded by the tangent structure $\Tbb$ with which the tangent category $\Cscr$ is equipped. However, because this notion of immersion is quite general, it has the downside of being quite a strong definition formally speaking. It in particular does not directly reference the differential bundle categories $\DBun(X)$. As such, we see that a potential trade-off for having immersions being definable in \emph{any} tangent category with a condition that can be checked for \emph{any} morphism is that the inherent linearity in characterizing immersions takes a backseat role. For the reader who is bothered by this, do not fret; we take a more ``linear'' approach to immersions in Section \ref{Section: TImmersions} below.

In order to make it so that strong immersions depend \emph{solely} on data encoded by the tangent structure $\Tbb$ and in particular the naturality square
\[
\begin{tikzcd}
TX\ar[r]{}{Tf} \ar[d, swap]{}{p_X} & TY \ar[d]{}{p_Y} \\
X \ar[r, swap]{}{f} & Y
\end{tikzcd}
\]
we will need to recall some classical category theory of prepullbacks. While we generically focus on morphisms which are $p$-carrable, it is interesting for the maximally general case to give such a connection to prepullbacks. Consequently we will give the definition for general tangent categories before specializing and focusing more on the $p$-carrable case. We will also examine how strong $T$-immersions relate to $T$-unramified and $0$-carrable morphisms later on in this section (cf.\@ Proposition \ref{Prop: Immersion Unramified} in particular).

\subsection{Prepullbacks and Immersions}

In order to define strong $T$-immersions, let us recall first what it means to be a prepullback of a cospan $Z \xrightarrow{k} W \xleftarrow{h} Y$ in a category. 
\begin{definition}\label{Defn: Prepullback}
In a category $\Cscr$, a commuting square
\[
\begin{tikzcd}
X \ar[r]{}{f} \ar[d, swap]{}{g} & Y \ar[d]{}{h} \\
Z \ar[r, swap]{}{k} & W
\end{tikzcd}
\]
is a \emph{prepullback in $\Cscr$} if for any commutative diagram 
\[
\begin{tikzcd}
V \ar[r]{}{s} \ar[d, swap]{}{t} & Y \ar[d]{}{h} \\
Z \ar[r, swap]{}{k} & W
\end{tikzcd}
\]
there is at most one map $\ell:V \to X$ which makes the following diagram commute: 
\[
\begin{tikzcd}
V \ar[drr, bend left = 20]{}{s} \ar[ddr, swap, bend right = 20]{}{t} \ar[dr, dashed]{}{\exists\leq 1} \\
& X \ar[r]{}{f} \ar[d, swap]{}{g} & Y \ar[d]{}{h} \\
& Z \ar[r, swap]{}{k} & W
\end{tikzcd}
\]
Additionally, if $T:\Cscr \to \Cscr$ is an endofunctor then we say that a commuting square is a \emph{$T$-prepullback} if it is a prepullback such that applyinng $T^m$ to the diagram, for all $m \in \N$, remains a prepullback. 
\end{definition}
We will see below that a strong $T$-immersion will precisely be a map $f$ in a tangent category such that the projection naturality square for $f$ is a $T$-prepullback. 
 
\begin{definition}\label{Defn: Strong Timmersion}
In a tangent category $\Cscr$, a map $f:X \to Y$ is a \emph{strong $T$-immersion} if the naturality square
\[
\begin{tikzcd}
TX \ar[d, swap]{}{p_X} \ar[r]{}{Tf} & TY \ar[d]{}{p_Y} \\
X \ar[r, swap]{}{f} & Y
\end{tikzcd}
\]
is a $T$-prepullback.
\end{definition}

We now show that strong $T$-immersions are stable under composition and $T$-pullbacks. To go about this, however, we need to provide a folkloric structural lemma for the sake of completeness. While we cannot and do not claim originality of the result\footnote{We have been unable to find it directly stated in the literature, but the second author learned it from Robin Cockett, who informed the second author that this is a well-known structural result.}, it is a useful technical tool which we will require.

\begin{lemma}[Folklore]\label{Lemma: The prepullback stacking lemma}
Consider commutative diagrams with squares numbered as below:
\[
\begin{tikzcd}
X \ar[r, ""{name = UL}]{}{k} \ar[d, swap]{}{f} & Y \ar[d]{}{g} \ar[r, ""{name = UR}]{}{\ell} & Z \ar[d]{}{h} \\
W \ar[r, swap, ""{name = DL}]{}{k^{\prime}} & V \ar[r, swap, ""{name = DR}]{}{\ell^{\prime}} & U
\ar[from = UL, to = DL, symbol = {\rotatebox[origin=c]{90}{$(1)$}}]
\ar[from = UR, to = DR, symbol = {\rotatebox[origin=c]{90}{$(2)$}}]
\end{tikzcd}
\qquad
\begin{tikzcd}
X \ar[r, ""{name = U}]{}{\ell \circ k} \ar[d, swap]{}{f} & Z \ar[d]{}{h} \\
W \ar[r, swap, ""{name = D}]{}{\ell^{\prime} \circ k^{\prime}} & U
\ar[from = U, to = D, symbol = {\rotatebox[origin=c]{90}{$(3)$}}]
\end{tikzcd}
\]
Then:
\begin{enumerate}[{\em (i)}] 
    \item If Square $(3)$ is a prepullback then Square $(1)$ is a prepullback as well.
    \item If Square $(1)$ is a prepullback and Square $(2)$ is a prepullback then Square $(3)$ is a prepullback.
\end{enumerate}
In particular, if Square $(2)$ is a pullback then Square $(1)$ is a prepullback if and only if Square $(3)$ is a prepullback.
\end{lemma}
\begin{proof}
We first prove Statement (i). Assume that Square $(3)$ is a prepullback and that we have a commuting diagram of the form:
\[
\begin{tikzcd}
P \ar[dr, shift left = 1]{}{\alpha} \ar[dr, shift right = 1, swap]{}{\beta} \ar[drr, bend left =20]{}{s} \ar[ddr, swap, bend right = 20]{}{t} & \\
& X \ar[r]{}{k} \ar[d, swap]{}{f} & Y \ar[d]{}{g} \\
& W \ar[r, swap]{}{k^{\prime}} & V
\end{tikzcd}
\]
Because then the diagram
\[
\begin{tikzcd}
P \ar[dr, shift left = 1]{}{\alpha} \ar[dr, shift right = 1, swap]{}{\beta} \ar[drr, bend left =20]{}{\ell \circ s} \ar[ddr, swap, bend right = 20]{}{t} \\
& X \ar[r, ""{name = U}]{}{\ell \circ k} \ar[d, swap]{}{f} & Z \ar[d]{}{h} \\
& W \ar[r, swap, ""{name = D}]{}{\ell^{\prime} \circ k^{\prime}} & U
\end{tikzcd}
\]
commutes, using that Square $(3)$ is a prepullback implies that $\alpha = \beta$ and so proves that Square $(1)$ is a prepullback.

We now prove Statement (ii). Assume that Squares $(1)$ and $(2)$ are prepullbacks and that we have a commuting diagram of the form:
\[
\begin{tikzcd}
P \ar[dr, shift left = 1]{}{\alpha} \ar[dr, shift right = 1, swap]{}{\beta} \ar[drr, bend left =20]{}{s} \ar[ddr, swap, bend right = 20]{}{t} \\
& X \ar[r, ""{name = U}]{}{\ell \circ k} \ar[d, swap]{}{f} & Z \ar[d]{}{h} \\
& W \ar[r, swap, ""{name = D}]{}{\ell^{\prime} \circ k^{\prime}} & U
\end{tikzcd}
\]
Now $g \circ k = k^{\prime} \circ f$ and $f \circ \alpha = t = f \circ \beta$, the following diagram commutes: 
\[
\begin{tikzcd}
P \ar[dr, shift left = 1]{}{k \circ \alpha} \ar[dr, shift right = 1, swap]{}{k \circ \beta}  \ar[drr, bend left =20]{}{s} \ar[ddr, swap, bend right = 40]{}{k^{\prime} \circ t} \\
& Y \ar[r, ""{name = U}]{}{\ell} \ar[d, swap]{}{g} & Z \ar[d]{}{h} \\
& V \ar[r, swap, ""{name = D}]{}{\ell^{\prime}} & U
\end{tikzcd}
\]
Then using that Square $(2)$ is a prepullback, we get that $k \circ \alpha = k \circ \beta$. This then implies that the diagram
\[
\begin{tikzcd}
P \ar[dr, shift left = 1]{}{\alpha} \ar[dr, shift right = 1, swap]{}{\beta} \ar[drr, bend left =20]{}{k \circ \alpha} \ar[ddr, swap, bend right = 20]{}{t} \\
& X \ar[r, ""{name = U}]{}{k} \ar[d, swap]{}{f} & Y \ar[d]{}{g} \\
& W \ar[r, swap, ""{name = D}]{}{k^{\prime}} & V
\end{tikzcd}
\]
commutes as well. Using that Square $(1)$ is a prepullback gives that $\alpha = \beta$ and hence that Square $(3)$ is a prepullback as well.
\end{proof}

With this at hand we now have the technical ingredients to show that strong $T$-immersions are stable under composition, application of all powers of the tangent bundle functor, and $T$-pullbacks.

\begin{proposition}\label{Prop: Strong Immersions comp and Tm stable}
In a tangent category $\Cscr$, if $f:X \to Y$ and $g:Y \to Z$ are strong $T$-immersions then:
\begin{enumerate}[{\em (i)}] 
    \item The composite $g \circ f:X \to Z$ is a strong $T$-immersion.
    \item The map $T^mf$ is a strong $T$-immersion for all $m \in \N$.
    \item If the $T$-pullback 
\[
\begin{tikzcd}
W \ar[r]{}{\pi_1} \ar[d, swap]{}{\pi_0} & Z \ar[d]{}{g} \\
X \ar[r, swap]{}{f} & Y
\end{tikzcd}
\]
exists in $\Cscr$, then $\pi_1$ is a strong $T$-immersion as well.
\end{enumerate}
\end{proposition}
\begin{proof}
Statement (i) follows from an immediate application of Lemma \ref{Lemma: The prepullback stacking lemma}.i. Statement (ii) holds because by assumption all powers of the tangent bundle functor preserve the $T$-prepullback
\[
\begin{tikzcd}
TX \ar[r]{}{Tf} \ar[d, swap]{}{p_X} & TY \ar[d]{}{p_Y} \\
X \ar[r]{}{f} & Y
\end{tikzcd}
\]
and use of the various whiskerings of the canonical flip allow one to drag $T^m$ of the square above to the isomorphic square with vertical maps $(p \ast T^m)_X$ and $(p \ast T^m)_Y$. 

For statement (iii), consider the commuting diagram:
\[
\begin{tikzcd}
TW \ar[r]{}{p_W} \ar[d, swap]{}{T\pi_1} & W \ar[d]{}{\pi_1} \ar[r]{}{\pi_0} & X \ar[d]{}{f} \\
TZ \ar[r, swap]{}{p_Z} & Z \ar[r, swap]{}{g} & Y
\end{tikzcd}
\]
Because the rightmost square is a pullback by assumption, by Lemma \ref{Lemma: The prepullback stacking lemma} it suffices to prove the whole rectangle is a prepullback in order to conclude that the leftmost square is a prepullback. To this end observe also that the commuting rectangle above is equal to the commuting rectangle
\[
\begin{tikzcd}
TW \ar[r]{}{T\pi_0} \ar[d, swap]{}{T\pi_1} & TX \ar[r]{}{p_X} \ar[d]{}{Tf} & X \ar[d]{}{f} \\
TZ \ar[r, swap]{}{Tg} & TY \ar[r, swap]{}{p_Y} & Y
\end{tikzcd}
\]
we see that is suffices to prove that this last rectangle is a prepullback. To this end assume that we have a commuting diagram:
\[
\begin{tikzcd}
V \ar[drr, bend left = 20]{}{t} \ar[ddr, swap, bend right = 20]{}{s} \ar[dr, shift left = 1]{}{\alpha} \ar[dr, shift right = 1, swap]{}{\beta} \\
 & TW \ar[r]{}{T\pi_1} \ar[d]{}{p_X \circ T\pi_0} & TZ \ar[d]{}{p_Y \circ Tg} \\
 & X \ar[r, swap]{}{f} & Y
\end{tikzcd}
\]
Because $TW$ is a pullback with projections $T\pi_0$ and $T\pi_1$, to prove that $\alpha = \beta$ it suffices to prove that $T\pi_1 \circ \alpha = T\pi_1 \circ \beta$ and $T\pi_0 \circ \alpha = T\pi_0 \circ \beta$. The first of these identities can be read off the diagram directly as
\[
T\pi_1 \circ \alpha = t = T\pi_1 \circ \beta.
\]
For the other identity, consider that the computations
\begin{align*}
& Tf \circ T\pi_0 \circ \alpha = Tg \circ T\pi_1 \circ \alpha = Tg \circ t = Tg \circ T\pi_1 \circ \beta = Tf \circ T\pi_0 \circ \beta
\end{align*}
and
\begin{align*}
& p_X \circ T\pi_0 \circ \alpha = s = p_X \circ T\pi_0 \circ \beta
\end{align*}
imply that the diagram
\[
\begin{tikzcd}
V \ar[drr, bend left = 20]{}{Tg \circ t} \ar[ddr, swap, bend right = 60]{}{s} \ar[dr, shift left = 1]{}{T\pi_0 \circ \alpha} \ar[dr, shift right = 1, swap]{}{T\pi_0 \circ \beta} \\
 & TX \ar[r]{}{Tf} \ar[d, swap]{}{p_X} & TY \ar[d]{}{p_Y} \\
 & X \ar[r, swap]{}{f} & Y
\end{tikzcd}
\]
commutes. Using that $f$ is a $T$-immersion gives that $T\pi_0 \circ \alpha = T\pi_0 \circ \beta$ and so allows us to deduce that $\alpha = \beta$. So that the naturality square for $\pi_1$ is a prepullback. Running this same argument for arbitrary $m \in \N$ gives that $\pi_1$ is a $T$-immersion.
\end{proof}

When given a $p$-carrable morphism $f:X \to Y$, it is significantly simpler to check when $f$ is a strong $T$-immersion by making use of the horizontal descent of $f$. In particular, we will show that $f$ is a strong $T$-immersion precisely when its horizontal descent of $f$ is $T$-monic. 

\begin{lemma}\label{Lemma: Strong Timmersion if and only if thetaf monic}
In a tangent category $\Cscr$, if $f:X \to Y$ is $p$-carrable then $f$ is a strong $T$-immersion if and only if $\theta_f:TX \to f^{\ast}(TY)$ is $T$-monic.
\end{lemma}
\begin{proof}
$\implies:$ Assume that $f$ is a strong $T$-immersion and that there are maps $g$ and $h$ rendering the diagram
\[
\begin{tikzcd}
Z \ar[r, shift left = 1]{}{g} \ar[r, shift right = 1, swap]{}{h} & TX \ar[r]{}{\theta_f} & f^{\ast}(TY)
\end{tikzcd}
\]
commutative. Then $p_X \circ g = \pi_0 \circ \theta_f \circ g = \pi_1 \circ \theta_f \circ g = Tf \circ g$, $p_X \circ h = Tf \circ h$, and both $g$ and $h$ make the diagram
\[
\begin{tikzcd}
Z \ar[drr, bend left = 20]{}{} \ar[ddr, swap, bend right = 20]{}{} \ar[dr, shift left = 0.5]{}{g} \ar[dr, shift right = 0.5, swap]{}{h} \\
 & TX \ar[r]{}{Tf} \ar[d, swap]{}{p_X} & TY \ar[d]{}{p_Y} \\
 & X \ar[r, swap]{}{f} & Y
\end{tikzcd}
\]
commute. But since the naturality square for $p$ is a prepullback we have that $g = h$ and so that $\theta_f$ is monic. Finally, using that all powers $T^m$ preserve the given prepullback and running the same argument for $T^m\theta_f$ shows that $\theta_f$ is $T$-monic.

$\impliedby:$ If $\theta_f$ is $T$-monic then it is immediate that for all $m \in \N$ the naturality square for $(T^m \ast p)$ at $T^mf$ has at most one factorization through $T^{m+1}X$ by virtue of $T^m\theta_f$ being monic.
\end{proof}

We now characterize the strong $T$-immersions in our main example categories of interest. 

\begin{example}\label{Example: TImmersions in SMan}
In $\SMan$, a morphism $f:X \to Y$ is a strong $T$-immersion if and only if it is an immersion of smooth manifolds. Recall from Example \ref{Example: What  is thetaf in SMan} that $\theta_f:TX \to f^{\ast}(TY)$ is given by $\theta_f\left(x,\overrightarrow{v}\right) = \left(x,D[f](x)\cdot \overrightarrow{v}\right).$ Moreover, since $p_Y:TY \to Y$ is a submersion of manifolds, we have that:
\[f^{\ast}(TY) \cong X \times_{Y} TY \cong \coprod_{x \in X} T_{f(x)}Y.\] 
Thus we see that $\theta_f$ is monic if and only if for all $x \in X$, the map $T_xf:T_xX \to T_{f(x)}Y$ is injective. However, this is the case if and only if $f$ is an immersion.
\end{example}

\begin{example}\label{Example: TImmersions in CAlgR}
Let $R$ be a commutative rig. In $\mathbf{CAlg}_R$ the strong $T$-immersions are precisely the injective morphisms of $R$-algebras. To see this note the diagram
\[
\begin{tikzcd}
A[\epsilon] \ar[dr, dashed]{}{\theta_f} \ar[drr, bend left = 20]{}{f[\epsilon]} \ar[ddr, swap, bend right = 20]{}{p_A} \\
 & f^{\ast}(B[\epsilon]) \ar[r]{}{\pi_1} \ar[d, swap]{}{\pi_0} & B[\epsilon] \ar[d]{}{p_B}  \\
 & A \ar[r, swap]{}{f} & B
\end{tikzcd}
\]
has $f^{\ast}(B[\epsilon]) \cong A \ltimes B$. Under the isomorphism $A \ltimes B \cong \left\lbrace (a, b\epsilon) \; | \; a \in A, b \in B \right\rbrace$, we see that $\theta_f:A[\epsilon] \to f^{\ast}(B[\epsilon])$ takes the form $\theta_f(a+x\epsilon) = (a,f(x)).$ Thus we see immediately that $\theta_f$ is monic if and only if $f$ is monic.
\end{example}

\begin{example}\label{Example: Timmersions in CAlgop}
Let $R$ be a commutative ring. In $\mathbf{CAlg}_R^{\op}$ the strong $T$-immersions are precisely the formally unramified morphisms of $R$-algebras\footnote{Which are $R$-algebra morphisms $f: A \to B$ such that $\Kah{B}{A} \cong 0$ (cf.\@ \cite[D{\'e}finition 19.10.2 \& Proposition 20.7.4]{EGA04}).}. To see this, consider that the map $\Spec f:\Spec B \to \Spec A$ induces corresponding structure map $\theta_{f}:T_{B/R} \to f^{\ast}(T_{A/R})$ given by $\theta_f = \Spec(\Sym_B(v))$ for $v$ the map in the relative cotangent exact sequence:
\[
\begin{tikzcd}
\Kah{A}{R} \otimes_A B \ar[r]{}{v} & \Kah{B}{R} \ar[r] & \Kah{B}{A} \ar[r] & 0
\end{tikzcd}
\]
Now, since $\DBun(B)^{\op} \simeq \Mod{B}$ by \cite{GeoffJSDiffBunComAlg}, we see that $\Spec(\Sym_B(v))$ is monic if and only if $v$ is an epimorphism of $B$-modules. However, as the relative cotangent sequence is exact, $v$ is an epimorphism if and only if $0 \cong \operatorname{Coker}(v) \cong \Kah{B}{A}$.
\end{example}

\begin{example}\label{Example: Timmersions in Schemes}
The above example extends to the full category of schemes $\Sch_{S}$ for arbitrary base schemes $S$. Indeed, in $\Sch_{/S}$, the morphisms for which $\theta_f$ is monic are precisely the formally unramified morphisms of $S$-schemes. Recall from \cite[Proposition 17.2.1]{EGA44} that a morphism of schemes $f:X \to Y$ is formally unramified if and only if the sheaf of relative K{\"a}hler differentials satisfies $\Kah{X}{Y} \cong 0$. Recall also from the relative cotangent sequence for K{\"a}hler differentials of schemes that there is an exact sequence of quasi-coherent sheaves on $X$
\[
\Kah{Y}{S} \otimes_{\Ocal_Y} \Ocal_X \to \Kah{X}{S}  \to \Kah{X}{Y} \to 0
\]
so $\Kah{X}{Y}$ is the cokernel of the map $\Kah{Y}{S} \otimes_{\Ocal_Y} \Ocal_X \to \Kah{X}{S}$ in $\QCoh(X)$. We now use the equivalence $\QCoh(X) \simeq \DBun(X)^{\op}$ of \cite{GeoffJSDiffBunComAlg} to observe that the equivalences
\begin{prooftree}
    \AxiomC{$\Kah{X}{Y} \cong 0$ in $\QCoh(X)$}
    \UnaryInfC{$T_{X/Y} \cong 0$ in $\DBun(X)$}
\end{prooftree}
and
\begin{prooftree}
    \AxiomC{Exact Sequence $\Kah{Y}{S} \otimes_{\Ocal_Y} \Ocal_X \to \Kah{X}{S} \to \Kah{X}{Y} \to 0$ in $\QCoh(X)$}
    \UnaryInfC{Exact Sequence $X \to T_{X/Y} \to T_{X/S} \to T_{Y/S} \times_Y X$ in $\DBun(X)$}
\end{prooftree}
combine to allow us to deduce
\begin{prooftree}
    \AxiomC{Exact Sequence $X \to T_{X/Y} \to T_{X/S} \to T_{Y/S} \times_Y X$ in $\DBun(X)$}
    \UnaryInfC{Exact Sequence $X = X \to T_{X/S} \to T_{Y/S} \times_Y X$ in $\DBun(X)$}
    \UnaryInfC{$\theta_f:T_{X/S} \to T_{Y/S} \times_Y X$ a kernel in $\DBun(X)$}
    \UnaryInfC{$\theta_f$ monic in $\DBun(X)$, as $\DBun(X) \simeq \QCoh(X)^{\op}$ is Abelian}
\end{prooftree}
\end{example}

\begin{example}\label{Example: Immersions in CDC} In a CDC, as $\theta_f = \langle \pi_0, D[f]\rangle$ it follows that $\theta_f$ is $T$-monic if and only if $D[f]$ is $T$-monic. Therefore, by Lemma \ref{Lemma: Strong Timmersion if and only if thetaf monic}, $f$ is a strong $T$-immersion if and only if $D^n[f]$ is monic for every $n \in \mathbb{N}$. In particular, if $f$ is linear then $f$ is a strong $T$-immersion if and only if $f$ is monic.
\end{example}

\subsection{Strong Immersions and Unramification}\label{Subsection: STrong Immersions and Unramification}
We now show that strong $T$-immersions are $T$-unramified morphisms in any tangent category.
 
\begin{proposition}\label{Prop: Immersion Unramified}
In a tangent category $\Cscr$, if $f:X \to Y$ is a strong $T$-immersion then $f$ is also $T$-unramified.
\end{proposition}
\begin{proof}
Assume that we have a cone
\[
\begin{tikzcd}
C \ar[r]{}{\beta} \ar[d, swap]{}{\alpha} & Y \ar[d]{}{0_Y} \\
TX \ar[r, swap]{}{Tf} & TY
\end{tikzcd}
\]
and note that if we can show that $(C,\alpha,\beta)$ factors through $(X,0_X,f)$ then because $0_X$ is a section (and hence monic), it will follow immediately that $(X,0_X,f)$ is the terminal cone over $TX \xrightarrow{Tf} TY \xleftarrow{0_Y} Y$. To this end we compute that
\[
f \circ p_X \circ \alpha = p_Y \circ Tf \circ \alpha = p_Y  \circ 0_Y \circ \beta = \beta
\]
so the diagram
\[
\begin{tikzcd}
C \ar[r]{}{0_Y \circ \beta} \ar[d, swap]{}{p_X \circ \alpha} & TY \ar[d]{}{p_Y} \\
X \ar[r, swap]{}{f} & Y
\end{tikzcd}
\]
commutes. But now consider the diagram:
\[
\begin{tikzcd}
 & C \ar[ddr, bend left = 60]{}{0_Y \circ \beta} \ar[ddl, swap, bend right = 60]{}{p_X \circ \alpha} \ar[d, shift left = 0.5]{}{\alpha} \ar[d, shift right = 0.5, swap]{}{0_X \circ p_X \circ \alpha}  \\
 & TX \ar[dr]{}{Tf} \ar[dl, swap]{}{p_X} \\
X \ar[dr, swap]{}{f} & & TY \ar[dl]{}{p_Y} \\
 & Y
\end{tikzcd}
\]
To see it commutes, we compute that:
\begin{align*}
 T(f) \circ \alpha &= 0_Y \circ \beta, \\
 p_X \circ \alpha &= p_X \circ 0_X \circ p_X \circ \alpha, \\
 T(f) \circ 0_X \circ p_X \circ \alpha &= 0_Y \circ f \circ p_X \circ \alpha = 0_Y \circ p_Y \circ 0_Y \circ \beta = 0_Y \circ \beta.
\end{align*}
As such $(C, p_X \circ \alpha, 0_Y \circ \beta)$ is a cone over $X \xrightarrow{f} Y \xleftarrow{p_Y} TY$ with cone maps $\alpha, 0_X \circ p_X \circ \alpha:C \to TX$. Because the naturality square for $p$ at $f$ is a prepullback, $\alpha = 0_X \circ p_X \circ \alpha$ and so the diagram
\[
\begin{tikzcd}
C \ar[drr, bend left = 20]{}{\beta} \ar[ddr, swap, bend right = 20]{}{p_X \circ \alpha} \ar[dr, dashed]{}{p_X \circ \alpha} \\
 & X \ar[r]{}{f} \ar[d, swap]{}{0_X} & Y \ar[d]{}{0_Y} \\
 & TX \ar[r, swap]{}{Tf} & TY
\end{tikzcd}
\]
commutes. Finally that all powers $T^m$ preserve these pullbacks follows mutatis mutandis by recognizing that since $T^m$ preserves the prepullback giving the $T$-immersive structure on $f$, we can argue as above after applying $T^m$ to the zero naturality square.
\end{proof}

We can use Theorem \ref{Thm: ZeroCarrable makes relative bundle a kernel in DBun} above to show that when each $\DBun(X)$ is enriched in the category $\mathbf{Ab}$ of Abelian groups, then $\theta_f$ is monic if and only if the relative tangent bundle is trivial in the sense that $T_{X/Y} \cong X$. This gives an alternative way to characterize the strong $T$-immersions in $p$-carrable Rosick{\'y} tangent categories as exactly the $T$-unramified maps. To prove this we recall the following well-known characterization of maps with trivial kernels coinciding with monics in $\mathbf{Ab}$-enriched categories with zero objects. We present a sketch of the lemma for the sake of completeness.

\begin{lemma}\label{Lemma: The homological algebra characcterization of trivial kernels in additive categories}
Let $\Ascr$ be a $\mathbf{Ab}$-enriched category with a zero object. The following are equivalent for a morphism $f:X \to Y$:
\begin{enumerate}[{\em (i)}] 
    \item The map $f$ is monic.
    \item For any map $g:Z \to X$, $f \circ g = 0$ if and only if $g = 0$.
    \item The equalizer $\operatorname{Eq}(f,0_{XY}:X \to Y)$ satisfies $\operatorname{Eq}(f,0_{XY}) \cong 0$.
\end{enumerate}
\end{lemma}
\begin{proof} For $(i) \implies (ii)$: If $f$ is monic then the equation $f \circ 0 = 0 = f \circ g$ implies that $g = 0$. For $(ii) \implies (iii)$: if there is a map making
\[
\begin{tikzcd}
Z \ar[r]{}{g} & X \ar[r, shift left = 1]{}{f} \ar[r, shift right = 1, swap]{}{0_{XY}} & Y
\end{tikzcd}
\]
commute if and only if $g$ is the zero map, then such a $g$ uniquely factors through the diagram:
\[
\begin{tikzcd}
Z \ar[rr]{}{g} \ar[dr, swap]{}{\bang_Z} & & X \\
 & 0 \ar[ur, swap]{}{\gnab_X}
\end{tikzcd}
\]
Thus $\operatorname{Eq}(f,0_{XY}) \cong 0$. Lastly for $(iii) \implies (i)$: assume for some $g, h:Z \to X$ that the diagram
\[
\begin{tikzcd}
Z \ar[r, shift left = 1]{}{g} \ar[r, swap, shift right = 1]{}{h} & X \ar[r]{}{f} & Y
\end{tikzcd}
\]
and note that we have that $f \circ g = f \circ h$. Then $f \circ (g-h) = f \circ g - f \circ h = 0$ and so $g - h$ must factor through $\operatorname{Eq}(f,0_{XY})$. However, as $\operatorname{Eq}(f,0) \cong 0$ this implies that $g-h = 0_{ZX}$ and so that $g = h$.
\end{proof}

\begin{theorem}\label{Thm: Rosckiy Timmersion iff Tunramified}
In a Rosick{\'y} tangent category $\Cscr$, if $f: X \to Y$ is $p$-carrable, then $f:X \to Y$ is a strong $T$-immersion if and only if $f$ is $T$-unramified.
\end{theorem}
\begin{proof}
Because $\Cscr$ is a Rosick{\'y} tangent category, for every object $X \in \Cscr_0$ the category $\DBun(X)$ has a zero object and is $\mathbf{Ab}$-enriched. From here the theorem follows by combining Theorem \ref{Thm: ZeroCarrable makes relative bundle a kernel in DBun} and Lemma \ref{Lemma: The homological algebra characcterization of trivial kernels in additive categories}.
\end{proof}

While every strong $T$-immersion is $T$-unramified, the converse is not necessarily true in an arbitrary tangent category. By Theorem \ref{Thm: Rosckiy Timmersion iff Tunramified} this distinction may be detected only by non-Rosick{\'y} tangent categories. We conclude this section with an example of a tangent category that has a $T$-unramified map that is not a $T$-immersion. 

\begin{example}\label{Example: CDC where T unram is not strong immersion} 
The category $\mathbf{CMon}$ of commutative monoids is a CDC where every map is linear, so for any map $f:X \to Y$, the differential $D[f]:X \times X \to Y$ satisfies $D[f](x,y) = f(y)$ and so the horizontal descent $\theta_f:X \times X \to X \times Y$ takes the form $(x,y) \mapsto (x,f(y))$. As such, by the explanation in Example \ref{Example: T monics for CDC} regarding linear morphisms, $f:M \to N$ is a strong $T$-immersion if and only if $f$ is monic, which in $\mathbf{CMon}$ means that $f$ is injective. On the other hand, by Example \ref{Example: Tunramified in CDC}, $f$ is $T$-unramified if and only if $\operatorname{Ker}(f) \cong 0$. So to give an example of a $T$-unramified map that is not a strong $T$-immersion, it suffices to give an example of a monoid morphism $f$ whose kernel is trivial but for which $f$ is not injective. So consider the addition on the natural numbers $\operatorname{add}_{\N}:\N \times \N \to \N$. Note $\operatorname{Ker}(\operatorname{add}_{\N})=0$; however $\operatorname{add}_{\N}$ is not injective since, for example, $\operatorname{add}_{\N}(2,3) = 2+3 = 5 = 3+2 = \operatorname{add}_{\N}(3,2).$ Thus $\operatorname{add}_{\N}$ is $T$-unramified but not a strong $T$-immersion. 
\end{example}

\section{Tangent Immersions}\label{Section: TImmersions}

We now move on from the more globally defined notion of $T$-immersions we just defined and studied and move to a notion of $T$-immersion which more directly mirrors the sprit of the definition as presented in $\SMan$. Namely, we define a flavour of $T$-immersion which is fundamentally based on the idea that it not just that immersions have monic horizontal descent but instead based on the idea that immersions are \emph{linear} and monic in the fibres of their differentials. Because this notion of immersion is more involved, i.e., it requires us to use the category of differential bundles, in order for it to make sense it requires more technology at hand. This in particular means that while asking for the horizontal descent to be only linear and monic, we of course need to know that we can talk about the horizontal descent in the first place. This necessitates that this weaker notion of immersion be definable in fewer tangent categories than the strong $T$-immersions of Definition \ref{Defn: Strong Timmersion}. However, we will see that the trade-off is that it is often much simpler to verify that any given $p$-carrable map satisfies this weaker notion of $T$-immersion.

\subsection{Definitions, Basic Properties, and Examples of Immersions}\label{Subsection: Immersions defns and examples} Because the fundamental notion of a $T$-immersion is that it is a monic \emph{linear} map of differential bundles over $X$, we need to find a way to reformulate what $T$-limits and $T$-colimits are in this more linearized sense. We cannot directly ask that a diagram in $\DBun(X)$ is a $T$-limit or $T$-colimit because the tangent functor $T$ does \emph{not} restrict to an endofunctor on $\DBun(X)$. Instead, by \cite[Lemma 2.5]{GeoffRobinBundle}, the tangent functor restricts to a functor $T_{\ast}:\DBun(X) \to \DBun(TX)$  which sends a differential bundle $q:E \to X$ to the differential bundle $Tq: TE \to TX$, which simply applies $T$ to the differential bundle structural maps, but where the lift also involves the canonical flip $c_E \circ T\lambda: TE \to T^2E$. 

In order to get a linear notion of $T$-limit or $T$-colimit, we thus are not simply led to ask that $T^m$ preserve the (co)limit of a certain diagram for all $m \in \N$. Instead we need to ask that each of the induced functors $T_{\ast}^m:\DBun(X) \to \DBun(T^mX)$ sends the (co)limit of a given diagram $D:I \to \DBun(X)$ to a limit of the correspondign diagram $T^m_{\ast} \circ D:I \to \DBun(T^mX)$ for all $m \in \N$. In order to make this precise, we present the definition below.

\begin{definition}\label{Defn: Bundle Linear Limit/Colimit}
Let $\Cscr$ be a tangent category and let $X$ be an object of $\Cscr$. We say that an object $L$ of $\DBun(X)$ is a \emph{linear $T$-(co)limit of $D$} of a diagram $D:I \to \DBun(X)$ if $L$ is a (co)limit of $D$ in $\DBun(X)$ and for all $m \in \N$ the induced functor $T^m_{\ast}:\DBun(X) \to \DBun(T^mX)$ preserves the (co)limit $L$.
\end{definition}

\begin{remark}\label{Remark: Linear Tmonic}
If we want to say that if there is a map $f:L \to E$ in $\DBun(X)$ for which the diagram
\[
\begin{tikzcd}
L \ar[r, equals] \ar[d, equals] & L \ar[d]{}{f} \\
L \ar[r, swap]{}{f} & E
\end{tikzcd}
\]
is a linear $T$-pullback then we will instead say that $f$ is a \emph{linear $T$-monomorphism} (in $\DBun(X)$, should specification be necessary). Similarly, if instead the diagram
\[
\begin{tikzcd}
L \ar[r]{}{f} \ar[d, swap]{}{f} & E \ar[d, equals] \\
E \ar[r, equals] & E
\end{tikzcd}
\]
is a linear $T$-pushout then we will instead say that $f$ is a \emph{linear $T$-epimorphism} (in $\DBun(X)$, should specification be necessary).
\end{remark}

As a sanity check we now verify that linear morphisms of differential bundles which happen to be $T$-(co)limits in $\Cscr$ are linear $T$-(co)limits in the corresponding bundle categories. Namely, we show that if $T$-limits factor through the category of differential bundles over $X$ then their restriction from the full tangent category $\Cscr$ to the differential bundle category $\DBun(X)$ gives rise to linear $T$-limits. This becomes relevant when comparing strong immersions with immersions as defined in Definition \ref{Defn: TImmersion} below.
\begin{proposition}\label{Prop: When Tlimits restrict to Linear Tlimits}
Let $D:I \to \Cscr$ be a diagram in a tangent category $\Cscr$ which factors as
\[
\begin{tikzcd}
I \ar[rr]{}{D} \ar[dr, swap]{}{D} & & \Cscr \\
 & \DBun(X) \ar[ur, swap]{}{\operatorname{Forget}}
\end{tikzcd}
\]
with corresponding $T$-limit $L$ and limit cone
\[
\varphi_i:L \to Di.
\]
If $L$ admits the structure of a differential bundle over $X$ and each map $\varphi_i:L \to D_i$ is a linear morphism of differential bundles then $L$ is a linear $T$-limit in $\DBun(X)$.
\end{proposition}
\begin{proof}
Let us first observe that since $L$ is a $T$-limit of $D$, for all $m \in \N$, the diagram $T^m \circ D$ has limit $T^mL$ with limit cone $T^m\varphi_i$. Additionally, by construction and \cite{GeoffRobinBundle} we have that $T^mL$ is a differential bundle over $T^mX$ and that each map $T^m\varphi_i$ is a linear map of differential bundles. In particular, this shows that $T^m \circ D$ factors as:
\[
\begin{tikzcd}
I \ar[rr]{}{T^m \circ D} \ar[dr, swap]{}{T^m \circ D} & & \Cscr \\
 & \DBun(T^mX) \ar[ur, swap]{}{\operatorname{Forget}}
\end{tikzcd}
\]
Additionally, note that each object $T^mL$ remains as a limit to $T^m \circ D$ and that the forgetful functor
\[
\operatorname{Forget}:\DBun(T^mX) \to \Cscr
\]
reflects limits. Consequently $T^mL$ is a limit of $T^m \circ D:I \to \DBun(T^mX)$ and hence $T^m_{\ast}:\DBun(X) \to \DBun(T^mX)$ sends $L$ to the limit of $T^m \circ D$. Thus $L$ is a linear $T$-limit in $\DBun(X)$.
\end{proof}
We now can dualize Proposition \ref{Prop: When Tlimits restrict to Linear Tlimits} to derive the corresponding result for linear $T$-colimits.
\begin{proposition}\label{Prop: When Tcolimits restrict to Linear Tcolimits}
Let $D:I \to \Cscr$ be a diagram in a tangent category $\Cscr$ which factors as
\[
\begin{tikzcd}
I \ar[rr]{}{D} \ar[dr, swap]{}{D} & & \Cscr \\
 & \DBun(X) \ar[ur, swap]{}{\operatorname{Forget}}
\end{tikzcd}
\]
with corresponding $T$-colimit $C$ and colimit cocone
\[
\varphi_i:Di \to C
\]
If $C$ admits the structure of a differential bundle over $X$ and each map $\varphi_i:D_i \to C$ is a linear morphism of differential bundles then $C$ is a linear $T$-colimit in $\DBun(X)$.
\end{proposition}
Some immediate corollaries to Propositions \ref{Prop: When Tlimits restrict to Linear Tlimits} and \ref{Prop: When Tcolimits restrict to Linear Tcolimits} we use to build (classes of) examples of linear $T$-limits are given below.
\begin{corollary}\label{Cor: Tlinear epics and monics in DBun}
Let $\varphi: q \to r$ be a linear morphism of differential bundles over $X$. If $\varphi$ is $T$-monic (resp. $T$-epic) in $\Cscr$ then $\varphi$ is a linear $T$-monomorphism (resp. a linear $T$-epimorphism) in $\DBun(X)$.
\end{corollary}

\begin{corollary}\label{Cor: When thetaf is tmonic is tlinear monic}
If $f:X \to Y$ is a $p$-carrable map and if $\theta_f$ is $T$-monic (resp. $T$-epic) then $\theta_f$ is a linear $T$-monomorphism (resp. linear $T$-epimorphism). 
\end{corollary}

\begin{corollary}\label{Cor: When f tmonic then thetaf linear monic}
If $f:X \to Y$ is $T$-monic and $p$-carrable then $\theta_f:TX \to f^{\ast}(TY)$ is a linear $T$-monomorphism.
\end{corollary}
\begin{proof}
Note that $\theta_f$ is always a linear morphism of differential bundles. By Proposition \ref{Prop: f monic and thetaf monic iff Tmonic} $f$ is $T$-monic if and only if $f$ is monic and $\theta_f$ is $T$-monic, so from here an application of Corollary \ref{Cor: When thetaf is tmonic is tlinear monic} proves the corollary.
\end{proof}


We now are led to define immersions in general tangent categories. As we will see below, these are the morphisms which have the property that $\theta_f$ is a linear $T$-monomorphism. Afterwards, we will prsent some basic facts regarding $T$-immersions and then present some examples before moving on to study the relationship between $T$-immersions and unramified morphisms in Section \ref{Subsection: Immersions and Unramification}.

\begin{definition}\label{Defn: TImmersion}
Let $\Cscr$ be a tangent category and let $f:X \to Y$ be a $p$-carrable morphism in $\Cscr$. We say that $f$ is a \emph{$T$-immersion} if the morphism $\theta_f:TX \to f^{\ast}(TY)$ is a linear $T$-monomorphism in $\DBun(X)$.
\end{definition}
\begin{remark}
Our definition for $T$-immersions above has forced us to require that immersions $f:X \to Y$ are $p$-carrable. While this is strictly less general than the definition of $T$-immersion which appeared in an earlier version of this paper, we think that this is a \emph{good} definition because it encodes exactly the local linear injectivity that one uses to \emph{define} immersions in $\SMan$ in a minimal setting, as the injectivity of the maps $T_xf:T_xX \to T_{f(x)}Y$ only every see the local injectivity of the differential bundle map $\theta_f:TX \to f^{\ast}(TY)$ and not the rest of the bundle category $\DBun(\Cscr)$ or even the full tangent category $\Cscr$.
\end{remark}
A straightforward, but important for sanity reasons, observation given the results above shows that every strong immersion is an immersion. We can also conclude that every $T$-monic is \emph{also} an immersion, as is expected.
\begin{proposition}\label{Prop: Strong immersion is immersion}
Let $f:X \to Y$ be a $p$-carrable strong immersion in a tangent category. Then $f$ is a $T$-immersion.
\end{proposition}
\begin{proof}
As $\theta_f$ is a linear morphism of differential bundles, this follows from Corollary \ref{Cor: When thetaf is tmonic is tlinear monic}.
\end{proof}
\begin{proposition}\label{Prop: Tmonic is immersion}
Let $f:X \to Y$ be a $T$-monic morphism which is $p$-carrable. Then $f$ is a $T$-immersion.
\end{proposition}
\begin{proof}
Apply Corollary \ref{Cor: When f tmonic then thetaf linear monic}.
\end{proof}

We now show that $T$-immersions are stable with respect to composition.

\begin{proposition}\label{Proposition: Timmersions composition}
$T$-immersions are composition stable. That is, if $f:X \to Y$ and $g:Y\to Z$ are $T$-immersions and if $g \circ f$ is $p$-carrable then $g \circ f$ is a $T$-immersion.
\end{proposition}
\begin{proof}
Begin by noting that each of $f$, $g$, and $g \circ f$ are $p$-carrable. Thus, by Lemma \ref{Lemma: Thetaf and composition}, the diagram
\[
\begin{tikzcd}
TX \ar[d, swap]{}{\theta_{g \circ f}} \ar[r]{}{\theta_f} & f^{\ast}TY \ar[d]{}{f^{\ast}\theta_g} \\
(g \circ f)^{\ast}TZ & f^{\ast}(g^{\ast}TZ) \ar[l, swap]{}{\cong} \ar[l]{}{\gamma}
\end{tikzcd}
\]
commutes in $\DBun(X)$ and hence allows us to write $\theta_{g \circ f} = \gamma \circ f^{\ast}\theta_g \circ \theta_f$. Because $\gamma$ and $\theta_f$ are monic in $\DBun(X)$, it suffices to prove that $f^{\ast}\theta_g$ is monic in $\DBun(X)$ in order to complete the proof it suffices to verify that $f^{\ast}\theta_g$ is monic. However, this is immediate because $\theta_g$ is monic in $\DBun(Y)$ and because pullbacks send monics to monics.
\end{proof}

While we do not show it here, Proposition \ref{Prop: Immersion has analogous prepullvack property} says that $T$-immersions have an analogous property to the prepullbacks which define strong $T$-immersions (cf.~ Definition \ref{Defn: Strong Timmersion}). Instead, for the moment, we proceed to examine $T$-immersions in our working example categories of interest before moving on in the next sub-subsection to study a middle strength of immersion which sits between $T$-immersions and strong $T$-immersions.

\begin{example}\label{Example: Immersions in SMan}
In $\SMan$ a morphism $f:X \to Y$ is a $T$-immersion if and only if the map
\[
\theta_f:TX \to f^{\ast}TY, \quad \left(x,\overrightarrow{v}\right) \mapsto \left(x,D[f](x)\overrightarrow{v}\right)
\]
is monic in $\DBun(X)$. However, this is equivalent to asking that for all $x \in X$ the map
\[
\id_x \times D[f](x):\lbrace x \rbrace \times T_{x}X \to \lbrace x \rbrace \times T_{f(x)}Y
\]
is monic. Once again, as $D[f](x)= T_xf$, this is further equivalent to asking that each linear map $T_xf:T_xX \to T_{f(x)}Y$ is injective and hence that $f$ is an immersion in the usual sense. 
\end{example}
\begin{example}\label{Example: Immersions in CAlg}
If $R$ is a commutative rig, then a map $f$ is a $T$-immersion if and only if $f$ is monic. To see this let $f:A \to B$ be a morphism of commutative $R$-algebras. In order that $f:A \to B$ be a $T$-immersion, we need the map
\[
\theta_f:TA \to f^{\ast}TB
\]
be monic in $\DBun(A)$. As we have that $TA = A[\epsilon],$ $f^{\ast}TB = A \ltimes B$, and $\theta_f = \id \ltimes f$, we see that the map
\[
A[\epsilon] \to A\ltimes B, \quad a+x\epsilon \mapsto (a,f(x))
\]
is monic in $\DBun(A)$ if and only if the underling map of $A$-modules $f:A \to B$ is monic in $\Mod{A}$ by virtue of the equivalence $\DBun(A) \simeq \Mod{A}$. As morphisms of $A$-modules are monic if and only if $f$ is injective, we finally get that $f$ is a $T$-immersion if and only if $f$ is injective as a map of $R$-algebras..
\end{example}
\begin{example}\label{Example: Immersions in CAlgOp}
If $R$ is a commutative rig then a morphism $f:A \to B$ of commutative $R$-algebras is a $T$-immersion if and only if $f$ is formally unramified in the sense that $\Kah{B}{A} \cong 0$. To see this recall from the opposite equivalence $\DBun(B) \simeq \Mod{B}^{\op}$ that the morphism
\[
\theta_f:T_{B/R} \to f^{\ast}T_{A/R}
\]
is monic if and only if the corresponding morphism
\[
\Sym{A}(u_f):\Sym{B}(\Kah{A}{R} \otimes_A B) \to \Kah{B}{R}
\]
is epic in $\DBun(B)^{\op}$. But as $u_f$ arises from the cotangent sequence
\[
\begin{tikzcd}
\Kah{A}{R} \otimes_A B \ar[r]{}{u_f} & \Kah{B}{R} \ar[r]{}{} & \Kah{B}{A} \ar[r] & 0
\end{tikzcd}
\]
we see that $u_f$ is epic if and only $u_f$ is surjective as a morphism of $B$-modules. But as we have an isomorphism of $B$-modules
\[
\Kah{B}{A} \cong \frac{\Kah{B}{R}}{\operatorname{Im}(u_f)}
\]
for $\operatorname{Im}(u_f)$ the congruence generated by the image of $u_f$. But as $u_f$ is surjective, we get $\operatorname{Im}(u_f) = \Kah{B}{R}$ and hence that $\Kah{B}{A} \cong 0$.
\end{example}
\begin{example}\label{Example: Immersions in Sch}
If $S$ is a scheme, then $f:X\to Y$ is a $T$-immersion if and only if $f$ is formally unramified. To see this we recall from the equivalence of categories $\DBun(X) \simeq \QCoh(X)^{\op}$ for any $S$-scheme $X$, when considering the tangent-categorical relative cotangent sequence
\[
\begin{tikzcd}
X \ar[r] & T_{X/Y} \ar[r]{}{\iota_f} & TX \ar[r]{}{\theta_f} & f^{\ast}TY
\end{tikzcd}
\]
we have that $\theta_f$ is monic if and only if the corresponding cotangent sequence
\[
\begin{tikzcd}
f^{\ast}\Kah{Y}{S} \ar[r]{}{u_f} & \Kah{X}{S} \ar[r] & \Kah{X}{Y} \ar[r] & 0
\end{tikzcd}
\]
of quasi-coherent sheaves on $X$ has that $u_f$ is an epimorphism. But as the cotangent sequence is an exact sequence, $u_f$ is an epimorphism if and only if $\Kah{X}{Y} \cong 0$. Because $f$ is a formally unramified morphism of schemes if and only if $\Kah{X}{Y} \cong 0$ by \cite[Proposition 17.2.1]{EGA44}.
\end{example}

\begin{example}\label{Example: Immersions in CDCs} In a CDC, for a map of type $f: C \times X \to Y$, we can define its partial tangent $T^C f: C \times TX \to TY$ by inserting zeroes into $Tf$. As such, we can talk about a map $f: C \times X \to Y$ being (linear) $T$-monic in its second argument by simply saying $(T^C)^n f: C \times (T^C)^n X \to (T^C)^n Y$ is monic in the (linear) simple slice category. Therefore, a map $f: X \to Y$ in a CDC is a $T$-immersion  if and only if $D[f]: X \times X \to X$ is linear $T^X$-monic, which we interpret as saying that all higher order partial derivatives of $f$ are monic on linear maps in their linear arguments. 
\end{example}

\subsection{Tangent Immersions and Unramified Morphisms}\label{Subsection: Immersions and Unramification} Above we showed that every strong $T$-immersion was unramified. We now show the same is true for $T$-immersions.  

\begin{proposition}\label{Prop: Immersions are unramified} If $f:X \to Y$ is a $T$-immersion in a tangent category then $f$ is $T$-unramified.
\end{proposition}
\begin{proof}
Because we have the section/retraction pair
\[
\begin{tikzcd}
X \ar[r]{}{0_{X/Y}} \ar[dr, equals] & T_{X/Y} \ar[d]{}{p_{X/Y}} \\
 & X
\end{tikzcd}
\]
in order to prove that $T_{X/Y} \cong X$ it suffices to prove instead that $p_{X/Y} \circ 0_{X/Y} = \id_{T_{X/Y}}$. Additionally, recall by Theorem \ref{Thm: ZeroCarrable makes relative bundle a kernel in DBun}the diagram
\[
\begin{tikzcd}
T_{X/Y} \ar[r]{}{\iota_f} & TX \ar[rr, shift left = 1]{}{\theta_f} \ar[rr, shift right = 1, swap]{}{f^{\ast}(0_Y) \circ p_X} & & f^{\ast}(TY)
\end{tikzcd}
\]
is an equalizer in $\DBun(X)$. In particular, this implies that $\iota_f$ is monic and also that the diagram
\[
\begin{tikzcd}
T_{X/Y} \ar[r]{}{\iota_f} \ar[d, swap]{}{p_{X/Y}} & TX \ar[d]{}{\theta_f} \\
X \ar[r, swap]{}{f^{\ast}(0_Y)} & f^{\ast}(TY)
\end{tikzcd}
\]
commutes in $\DBun(X)$. In particular, this shows that $\theta_f \circ \pi_0$ is monic in $\DBun(X)$. Since $\pi_0$ and $\theta_f$ are linear differential bundle morphisms the diagrams
\[
\begin{tikzcd}
T_{X/Y} \ar[r]{}{\pi_0} & TX \\
X \ar[u]{}{0_{X/Y}} \ar[ur, swap]{}{0_X}
\end{tikzcd}\quad
\begin{tikzcd}
TX \ar[r]{}{\theta_f} & f^{\ast}(TY) \\
X \ar[u]{}{0_X} \ar[ur, swap]{}{f^{\ast}(0_Y)}
\end{tikzcd}
\]
commute in $\DBun(X)$. This allows us to compute
\begin{align*}
    \theta_f \circ \iota_f \circ \id_{T_{X/Y}} &= \theta_f \circ \iota_f = f^{\ast}(0_Y) \circ p_{X/Y} = \theta_f \circ 0_X \circ p_{X/Y} = \theta_f \circ \iota_f \circ 0_{X/Y} \circ p_{X/Y},
\end{align*}
which, because $\theta_f \circ \iota_f$ is monic in $\DBun(X)$, allows us to further deduce that $\id_{T_{X/Y}} = 0_{X/Y} \circ p_{X/Y}$. Thus we get that the diagrams
\[
\begin{tikzcd}
X \ar[dr, equals] \ar[r]{}{0_{X/Y}} & T_{X/Y} \ar[d]{}{p_{X/Y}} \\
 & X
\end{tikzcd}\quad
\begin{tikzcd}
T_{X/Y} \ar[r]{}{p_{X/Y}} \ar[dr, equals] & X \ar[d]{}{0_X} \\
 & T_{X/Y}
\end{tikzcd}
\]
both commute in $\DBun(X)$ and hence gives the isomorphism $X \cong T_{X/Y}$ in $\DBun(X)$. Finally, passing through the forgetful functor $\DBun(X) \to \Cscr$  gives that $X \cong T_{X/Y}$ in $\Cscr$ and hence that $f$ is unramified, i.e., that the diagram
\[
\begin{tikzcd}
X \ar[r]{}{f} \ar[d, swap]{}{0_X} & Y \ar[d]{}{0_Y} \\
TX \ar[r, swap]{}{Tf} & TY
\end{tikzcd}
\]
is a pullback.
\end{proof}

We now show that being a $T$-immersion has an analogous property to being a prepullback, save that we must restrict our attention to the linear bundle morphisms.

\begin{proposition}\label{Prop: Immersion has analogous prepullvack property}
If $f:X \to Y$ is a $T$-immersion then for any differential bundle $q:E \to X$ and linear differential bundle morphism $\binom{\varphi}{f}:q \to p_Y$, there is at most one linear differential bundle map from $q$ to $p_X$ rendering the diagram
\[
\begin{tikzcd}
q \ar[rr]{}{\binom{\varphi}{f}} \ar[dr, dashed]{}{\exists\leq 1} & & p_Y \\
 & p_X \ar[ur, swap]{}{\binom{Tf}{f}}
\end{tikzcd}
\]
commutative in $\DBun(\Cscr)$.
\end{proposition}
\begin{proof}
Assume that we have linear morphisms $\rho, \psi \in \DBun(X)(E,TX)$ which renders the diagram
\[
\begin{tikzcd}
E \ar[drr, bend left = 20]{}{\varphi} \ar[ddr, swap, bend right = 20]{}{q} \ar[dr, shift left = 1]{}{\psi} \ar[dr, swap, shift right = 1]{}{\rho} \\
 & TX \ar[r]{}{Tf} \ar[d, swap]{}{p_X} & TY \ar[d]{}{p_Y} \\
 & X \ar[r, swap]{}{f} & Y
\end{tikzcd}
\]
commutative. Then, by the universal property of the pullback we find that there are maps $\tilde{\psi}, \tilde{\varphi}:E \to f^{\ast}(TY)$ which render both diagrams
\[
\begin{tikzcd}
E \ar[drr, bend left = 20]{}{\tilde{\rho}} \ar[ddr, swap, bend right = 20]{}{q} \ar[dr, shift left = 1]{}{\psi} \ar[dr, swap, shift right = 1]{}{\rho} \ar[drrr, shift left = 1, bend left = 20]{}{\varphi} \\
 & TX \ar[r]{}{\theta_f} \ar[d, swap]{}{p_X} & f^{\ast}TY \ar[d]{}{\pi_0} \ar[r]{}{\pi_1} & TY \ar[d]{}{p_Y} \\
 & X \ar[r, equals] & X \ar[r, swap]{}{f} & Y
\end{tikzcd}\quad
\begin{tikzcd}
E \ar[drr, bend left = 20]{}{\tilde{\psi}} \ar[ddr, swap, bend right = 20]{}{q} \ar[dr, shift left = 1]{}{\psi} \ar[dr, swap, shift right = 1]{}{\rho} \ar[drrr, shift left = 1, bend left = 20]{}{\varphi} \\
 & TX \ar[r]{}{\theta_f} \ar[d, swap]{}{p_X} & f^{\ast}TY \ar[d]{}{\pi_0} \ar[r]{}{\pi_1} & TY \ar[d]{}{p_Y} \\
 & X \ar[r, equals] & X \ar[r, swap]{}{f} & Y
\end{tikzcd}
\]
commute in $\DBun(\Cscr)$. Observe also that as the equations
\[
\pi_0 \circ \tilde{\rho} = q = \pi_0 \circ \tilde{\psi},\qquad \pi_1 \circ \tilde{\rho} = \varphi = \pi_1 \circ \tilde{\psi}
\]
both hold, we get that $\tilde{\rho} = \tilde{\psi}$. Additionally, a routine computation allows us to verify that both $\tilde{\rho}$ and $\tilde{\psi}$ are linear morphisms of differential bundles over $X$. This allows us to use the fact that $\theta_f$ is monic in $\DBun(X)$ in order to deduce that $\rho = \psi$ and hence that there is at most one morphism rendering
\[
\begin{tikzcd}
E \ar[dr, dashed]{}{\exists\,\leq 1} \ar[ddr, swap, bend right = 20]{}{q} \ar[drr, bend left = 20]{}{\varphi} \\
 & TX \ar[r]{}{Tf} \ar[d, swap]{}{p_X} & TY \ar[d]{}{p_Y} \\
 & X \ar[r, swap]{}{f} & Y
\end{tikzcd}
\]
commutative in $\DBun(\Cscr)$.
\end{proof}

Now since strong $T$-immersions are $T$-immersions, Example \ref{Example: CDC where T unram is not strong immersion} gives an a tangent category where not every $T$-immersion is $T$-uramified. The issue lies with the lack of \textit{negatives}. So we now conclude this section by proving a very useful and somewhat surprising result: in Rosick{\'y} tangent categories (tangent categories with negatives), all flavours of $T$-immersion coincide with $T$-unramified morphisms. In particular this implies that every $T$-immersion is also a strong $T$-immersion. This is surprising because the formal structure of immersions, at all levels, seems to come just shy of being able to show that the kernel $T_{X/Y}$ vanishing is equivalent to $\theta_f$ being monic. However, when one has negatives, being monic is equivalent to having trivial kernel and so we can reduce $\theta_f$ being a linear $T$-monomorphism to the naturality square for the transformation $0$ at $f$ being a $T$-pullback. That said, when one does \emph{not} have negatives, 

\begin{theorem}\label{Thm: If C Rosicky then immersion iff unrmaified}
If $\Cscr$ is a Rosick{\'y} tangent category then a $p$-carrable morphism $f:X \to Y$ is a $T$-immersion if and only if $f$ is $T$-unramified. In particular, in a Roscik{\'y} tangent category a $p$-carrable morphism $f$ is a $T$-immersion if and only if it is a strong $T$-immersion.
\end{theorem}
\begin{proof}
Recall that $f$ is a $T$-immersion precisely when $\theta_f$ is a linear $T$-monomorphism in $\DBun(X)$. But as $\Cscr$ is Rosick{\'y} and we by convention only study differential bundles with negatives over Rosick{\'y} tangent categories, $\DBun(X)$ is not merely $\mathbf{CMon}$-enriched but also $\mathbf{Ab}$-enriched. As each category $\DBun(T^mX)$ has $T^mX$ as a zero object, applying Lemma \ref{Lemma: The homological algebra characcterization of trivial kernels in additive categories} to each monic $T^m\theta_f$ implies that asking if there is an isomorphism
\[
T^m(T_{X/Y}) \overset{?}{\cong} T^mX
\]
in $\DBun(T^mX)$ is equivalent to asking if $T^m\theta_f$ is monic in $\DBun(T^mX)$ by virtue of the relative cotangent sequence
\[
\begin{tikzcd}
T^mX \ar[r] & T^m(T_{X/Y}) \ar[r]{}{T^m\iota_f} & T^{m+1}X \ar[r]{}{T^m\theta_f} & T^m(f^{\ast}TY)
\end{tikzcd}
\]
exhibiting $T^m(T_{X/Y})$ as a kernel of $T^m\theta_f$. Conseuqently, $f$ is $T$-unramified if and only if $f$ is a $T$-immersion. For the final claim of the theorem, we note that by Theorem \ref{Thm: Rosckiy Timmersion iff Tunramified}, $f$ is a $T$-immersion if and only if $f$ is a strong $T$-immersion. 
\end{proof}

\section{Tangent Submersions}\label{Section: TSubmersions}

Submersions are some of the most important morphisms in differential geometry because they are the morphisms $f:X \to Y$ of smooth manifolds which meet every other morphism $g:Z \to Y$ transversely (cf.~ \cite{KMSDiffGeo}). They appear also in studying equivariant differential geometry (cf.~ \cite[Section 5]{PronkVooys}) and provide a class of morphisms $\Dcal$ in $\SMan$ for which not only do pullback functors $f^{\ast}:\SMan_{/Y} \to \SMan_{/X}$ given by categorical pullback
\[
\begin{tikzcd}
f^{\ast}(M) \ar[r]{}{\pi_1} \ar[d, swap]{}{\pi_0} & M \ar[d]{}{g} \\
X \ar[r,swap]{}{f} & Y
\end{tikzcd}
\]
exist for any submersion $f:X \to Y$ by \cite{KMSDiffGeo}, but these functors also have the property that they induce natural isomorphisms
\[
\alpha:T_{(-)/Y} \circ f^{\ast} \xRightarrow{\cong} f^{\ast} \circ T_{(-)/X}.
\]
Submersions also enjoy the property that, when given a submersion $f:X \to Y$ and a pullback
\[
\begin{tikzcd}
X \times_Y M \ar[r]{}{\pi_1} \ar[d, swap]{}{\pi_0} & M \ar[d]{}{g} \\
X \ar[r, swap]{}{f} & Y
\end{tikzcd}
\]
this pullback is actually a $T$-pullback \emph{and} $Tf$ is also a submersion. Additionally, there is a Grothendieck topology $J$ on $\SMan$ whose covering morphisms are surjective submersions which is important in various sheaf-theoretic approaches to studying $\SMan$. Unfortunately, however, generalizing submersions to tangent categories is not at all straightforward because of various technical issues which arise due to the level of generality at which tangent categories exist.

Let us make this more explicit. In \cite[Subsection 1.6]{AintablianBlohmann} the authors point out arguably the four most important properties that submersions $f:X \to Y$ in $\SMan$ enjoy:
\begin{enumerate}
    \item If $x \in X$ then there exists an open $V \subseteq Y$ with $f(x) \in V$ and with the property that there is a function $s:V \to X$ for which $f \circ s =\id_{V}$. That is, every point in the image of $f$ has the property that there is a local section of $f$ defined around $f(x)$. That is, for any $x \in X$ there is an open submanifold inclusion $j:V \to Y$ with $f(x) \in V$ and a morphism $s:V \to X$ for which
    \[
    \begin{tikzcd}
    V \ar[dr]{}{j} \ar[d, swap]{}{s} \\
    X \ar[r, swap]{}{f} & Y
    \end{tikzcd}
    \]
    commutes in $\SMan$.
    \item If $g:M \to Y$ is a morphism then the pullback $X \times_Y M$ exists in $\SMan$.
    \item Every natural morphism $T(X \times_Y M) \to TX \times_{TY} TM$ is an isomorphism.
    \item The tangent $Tf:TX \to TY$ is a submersion.
\end{enumerate}
These properties, unfortunately, need not all hold in arbitrary tangent categories. Work has been done historically in terms of what are called $T$-display maps (cf.~ \cite{GeoffRobinBundle}, \cite{BenVectorBundles}, \cite{JonathanThesis}, and \cite{GeoffMarcelloTSubmersionPaper}) as generalizations of submersions\footnote{It is worth observing that the original ways of framing what it means for a morphism to have transverse meetings (cf.~ \cite{JonathanThesis}, \cite{BenVectorBundles}) are encoded and captured in what we call \emph{split $T$-submersions} in this paper. These are studied in detail in Section \ref{Section: Split Submersions}.}. The idea here is that $T$-display morphism in a tangent category $\Cscr$ is \emph{exactly} a morphism $f:X \to Y$ which satisfies Statements $(2) - (4)$ above (suitably refined by saying, of course, that $f:X \to Y$ admits all $T$-pullbacks and that $T^mf$ also preserves and admits all $T$-pullbacks). However, approaching submersions in this way need not capture all possible generalizations and ways of extending submersions. For instance, when one is given nothing more than a $T$-display morphism $f$ in the abstract, it is not at all clear that the local surjectivity of $Tf$ need hold. That is, it is not at all clear that when $f$ is both $p$-and-$0$-carrable, the diagram of differential bundles
\[
\begin{tikzcd}
 X \ar[r] & T_{X/Y} \ar[r]{}{\iota_f} & TX \ar[r]{}{\theta_f} & f^{\ast}(TY) \ar[r] & X
\end{tikzcd}
\]
has the property that $\theta_f$ is a cokernel of $\iota_f$.

In this paper, we take the perspective\footnote{Inspired by coversations with Robin Cockett, Geoff Cruttwell, Richard Garner, and Marcello Lanfranchi.} that it is the \emph{locally surjective} differential that fundamentally defines what it means to be a submersion and that the properties of being $T$-display are something quite special to the category $\SMan$  --- in particular, being $T$-display is a \emph{property} of being a submersion in $\SMan$ and not the \emph{definition} of being a submersion. The idea we take behind what submersions should be at the full level of tangent categories is that because submersions $f$ in $\SMan$ are maps for which \emph{definitionally first and foremost} $T_xf$ has the property that it is a surjective linear map for all $x \in X.$ We will see below that taking this perspective requires us to both make some sacrifices on the consequences of being a submersion (cf.~ Propositions \ref{Prop: Section Sub: When submersions compose} and \ref{Prop: Section Sub: When submersions pull back} give conditions when submersions in tangent categories --- in the sense of Definition \ref{Defn: TSubmersion} --- are stable under composition and pullback), but that some properties (such as submersions being stable under application of the functors $T^m$) still remain true; cf.~ Proposition \ref{Prop: Section Sub: Submersions are Tstable}. Additionally, in order to focus on and really express the \emph{linearity} of submersions we will need to require that submersions are $p$-and-$0$-carrable.

\subsection{Definitions and Examples of Tangent Submersions}\label{Subsection:  Tsubmersion Defns, and Examples}

As we have seen prior, we are taking the perspective that it is the local surjective linearity of submersions which is their defining feature. Because in $\SMan$ this is equivalent to asking that
\[
\theta_f:TX \to f^{\ast}(TY)
\]
is a regular epimorphism in $\DBun(X)$, we need to be able to have access to the horizontal descent to express this. However, we do not \emph{just} want $\theta_f$ to be a regular epimorphism. We also want to know that it is precisely the coequalizer of the zero map and the kernel of $\theta_f$, i.e.~ and in light of Theorem \ref{Thm: ZeroCarrable makes relative bundle a kernel in DBun}, that the diagram
\[
\begin{tikzcd}
T_{X/Y} \ar[r]{}{\iota_f} \ar[d, swap]{}{p_{X/Y}} & TX \ar[d]{}{\theta_f} \\
X \ar[r, swap]{}{f^{\ast}0_Y} & f^{\ast}TY
\end{tikzcd}
\]
is a pushout in $\DBun(X)$. Finally, as $\DBun(X)$ need not carry a tangent structure which coincides with the tangent structures on $\DBun(\Cscr)$ and $\Cscr$, we also need to encode tangential-stability of the pushouts above by asking that said pushout be a linear $T$-pushout. This leads us to the definition of $T$-submersions in tangent categories.

\begin{definition}\label{Defn: TSubmersion}
Let $\Cscr$ be a tangent category. Then a $p$-carrable and $0$-carrable morphism $f:X \to Y$ is a \emph{$T$-submersion} if the horizontal descent $\theta_f$ makes the diagram
\[
\begin{tikzcd}
    T_{X/Y} \ar[r, shift left = 1ex]{}{\iota_f} \ar[r, swap, shift right = 1ex]{}{0_X \circ p_{X/Y}} & TX \ar[r]{}{\theta_f} & f^{\ast}(TY)
\end{tikzcd}
\]
a linear $T$-coequalizer in $\DBun(X)$. That is, $\theta_f$ is a regular $T$-epimorphism in $\DBun(X)$.
\end{definition}

It is simultaneously instructive and useful to reframe $T$-submersions in terms of the  cotangent sequence. Recall that for all $0$-carrable and $p$-carrable morphisms $f:X \to Y$, the relative cotangent sequence of $f$ is the diagram
\[
\begin{tikzcd}
X \ar[r] & T_{X/Y} \ar[r]{}{\iota_f} & TX \ar[r]{}{\theta_f} & f^{\ast}(TY)
\end{tikzcd}
\]
in $\DBun(X)$. Furthermore, this diagram is an exact sequence in $\DBun(X)$ in the sense that $\iota_f = \operatorname{Ker}(\theta_f)$ in $\DBun(X)$. When we ask for $\theta_f$ to be the linear $T$-coequalizer of $\iota_f$ and the zero morphism $0_X \circ p_{X/Y}$ in $\DBun(X)$, we then get that $\theta_f = \operatorname{Coker}(\iota_f)$ and hence that the relative cotangent sequence extends into an exact sequence
\[
\begin{tikzcd}
X \ar[r] & T_{X/Y} \ar[r]{}{\iota_f} & TX \ar[r]{}{\theta_f} & f^{\ast}(TY) \ar[r] & X
\end{tikzcd}
\]
in $\DBun(X)$. While we record below that to ask for this sequence to be a $T$-exact sequence is equivalent to asking for $f$ to be a submersion, it is worth noting that \emph{defining} $T$-submersions in terms of the relative cotangent sequence is a perspective whose importance was first recognized by G.~ Cruttwell and M.~ Lanfranchi, which they use to study and examine connections in tangent categories \cite{GeoffMarcelloTSubmersionPaper}.

\begin{proposition}\label{Prop: Section TSub: Submersion iff Rel Cotan is T Exact Sequence}
Let $f:X \to Y$ be a $p$-carrable and $0$-carrable morphism in a tangent category $\Cscr.$ Then $f$ is a $T$-submersion if and only if for all $m \in \N$ the diagram
\[
\begin{tikzcd}
T^mX \ar[r] & T^m(T_{X/Y}) \ar[r]{}{T^m\iota_f} & T^{m+1}(X) \ar[r]{}{T^m(\theta_f)} & T^m(f^{\ast}(TY)) \ar[r] & T^mX
\end{tikzcd}
\]
is an exact sequence in $\DBun(T^mX)$.
\end{proposition}
\begin{proof}
$\implies:$ If $f:X \to Y$ is a $T$-submersion, then by definition $\theta_f$ is a linear $T$-coequalizer of $\theta_f$ and $0_X \circ p_{X/Y}$ in $\DBun(X)$. In particular, for any $m \in \N$ the functor $T_{\ast}^{m}:\DBun(X) \to \DBun(T^mX)$ renders $T^m(\theta_f)$ as the coequalizer of $T^m(\theta_f)$ and $T^m(0_X \circ p_{X/Y})$. Thus, as $T^m\iota_f$ is always the kernel of $T^m\theta_f$ in $\DBun(T^mX)$, the diagram
\[
\begin{tikzcd}
T^mX \ar[r] & T^m(T_{X/Y}) \ar[r]{}{T^m\iota_f} & T^{m+1}(X) \ar[r]{}{T^m(\theta_f)} & T^m(f^{\ast}(TY)) \ar[r] & T^mX
\end{tikzcd}
\]
is an exact sequence in $\DBun(T^mX)$.

$\impliedby:$ If the diagram
\[
\begin{tikzcd}
T^mX \ar[r] & T^m(T_{X/Y}) \ar[r]{}{T^m\iota_f} & T^{m+1}(X) \ar[r]{}{T^m(\theta_f)} & T^m(f^{\ast}(TY)) \ar[r] & T^mX
\end{tikzcd}
\]
is an exact sequence in $\DBun(T^mX)$ for all $m \in \N$, then we have by construction that 
\[
\operatorname{coker}(T^m\iota_f) = \theta_f = \operatorname{coeq}\left(T^m\iota_f,0_{T^mX} \circ p_{T^mX/T^mY}\right) = \operatorname{coeq}\left(T^m\iota_f,T^m(0_X \circ p_{X/Y})\right).
\]
Thus we have that $\theta_f$ is a regular $T$-epimorphism \emph{and} that $T_{\ast}^m$ carries $\theta_f$ to a regular $T$-epimorphism in $\DBun(T^mX)$ for all $m \in \N$. In particular, it follows by definition that $f$ is a $T$-submersion.
\end{proof}

\begin{example}\label{Example: TSubmersions in SMan}
In $\SMan$, a morphism $f:X \to Y$ is a $T$-submersion if and only if it is a submersion in the usual sense. Recall that in the cotangent exact sequence of $f$,
\[
\begin{tikzcd}
X \ar[r] & T_{X/Y} \ar[r]{}{\iota_f} & TX \ar[r]{}{\theta_f} & f^{\ast}(TY) \ar[r] & X
\end{tikzcd}
\]
we have that $\iota_f$ is the inclusion of the sub-bundle
\[
T_{X/Y} = \left\lbrace \left(x,\overrightarrow{v}\right) \in TX \; : \; D[f](x)\overrightarrow{v}= \overrightarrow{0}\right\rbrace
\]
and that by Example \ref{Example: What  is thetaf in SMan} we have $\theta_f$ defined as the map
\[
\theta_f:TX \to X \times_Y TY, \quad \left(x,\overrightarrow{v}\right) \mapsto \left(x, D[f](x)\overrightarrow{v}\right).
\]
From this we see that $\theta_f$ is a cokernel of the inclusion $\iota_f$ if and only if for every point $x \in X$, $D[f](x)$ is a surjective morphism of vector spaces.
\end{example}
\begin{example}\label{Example: TSubmersions in CAlg}
Let $R$ be a commutative rig and let $f:A \to B$ be a morphism of commutative $R$-algebras. Then $f$ is a $T$-submersion if and only if $f$ is surjective. To see this, recall that the relative cotangent sequence for $f$ is the diagram of $A$-algebras
\[
\begin{tikzcd}
A \ar[r] & A \ltimes \operatorname{Ker}(f) \ar[r] & A[\epsilon] \ar[r]{}{\id \ltimes f} & A \ltimes B \ar[r] & A
\end{tikzcd}
\]
where $\operatorname{Ker}(f) = \lbrace a \in A \; | \; f(a) = 0 \rbrace$ and where $(\id \ltimes f)(a+x\epsilon) = (a,f(x)).$. Now, as $\DBun(A) \simeq \Mod{A}$ via the equivalence $M \mapsto A \ltimes M$ for $A$-modules $M$, we see that the diagram
\[
\begin{tikzcd}
A \ltimes \operatorname{Ker}(f) \ar[r] \ar[d, swap] & A[\epsilon] \ar[d]{}{\id \ltimes f} \\
A \ar[r] & A \ltimes B
\end{tikzcd}
\]
is a pushout in $\DBun(A)$ if and only if the diagram
\[
\begin{tikzcd}
\operatorname{Ker}(f) \ar[r] \ar[d] & A \ar[d]{}{f} \\
0 \ar[r] & B
\end{tikzcd}
\]
is a pushout in $\Mod{A}$. But this holds if and only if $f$ is a regular epimorphism in $\Mod{A}$ and hence, as $\Mod{A}$ is the category of models of the Lawvere theory of $A$-modules in $\Set$ and since $\Set$ is a topos, this occurs if and only if $f$ is a surjection of $A$-modules.
\end{example}
\begin{example}\label{Example: TSubmersions in CAlgOp}
Let $R$ be a commutative rig. Then for any map $f:A \to B$ of commutative $R$-algebras, $f$ is a $T$-submersion in $\CAlg{R}^{\op}$ if and only if the map $\delta$ in the module relative cotangent sequence
\[
\begin{tikzcd}
\Kah{A}{R} \otimes_A B \ar[r]{}{\delta} & \Kah{B}{R} \ar[r] & \Kah{B}{A} \ar[r] & 0
\end{tikzcd}
\]
given by 
\[
\delta(\mathrm{d}a \otimes b) \mapsto b\mathrm{d}\big(f(a)\big)
\]
has the property that $\delta$ is injective in $\Mod{B}$. This follows from the equivalence of categories $\DBun(B) \simeq \Mod{B}^{\op}$ in $\CAlg{R}^{\op}$, the fact that regular monics in $\Mod{B}$ are precisely injective morphisms of modules, and the fact that when taking the opposite of a category regular monics get sent to regular epimorphisms.
\end{example}
\begin{example}\label{Example: TSubmersions in Sch}
Let $S$ be a scheme and let $f:X \to Y$ be a morphism of $S$-schemes. Then $f$ is a $T$-submersion if and only if the relative cotangent sequence of quasi-coherent sheaves
\[
\begin{tikzcd}
    f^{\ast}\Kah{Y}{S} \ar[r]& \Kah{X}{S} \ar[r] & \Kah{X}{Y} \ar[r] & 0
\end{tikzcd}
\]
is an exact sequence in $\QCoh(X)$. As before, this follows from the opposite equivalence of categories $\QCoh(X)^{\op} \simeq \DBun(X)$.
\end{example}

\begin{example} For an arbitrary CDC, to the best of our knowledge, not much more can be deduced apart from saying that $f$ is a $T$-submersion if and only if $\langle \pi_0,D[f]\rangle$ is a regular $T$-epimorphism in the category of linear diff bundles over $X$. 
\end{example}

We close this subsection by providing a counter-example to a sort of desideratum for $T$-submersions which one would be tempted to derive (and in fact was claimed by a draft of this paper at one point in time). When given a $T$-submersion $f:X \to Y$, it is tempting to desire or ask that in the diagram
\[
\begin{tikzcd}
X \ar[r] & T_{X/Y} \ar[r]{}{\iota_f} & TX \ar[r]{}{\theta_f} & f^{\ast}TY
\end{tikzcd}
\]
one has that $\theta_f$ is the coequalizer of $0_X \circ p_{X/Y}$ and $\iota_f$ if and only if $\theta_f$ is an epimorphism. This would be in line with the fact that this holds in each of the cases $\SMan$ and $\Sch_{/S}$ for any scheme $S$, as in those cases one has that to determine $T$-submersions one need only look at either the Abelian categories of quasi-coherent sheavesx $\QCoh(X)$ for schemes $X$ or the additive categories $\DBun(M)$ for smooth manifolds $M$. In either case, one can use techniques from categories which at least stalkwise look like regular categories, i.e., in each case one can test if the map
\[
\theta_f:TX \to f^{\ast}TY
\]
is a regular epimorphism by looking at the corresponding maps of stalks
\[
\theta_{f,x}:T_xX \to (f^{\ast}TY)_x
\]
and determining \emph{there} if $\theta_f$ is regular epic for each point $x \in X$. This means in particular, that the categories $\DBun(X)$ have certain degrees of regularity with which one can work, and this greatly aids one in working with regular epimorphisms. In particular, since in the $\SMan$ case the stalkwise categories look like vector spaces and in the $\Sch_{/S}$ case the stalkwise case the stalkwise categories look like modules, one can actually say that $\theta_f$ is a $T$-submersion if and only if it is merely epic in the first place. 

However, when one loses access to either negatives and regularity, one cannot make such a claim. As such we present here an example of a tangent category $\Cscr$ together with a morphism $f$ for which $\theta_f$ is epic but $f$ is not a $T$-submersion.

\begin{example}
Recall that $T$-coequalizers are necessarily regular $T$-epimorphisms. As such, it suffices to provide an example of a morpshim $f:X \to Y$ in a tangent category $\Cscr$ for which $\theta_f$ is epic but not regular epic. To this end, let $\Cscr = \mathbf{CMon}$ (equipped with the biproduct tangent structure) and let $f:\N \to \Z$ be the standard inclusion. Then $f$ is an epimorphism of commutative monoids which is not regular epic. As $\theta_f:\N \oplus \N \to \N \oplus \Z$ is the map $\id_{\N} \oplus f$ and the biproduct of epimorphisms remains epimorphic, the map $\id \oplus f:\N \oplus \N \to \N \oplus \Z$ is an epimorphism in $\DBun(\N)$. We now argue that $\theta_f$ cannot be a regular epimorphism. Now for every commutative monoid $X$ we have $\DBun(X) \simeq  \mathbf{CMon}$ with the equivalence $ \mathbf{CMon} \to \DBun(X)$ given by $E \mapsto X \oplus E$ on objects and $\varphi \mapsto \id_X \oplus \varphi$. Now because this implies $\mathbf{CMon} \simeq \DBun(\N)$, we see that $\theta_f = \id_{\N} \oplus f$ is a regular epimorphism in $\DBun(\N)$ if and only if $f$ is a regular epimorphism in $\mathbf{CMon}$. Because $f$ is not regular, so too is $\theta_f$ not regular.
\end{example}

\subsection{Composition and Pullback Stability of Tangent Submersions}\label{Subsection: Tsubmersion Properties and Counterexamples}

While it is true in each of the categories $\SMan, \Sch_{/S}, \CAlg{R}, \CAlg{R}^{\op}$ that $T$-submersions are stable under composition and pullback, this seems in practice to be something remarkable that we could not prove at the level of arbitrary tangent categories.  In this short subsection we will give conditions on when $T$-submersions compose and are stable under pullback, and also provide some indications as to why the composition-and-pullback stability of $T$-submersions is a special property which need not always hold.

We start by examining the question of when $T$-submersions are stable under $T$-pullback. Begin by assuming that we have a $T$-pullback
\[
\begin{tikzcd}
X \times_Z Y \ar[r]{}{\pi_1} \ar[d, swap]{}{\pi_0} & Y \ar[d]{}{g} \\
X \ar[r, swap]{}{f} & Z
\end{tikzcd}
\]
for which both $f$ and $\pi_1$ are $p$-carrable and $0$-carrable. Then by Proposition \ref{Prop: How thetaf pulls back} we have an isomorphism
\[
\rho:\pi_1^{\ast}(TY) \xrightarrow{\cong} f^{\ast}(TZ) \times_{TZ} TY
\]
which makes the diagram
\[
\begin{tikzcd}
T(X \times_Y Z) \ar[rr]{}{\cong} \ar[d, swap]{}{\theta_{\pi_1}} & & TX \times_{TZ} TY \ar[d]{}{\theta_f \times \id_{TZ}}\\
\pi_1^{\ast}(TY) \ar[rr, swap]{}{\rho} & & f^{\ast}(TZ) \times_{TZ} TY
\end{tikzcd}
\]
commute. Additionally, because $f$ and $\pi_1$ are both $p$-and-$0$-carrable, by Proposition \ref{Prop: Zero carrable nec and suf} we know that there is an isomorphism
\[
T_{(X \times_Z Y)/Y} \cong T_{X/Z}\times_{TZ} TY
\]
and this isomorphism makes
\[
\begin{tikzcd}
T_{(X \times_Y Z)/Y} \ar[rr]{}{\cong} \ar[d, swap]{}{\iota_{\pi_1}} & & T_{X/Z} \times_{TZ} TY \ar[d]{}{\iota_f \times \id_{TZ}} \\
T(X \times_Z Y) \ar[rr, swap]{}{\cong} & & TX \times_{TZ} TY
\end{tikzcd}
\]
commute in $\Cscr$. Because answering if the relative cotangent sequence can be done on the isomorphism classes of $\iota_{\pi_1}:T_{(X \times_Z Y)/Y} \to T(X \times_Z Y)$ and $\theta_{\pi_1}:T(X \times_Z Y) \to \pi_1^{\ast}(TY)$, to show $\pi_1$ is a $T$-submersion it suffices to prove that the diagram
\[
\begin{tikzcd}
T_{X/Z} \times_Z Y \ar[rrrr]{}{\langle \iota_f \circ \pi_0, 0_Y \circ \pi_1 \rangle} \ar[d, swap]{}{p_{X/Z} \times \id_Y} & & & & TX \times_{TZ} TY \ar[d]{}{\theta_f \times \id_{TY}} \\
X \times_Z Y \ar[rrrr, swap]{}{\langle \zeta_{f^{\ast}(TZ)} \circ \pi_0, 0_Y\circ \pi_1 \rangle} & & & & f^{\ast}(TZ) \times_{TZ} TY
\end{tikzcd}
\]
is a linear $T$-pushout in $\DBun(X \times_Z Y)$. Because of this and Proposition \ref{Prop: Section TSub: Submersion iff Rel Cotan is T Exact Sequence}, we find that $\pi_1$ is a $T$-submersion if and only if the relative cotangent sequence
\[
\begin{tikzcd}
X \times_Z Y \ar[r] & T_{X/Z} \times_{TZ} TY \ar[rr]{}{\iota_f \times \id} & & TX \times_{TZ} TY \ar[rr]{}{\theta_f \times \id} & & f^{\ast}(TZ) \times_{TZ} TY\ar[r] & X\times_Z Y
\end{tikzcd}
\]
is $T$-linear exact in $\DBun(X \times_Y Z)$. However, as each map in this sequence is a categorical pullback of either $\iota_f$ or $\theta_f$ against the map $Tg:TY \to TZ$ and since
\[
\begin{tikzcd}
T_{X/Z} \ar[r]{}{\iota_f} \ar[d, swap]{}{p_{X/Z}} & TX \ar[d]{}{\theta_f} \\
X \ar[r] & f^{\ast}TZ
\end{tikzcd}
\]
is a linear $T$-pullback, the same is true of
\[
\begin{tikzcd}
T_{X/Z} \times_{TZ} TY \ar[rr]{}{\iota_f \times \id} \ar[d] & & TX \times_{TZ} TY \ar[d]{}{\theta_f \times \id} \\
X \times_Y Z \ar[rr] & & f^{\ast}TZ \times_{TZ} TY
\end{tikzcd}
\]
As such, to show that $\pi_1$ is a $T$-submersion, it is necessary and sufficient to know that
\[
\begin{tikzcd}
T_{X/Z} \times_{TZ} TY \ar[rr]{}{\iota_f \times \id} \ar[d] & & TX \times_{TZ} TY \ar[d]{}{\theta_f \times \id} \\
X \times_Y Z \ar[rr] & & f^{\ast}TZ \times_{TZ} TY
\end{tikzcd}
\]
is a linear $T$-pushout. In particular, we see that if the functor $(-) \times \id_{TY}$ is finitely cocontinuous and if $f$ is a submersion, the so too is $\pi_1$. 
\begin{proposition}\label{Prop: Section Sub: When submersions compose}
If $f:X \to Y$ is a $T$-submersion and there is a $T$-pullback
\[
\begin{tikzcd}
X \times_Z Y \ar[r]{}{\pi_1} \ar[d, swap]{}{\pi_0} & Y \ar[d]{}{g} \\
X \ar[r, swap]{}{f} & Z
\end{tikzcd}
\]
with $\pi_1$ both $p$-and-$0$-carrable, then $\pi_1$ is a $T$-submersion if and only if the pullback 
\[
\theta_f \times \id_{TY}:TX \times_{TZ} TY \to f^{\ast}(TZ) \times_{TZ} TY
\]
is the cokernel of $\iota_f \times \id_{TY}$ in $\DBun(X \times_Y Z)$. In particular, if the functor $(-) \times_{TZ} TY$ is finitely cocontinuous then $\pi_1$ is a $T$-submersion
\end{proposition}
\begin{proof}
The discussion prior to the statement of the proposition shows that the diagram
\[
\begin{tikzcd}
T_{X/Z} \times_{TZ} TY \ar[rr]{}{\iota_f \times \id} \ar[d] & & TX \times_{TZ} TY \ar[d]{}{\theta_f \times \id} \\
X \times_Y Z \ar[rr] & & f^{\ast}TZ \times_{TZ} TY
\end{tikzcd}
\]
is a linear $T$-pushout if and only if the map $\theta_f \times \id_{TY}$ is the cokernel of $\iota_f \times \id_{TY}$. The final statement of the proposition simply gives a condition in which
\[
\begin{tikzcd}
T_{X/Z} \ar[r]{}{\iota_f} \ar[d, swap]{}{p_{X/Z}} & TX \ar[d]{}{\theta_f} \\
X \ar[r] & f^{\ast}TZ
\end{tikzcd}
\]
is a linear $T$-pushout and when the functor $(-) \times_{TY} TZ$ preserves this linear $T$-pushout.
\end{proof}

To better understand the issues which arise regarding the composition of $T$-submersions, let $f:X \to Y$ and $g:Y \to Z$ be $T$-submersions and assume that $g \circ f$ is both $p$-carrable and $0$-carrable. By Lemma \ref{Lemma: Thetaf and composition}, have a natural isomorphism $\alpha:f^{\ast}(g^{\ast}TZ) \to (g \circ f)^{\ast}TZ$ of differential bundles over $X$ which renders
\[
\begin{tikzcd}
TX \ar[rr]{}{\theta_{g \circ f}} \ar[d, swap]{}{\theta_f} & & (g \circ f)^{\ast}TZ \\
f^{\ast}(TY) \ar[rr, swap]{}{f^{\ast}\theta_g} & & f^{\ast}(g^{\ast}TY) \ar[u, swap]{}{\alpha} \ar[u]{}{\cong}
\end{tikzcd}
\]
commutative. These two observations allow us to reduce the question of whether or not the cotangent sequence is exact to asking if the particular sequence
\[
\begin{tikzcd}
    X \ar[r] & T_{X/Z} \ar[rr]{}{\pi_0} & & TX \ar[rr]{}{f^{\ast}(\theta_g)\circ \theta_f} & & f^{\ast}(g^{\ast}TZ)
\end{tikzcd}
\]
is exact in $\DBun(X)$, i.e., if $f^{\ast}(\theta_g) \circ \theta_f = \operatorname{coker}(\iota_{g \circ f})$. On one hand, if we know that the pullback $f^{\ast}\theta_g$ remains a regular epimorphism in the sense that the sequence
\[
\begin{tikzcd}
X \ar[r] & f^{\ast}T_{Y/Z} \ar[r]{}{f^{\ast}\iota_g} & f^{\ast}TY \ar[r]{}{f^{\ast}\theta_g} & f^{\ast}(g^{\ast}TZ) \ar[r] & X
\end{tikzcd}
\]
is exact in $\DBun(X)$, then we simply need to know that the kernels of both morphisms $\theta_f$ and $f^{\ast}\theta_g$ interact appropriately. That is, we first need to know that given the composite morphism
\[
\theta_f \circ \iota_{g \circ f}:T_{X/Z} \to f^{\ast}TY
\]
in $\DBun(X)$, the unique factorization
\[
\begin{tikzcd}
T_{X/Z} \ar[rr]{}{\theta_f \circ \iota_{g \circ f}} \ar[dr, swap, dashed]{}{\exists!\kappa} & & f^{\ast}(TY) \\
 & f^{\ast}T_{Y/Z} \ar[ur, swap]{}{f^{\ast}\iota_g}
\end{tikzcd}
\]
has the property that $\kappa$ is a $T$-epimorphism. We will show below that knowing this is sufficient for us to be able to deduce that the composition $g \circ f$ of $T$-submersions $f$ and $g$ is again a $T$-submersion.

\begin{proposition}\label{Prop: Section Sub: When submersions pull back}
Let $\Cscr$ be a tangent category with $0$-and-$p$-carrable morphisms $f:X \to Y$, $g:Y \to Z$, and $g \circ f:X \to Z$. Then if $f$ and $g$ are $T$-submersions and if:
\begin{itemize}
    \item The relative cotangent exact sequence
    \[
    \begin{tikzcd}
        Y \ar[r] & T_{Y/Z} \ar[r] & TY \ar[r] & g^{\ast}TZ \ar[r] & Y
    \end{tikzcd}
    \]
   in $\DBun(Y)$ pulls back via $f^{\ast}(-)$ to a $T$-linear exact sequence
   \[
   \begin{tikzcd}
       X \ar[r] & f^{\ast}(T_{Y/Z}) \ar[rr]{}{f^{\ast}\iota_g} & & f^{\ast}TY \ar[rr]{}{f^{\ast}\theta_g} & & f^{\ast}(g^{\ast}TZ) \ar[r] & X
   \end{tikzcd}
   \]
   in $\DBun(X)$;
   \item The unique morphism $\kappa$ factoring
   \[
   \begin{tikzcd}
T_{X/Z} \ar[rr]{}{\theta_f \circ \iota_{g \circ f}} \ar[dr, swap, dashed]{}{\exists!\kappa} & & f^{\ast}(TY) \\
 & f^{\ast}T_{Y/Z} \ar[ur, swap]{}{f^{\ast}\iota_g}
\end{tikzcd}
   \]
   is a $T$-linear epimorphism in $\DBun(X)$.
\end{itemize}
Then $g \circ f$ is a submersion.
\end{proposition}
\begin{proof}
We begin by assuming that we have a commuting diagram of the form
\[
\begin{tikzcd}
    T_{X/Z} \ar[r]{}{\iota_{g \circ f}} \ar[d, swap]{}{p_{X/Z}} & TX \ar[d]{}{\varphi} \\
    X \ar[r, swap]{}{\zeta} & E
\end{tikzcd}
\]
in $\DBun(X)$. Now observe that since the diagram
\[
\begin{tikzcd}
T_{X/Y} \ar[rr]{}{g \circ \iota^{f}} \ar[d, swap]{}{\iota_f} & & Z \ar[d]{}{0_Z} \\
TX \ar[rr, swap]{}{T(g \circ f)} & & TZ
\end{tikzcd}
\]
commutes because $0_Z \circ g \circ \iota^f = Tg \circ 0_Y \circ \iota^f = Tg \circ Tf \circ \iota_f = T(g \circ f) \circ \iota_f$, there is a unique morphism $\gamma_{f,g}:T_{X/Y} \to T_{X/Z}$ rendering the diagram
\[
\begin{tikzcd}
T_{X/Y} \ar[dr, dashed]{}{\exists!\gamma_{f,g}} \ar[drr, bend left = 20]{}{g \circ \iota^f} \ar[ddr, swap, bend right = 20]{}{\iota_f} \\
 & T_{X/Z} \ar[r]{}{\iota^{g \circ f}} \ar[d, swap]{}{\iota_{g \circ f}} & Z \ar[d]{}{0_Z} \\
& TX \ar[r, swap]{}{T(g \circ f)} & TZ
\end{tikzcd}
\]
commutative in $\Cscr$. But then a routine check shows that $\gamma_{f,g}$ is both monic and a linear morphism $\gamma_{f,g}:T_{X/Y} \to T_{X/Z}$ in $\DBun(X)$. This implies that the diagram
\[
\begin{tikzcd}
T_{X/Y} \ar[r]{}{\iota_f} \ar[d, swap]{}{p_X} & TX \ar[d]{}{\varphi} \\
X \ar[r, swap]{}{\zeta} & E
\end{tikzcd}
\]
commutes in $\DBun(X)$, as the equation
\[
\varphi \circ \iota_f = \varphi \circ \iota_{g \circ f} \circ \gamma_{f,g} = \zeta \circ p_{X/Z} \circ \gamma_{f,g} = \zeta \circ p_{X/Y}
\]
holds. But now, as $f$ is a $T$-submersion and hence $\theta_f$ is a linear $T$-pushout, we see that there is a unique morphism $\psi:f^{\ast}(TY) \to E$ in $\DBun(X)$ which renders
\[
\begin{tikzcd}
T_{X/Y} \ar[r]{}{\iota_f} \ar[d, swap]{}{p_{X/Y}} & TX \ar[ddr, bend left = 20]{}{\varphi} \ar[d]{}{\theta_f} \\
X \ar[r, swap]{}{f^{\ast}0_Y} \ar[drr, swap, bend right =20]{}{\zeta} & f^{\ast}(TY) \ar[dr, dashed]{}{\exists!\psi} \\
 & & E
\end{tikzcd}
\]
commutative. Furthermore, a routine computation gives the commutativity of the diagram
\[
\begin{tikzcd}
T_{X/Z} \ar[r]{}{\iota_{g \circ f}} \ar[d, swap]{}{p_{X/Z}} & TX \ar[d]{}{\theta_f} \ar[ddr, bend left = 20]{}{\psi} \\
X \ar[r, swap]{}{f^{\ast}0_Y} \ar[drr, swap, bend right = 20]{}{\zeta} & f^{\ast}(TY) \ar[dr]{}{\psi} \\
 & & E
\end{tikzcd}
\]
in $\DBun(X)$.

In order to proceed, rearrange the diagram above to the diagram:
\[
\begin{tikzcd}
T_{X/Z} \ar[rr]{}{\theta_f \circ \iota_{g \circ f}} \ar[d, swap]{}{p_{X/Z}} & & f^{\ast}(TY) \ar[d]{}{\psi} \\
X \ar[rr, swap]{}{\zeta} & & E
\end{tikzcd}
\]
Our goal now is to use the fact that the sequence
\[
\begin{tikzcd}
X \ar[r] & f^{\ast}T_{Y/Z} \ar[r]{}{f^{\ast}\iota_g} & f^{\ast}TY \ar[r]{}{f^{\ast}\theta_g} & f^{\ast}(g^{\ast}TZ) \ar[r] & X
\end{tikzcd}
\]
is exact in $\DBun(X)$ (and hence that $f^{\ast}\iota_g$ is the kernel of $f^{\ast}\theta_g$) in order to show that $\theta_f \circ \iota_{g \circ f}$ factors through $f^{\ast}\iota_g$. To this end, recall from Lemma \ref{Lemma: Thetaf and composition} that since there is a linear isomorphism $\alpha:f^{\ast}(g^{\ast}TZ) \xrightarrow{\cong} (g \circ f)^{\ast}TZ$ which makes the equation
\[
\theta_{g \circ f} = \alpha \circ f^{\ast}\theta_g \circ \theta_f
\]
hold. This allows us to compute that
\[
f^{\ast}\theta_g \circ \theta_f \circ \iota_{g \circ f} = \alpha^{-1} \circ \theta_{g \circ f} \circ \iota_{g \circ f} = \alpha^{-1} \circ (g \circ f)^{\ast}0_{Z} \circ p_{X/Z} = f^{\ast}(g^{\ast}0_Z) \circ p_{X/Z}.
\]
and so, by the universal property carried by $f^{\ast}T_{Y/Z}$, we see that there is a unique morphism $\kappa:T_{X/Z} \to f^{\ast}T_{Y/Z}$ in $\DBun(X)$ which renders the diagram
\[
\begin{tikzcd}
T_{X/Z} \ar[rr]{}{\theta_f \circ \iota_{g \circ f}} \ar[dr, dashed, swap]{}{\exists! \kappa} & & f^{\ast}TY \\
 & f^{\ast}T_{Y/Z} \ar[ur, swap]{}{f^{\ast}\iota_g}
\end{tikzcd}
\]
Now, by assumption the morphism $\kappa$ is an epimorphism. This allows us to compute that
\[
\psi \circ f^{\ast}\iota_g \circ \kappa = \psi \circ \theta_f \circ \iota_{g \circ f} =  \zeta \circ p_{X/Z} = \zeta \circ \pi_0 \circ \kappa
\]
and hence deduce the commutativity of the diagram
\[
\begin{tikzcd}
f^{\ast}T_{Y/Z} \ar[r]{}{f^{\ast}\iota_g} \ar[d, swap]{}{\pi_0} & f^{\ast}TY \ar[d]{}{\psi} \\
X \ar[r, swap]{}{\zeta} & E
\end{tikzcd}
\]
in $\DBun(X)$. We now use that as $f^{\ast}\theta_g$ is the $T$-linear cokernel of $f^{\ast}\iota_g$, there is a unique morphism $\rho:f^{\ast}(g^{\ast}TZ) \to E$ in $\DBun(X)$ which renders the diagram
\[
\begin{tikzcd}
f^{\ast}T_{Y/Z} \ar[rr]{}{f^{\ast}\iota_g} \ar[d, swap]{}{\pi_0} & & f^{\ast}TY \ar[ddr, bend left = 20]{}{\psi} \ar[d]{}{f^{\ast}\theta_g} \\
X \ar[rr, swap]{}{f^{\ast}(g^{\ast}0_Z)} \ar[drrr, swap, bend right = 20]{}{\zeta} & & f^{\ast}(g^{\ast}TZ) \ar[dr, dashed]{}{\exists!\rho} \\
 & & & E
\end{tikzcd}
\]
in $\DBun(X)$. Post-composing both $f^{\ast}\iota_g$ and $\pi_0$ by $\kappa$ then allows us to deduce that the diagram
\[
\begin{tikzcd}
T_{X/Z} \ar[rr]{}{\theta_f \circ \iota_{g \circ f}} \ar[d, swap]{}{p_{X/Z}} & & f^{\ast}TY \ar[d]{}{f^{\ast}\theta_g} \ar[ddr, bend left = 20]{}{\psi} \\
X \ar[rr, swap]{}{f^{\ast}(g^{\ast}0_Z)} \ar[drrr, bend right = 20, swap]{}{\zeta} & & f^{\ast}(g^{\ast}TZ) \ar[dr]{}{\rho} \\
 & & & E
\end{tikzcd}
\]
and hence its rearranged version
\[
\begin{tikzcd}
T_{X/Z} \ar[rr]{}{\iota_{g \circ f}} \ar[d, swap]{}{p_{X/Z}} & & TX \ar[d]{}{f^{\ast}\theta_g \circ \theta_f} \ar[ddr, bend left = 20]{}{\varphi} \\
X \ar[rr, swap]{}{f^{\ast}(g^{\ast}0_Z)} \ar[drrr, swap, bend right = 20]{}{\zeta} & & f^{\ast}(g^{\ast}TZ) \ar[dr, dashed]{}{\exists!\rho} \\
 & & & E
\end{tikzcd}
\]
commute in $\DBun(X)$; note that the uniqueness of $\rho$ follows from the fact that is an epimorphism $f^{\ast}(\theta_g) \circ \theta_f$ because $\theta_f$ and $g^{\ast}\theta_f$ are both epimorphisms.

We have just proved that the diagram
\[
\begin{tikzcd}
T_{X/Z} \ar[rr]{}{\iota_{g \circ f}} \ar[d, swap]{}{p_{X/Z}} & & TX \ar[d]{}{f^{\ast}\theta_g \circ \theta_f} \\
X \ar[rr, swap]{}{f^{\ast}(g^{\ast}0_Z)} & & f^{\ast}(g^{\ast}TZ)
\end{tikzcd}
\]
is a pushout in $\DBun(X)$. To see that it is a $T$-linear pushout, we simply note that both diagrams
\[
\begin{tikzcd}
T_{X/Y} \ar[r]{}{\iota_f} \ar[d, swap]{}{p_{X/Y}} & TX \ar[d]{}{\theta_f} \\
X \ar[r] & f^{\ast}TY
\end{tikzcd}\quad
\begin{tikzcd}
f^{\ast}T_{Y/Z} \ar[r]{}{f^{\ast}\iota_g} \ar[d, swap]{}{\pi_0} & f^{\ast}TY \ar[d]{}{f^{\ast}\theta_g} \\
X \ar[r] & f^{\ast}(g^{\ast}TZ)
\end{tikzcd}
\]
are $T$-linear pushouts by assumption while the factorizations $\kappa$ is a $T$-linear epimorphism by assumption. Because in addition the diagram
\[
\begin{tikzcd}
f^{\ast}T_{Y/Z} \ar[r]{}{f^{\ast}\iota_g} \ar[d, swap]{}{\pi_0} & f^{\ast}TY \ar[d]{}{f^{\ast}\theta_g} \\
X \ar[r] & f^{\ast}(g^{\ast}TZ)
\end{tikzcd}
\]
is a $T$-linear pullback, the maps $T^m\kappa$ are the desired factorizations of $T^m\theta_f \circ T^m\iota_{g \circ f}$ through $T^m(f^{\ast}\iota_g)$. Thus each formal tool used in the factorization of
\[
\begin{tikzcd}
T_{X/Z} \ar[r]{}{\iota_{g \circ f}} \ar[d, swap]{}{p_{X/Z}} & TX \ar[d]{}{\varphi} \\
X \ar[r, swap]{}{\zeta} & E
\end{tikzcd}
\]
exists in $\DBun(T^mX)$ as well and may be used after applying the functor $T_{\ast}^m:\DBun(X) \to \DBun(T^mX)$. This allows us to deduce that each diagram
\[
\begin{tikzcd}
T^m(T_{X/Z}) \ar[rr]{}{T^m\iota_{g \circ f}} \ar[d, swap]{}{T^m(p_{X/Z})} & & T^{m+1}X \ar[d]{}{T^m(f^{\ast}\theta_g) \circ T^m\theta_f} \\
T^mX \ar[rr, swap]{}{T^m(f^{\ast}(g^{\ast}0_Z))} & & T^m(f^{\ast}(g^{\ast}TZ)
\end{tikzcd}
\]
is a pushout in $\DBun(T^mX)$. Finally post-composing again by $\alpha:f^{\ast}(g^{\ast}TZ) \to (g \circ f)^{\ast}TZ$ in all relevant locations allows us to deduce that
\[
\begin{tikzcd}
T_{X/Z} \ar[rr]{}{\iota_{g \circ f}} \ar[d, swap]{}{p_{X/Z}} & & TX \ar[d]{}{\theta_{g \circ f}} \\
X \ar[rr, swap]{}{(g \circ f)^{\ast}0_Z} & & (g \circ f)^{\ast}TZ
\end{tikzcd}
\]
is a $T$-linear pushout in $\DBun(X)$ and hence that $f$ is a $T$-submersion.
\end{proof}

It may be tempting to think at this point that $T$-submersions are a remarkably poorly-behaved class of morphism in a tangent category (at full generality, at least). However, we close this section by showing that not \emph{all} pleasing formal properties of submersions as they are known in $\SMan$ are lost upon generalizing to $T$-submersions in tangent categories. $T$-submersions are still stable under application of all powers of the tangent bundle $T^m$. Any further development and study of $T$-submersions at this level of generality is, however, left as future work.
\begin{proposition}\label{Prop: Section Sub: Submersions are Tstable}
If $f:X \to Y$ is a $T$-submersion then $T^mf$ is a $T$-submersion for all $m \in \N$.
\end{proposition}
\begin{proof}
This follows from the fact that if the diagram
\[
\begin{tikzcd}
T_{X/Y} \ar[r]{}{\iota_f} \ar[d, swap]{}{p_{X/Y}} & TX \ar[d]{}{\theta_f} \\
X \ar[r, swap] & f^{\ast}TY
\end{tikzcd}
\]
is a $T$-linear pushout, then for all $m \in \N$ the diagram
\[
\begin{tikzcd}
T^m(T_{X/Y}) \ar[r]{}{T^m\iota_f} \ar[d, swap]{}{T^mp_{X/Y}} & T^{m+1}X \ar[d]{}{\theta_f} \\
T^mX \ar[r, swap]{}{T^m(f^{\ast}0_Y)} & T^m(f^{\ast}TY)
\end{tikzcd}
\]
is a $T$-linear pushout in $\DBun(T^mX)$. Furthermore, a straightforward but tedious diagram chase shows that these squares are naturally isomorphic to the diagrams
\[
\begin{tikzcd}
T^m(T_{X/Y}) \ar[d, swap]{}{p_{T^mX/T^mY}} \ar[r]{}{\iota_{T^mf}} & T^{m+1}X \ar[d]{}{\theta_{T^mf}} \\
T^mX \ar[r] & (T^mf)^{\ast}(T^{m+1}Y)
\end{tikzcd}
\]
for all $m \in \N$. Thus $T^mf$ is a $T$-submersion.
\end{proof}

\section{Split Tangent Submersions}\label{Section: Split Submersions}
In this section we study an alternative definition of submersion in tangent categories which is a kind of categorical dual to strong $T$-immersions (cf. Definition \ref{Defn: Strong Timmersion}). This kind of submersion, which we call a split $T$-submersion below, gives a categorical condition which is on one hand more strict than asking $T$-submersion above (in the sense that when $f$ is $p$-carrable and $0$-carrable, strong $T$-submersions ask for the relative cotangent sequence of $f$ to admit a right splitting and not merely be exact). However, on the other hand and perhaps surprisingly, we will see that split $T$-submersions are possible to cast in \emph{more} tangent categories than $T$-submersions because they need not require $p$-carrablitity nor $0$-carrability to exist. These are the morphisms $f:X \to Y$ which play a dual role to that of strong $T$-immersions: they are precisely the maps whose bundle projection naturality square are \emph{weak} pullbacks (and so have the existence but not necessarily uniqueness of lifts/factorizations).

What we call split $T$-submersions have been studied in tangent category theory already: a notable and important example justifying and placing this level of generality appeared in \cite{BenVectorBundles} when MacAdam used weak pullbacks to study submersions of smooth manifolds and prove, using the language of retractive display systems, that vector bundles in $\SMan$ correspond to differential bundles. However, as we indicated earlier, these maps require quite a strong technical condition to exist: they need the map $f:X \to Y$ to have the horizontal descent $\theta_f$ have a \emph{global} section andnot merely a local section. As such, it is generally only possible in $\SMan$ to have split $T$-submersions correspond to maps $f:X \to Y$ between paracompact smooth manifolds \emph{or} maps $f$ which admit a connection\footnote{We will prove below that to give a split $T$-submersion in $\SMan$ is equivalent to giving a submersion $f:X \to Y$ which admits a smooth partition of unity.}.

Despite these warnings above, in alternative contexts to differential geometry, we will see below in algebraic settings that meeting arbitrary maps in maximally transverse ways with global sections is much more subtle than one may initially expect. In $\CAlg{R}$, for a commutative rig $R$, we will find that $f:A \to B$ is a submersion precisely when it is surjective with a linear section. Additionally, when given a commutative rig $R$ and a map $f:A \to B$ in $\CAlg{R}$, in $\CAlg{R}^{\op}$ we will see that $f$ is  a submersion precisely when $\Kah{B}{R}$ is a split extension of $\Kah{B}{A}$ by $\Kah{A}{R} \otimes_A B$ in $\Mod{B}$. When $R$ is a commutative ring, this is related to studying when the map $f$ is formally smooth relative to $R$ \cite[D{\'e}finition 19.9.1]{EGA04}\footnote{We will go through this further later, but a warning to our readers: being formally smooth relative to $R$ is a significantly weaker property than a morphism being formally smooth in the slice category over $R$.}. What is common to all perspectives, however, is that submersions arise from weak pullbacks in each respective tangent category.

\subsection{Weak Pullbacks and Submersions}

In the general case we study submersions following the paths laid down by \cite{BenVectorBundles}, \cite{JonathanThesis}, and \cite{GeoffMarcelloTSubmersionPaper}: submersions are maps for which the projection naturality square is a weak pullback which is preserved by all powers of the tangent bundle functor. To this end we will proceed by first recalling what it means to be a weak pullback. In a category $\Cscr$, a square 
\[
\begin{tikzcd}
X \ar[r]{}{f} \ar[d, swap]{}{g} & Y \ar[d]{}{h} \\
Z \ar[r, swap]{}{k} & W
\end{tikzcd}
\]
is a \emph{weak pullback in $\Cscr$} if for any diagram 
\[
\begin{tikzcd}
V \ar[r]{}{s} \ar[d, swap]{}{t} & Y \ar[d]{}{h} \\
Z \ar[r, swap]{}{k} & W
\end{tikzcd}
\]
there is at least one map such that the following diagram commutes: 
\[
\begin{tikzcd}
V \ar[drr, bend left = 20]{}{s} \ar[ddr, swap, bend right = 20]{}{t} \ar[dr, dashed]{}{\exists\geq 1} \\
& X \ar[r]{}{f} \ar[d, swap]{}{g} & Y \ar[d]{}{h} \\
& Z \ar[r, swap]{}{k} & W
\end{tikzcd}
\]
Additionally, we say that a weak pullback is a \emph{$T$-weak pullback} if for all $m \in \N$, after applying $T^m$ to the diagram, the new diagram remains a weak pullback.
\begin{definition}[{\cite[Definition 9]{BenVectorBundles}}]\label{Defn: TSubmersions}
In a tangent category $\Cscr$, a map $f:X \to Y$ is a \emph{split $T$-submersion} if the naturality square
\[
\begin{tikzcd}
TX \ar[r]{}{Tf} \ar[d, swap]{}{p_X} & TY \ar[d]{}{p_Y} \\
X \ar[r, swap]{}{f} & Y
\end{tikzcd}
\]
is a $T$-weak pullback.
\end{definition}
\begin{remark}
In full generality, it is possible to define split $T$-submersions in more general tangent categories than $T$-submersions. This is because we do not need a morphism $f$ to be $p$-carrable in order for $f$ to be a split $T$-submersion.
\end{remark}

In order to prove some key facts about $T$-submersions, we will need the following folklorish result about weak pullbacks, which we prove for completeness sake. 

\begin{lemma}[Folklore]\label{Lemma: Folklore for weak pullbacks}
Consider commutative diagrams with squares numbered as below:
\[
\begin{tikzcd}
X \ar[r, ""{name = UL}]{}{k} \ar[d, swap]{}{f} & Y \ar[d]{}{g} \ar[r, ""{name = UR}]{}{\ell} & Z \ar[d]{}{h} \\
W \ar[r, swap, ""{name = DL}]{}{k^{\prime}} & V \ar[r, swap, ""{name = DR}]{}{\ell^{\prime}} & U
\ar[from = UL, to = DL, symbol = {\rotatebox[origin=c]{90}{$(1)$}}]
\ar[from = UR, to = DR, symbol = {\rotatebox[origin=c]{90}{$(2)$}}]
\end{tikzcd}
\qquad
\begin{tikzcd}
X \ar[r, ""{name = U}]{}{\ell \circ k} \ar[d, swap]{}{f} & Z \ar[d]{}{h} \\
W \ar[r, swap, ""{name = D}]{}{\ell^{\prime} \circ k^{\prime}} & U
\ar[from = U, to = D, symbol = {\rotatebox[origin=c]{90}{$(3)$}}]
\end{tikzcd}
\]
Then:
\begin{enumerate}[{\em (i)}] 
    \item If Squares $(1)$ and $(2)$ are weak pullbacks then so too is Square $(3)$.
    \item If Square $(3)$ is a weak pullback and if Square $(2)$ is a prepullback then Square $(1)$ is a weak pullback.
\end{enumerate}
In particular, if Square $(2)$ is a pullback then Square $(1)$ is a weak pullback if and only if Square $(3)$ is a weak pullback.
\end{lemma}
\begin{proof}
Note that the final statement of the lemma follows by combining both $(i)$ and $(ii)$ together with the fact that pullbacks are both weak pullbacks and prepullbacks. As such, it suffices to prove $(i)$ and $(ii)$ to prove the lemma. We first prove $(i)$. Assume that we have a commuting diagram
\[
\begin{tikzcd}
P \ar[r]{}{s} \ar[d, swap]{}{t} & Z \ar[d]{}{h} \\
W \ar[r, swap]{}{\ell^{\prime} \circ k^{\prime}} & U
\end{tikzcd}
\]
and note that this implies that in particular we obtain a commuting diagram with a lift $\lambda$
\[
\begin{tikzcd}
V \ar[drr, bend left = 20]{}{s} \ar[ddr, swap, bend right = 20]{}{k^{\prime} \circ t} \ar[dr]{}{\lambda} \\
& Y \ar[r]{}{\ell} \ar[d, swap]{}{g} & Z \ar[d]{}{h} \\
& V \ar[r, swap]{}{\ell^{\prime}} & U
\end{tikzcd}
\]
by virtue of Square $(2)$ being a weak pullback. But then since $g \circ \lambda = k^{\prime} \circ t$, the fact that Square $(1)$ is a weak pullback allows us to produce a map $\rho:P \to X$ through which $\lambda$ and $t$ are factored by $k$ and $f$, respectively. A straightforward check shows that $\rho$ renders the entire diagram
\[
\begin{tikzcd}
P \ar[d, swap]{}{\rho} \ar[dd, bend right = 40, swap]{}{t} \ar[dr]{}{\lambda} \ar[drr, bend left = 20]{}{s} \\
X \ar[r]{}{k} \ar[d, swap]{}{f} & Y \ar[d]{}{g} \ar[r]{}{\ell} & Z \ar[d]{}{h} \\
W \ar[r, swap]{}{k^{\prime}} & V \ar[r, swap]{}{\ell^{\prime}} & U
\end{tikzcd}
\]
commutative. Because the outer edges exactly compose to the initial cone above Square $(3)$, this shows that Squre $(3)$ is a weak pullback.

Now for $(ii)$, assume that we have a commuting diagram
\[
\begin{tikzcd}
P \ar[r]{}{s} \ar[d, swap]{}{t} & Y \ar[d]{}{g} \\
W \ar[r, swap]{}{k^{\prime}} & V
\end{tikzcd}
\]
and note that by post-composing this with Square $(2)$ we get a commuting diagram
\[
\begin{tikzcd}
P \ar[drr, bend left = 20]{}{\ell \circ s} \ar[ddr, swap, bend right = 20]{}{t} \ar[dr, dashed]{}{\exists\,\lambda} & & \\
 & X \ar[r]{}{\ell \circ k} \ar[d, swap]{}{f} & Z \ar[d]{}{h} \\
 & W \ar[r, swap]{}{\ell^{\prime} \circ k^{\prime}} & U
\end{tikzcd}
\]
where $\lambda$ exists precisely because Square $(3)$ is a weak pullback. Since $f \circ \lambda = t$, we also have that $g \circ k \circ \lambda = k^{\prime} \circ f \circ \lambda = k^{\prime} \circ t$. Similarly, by the assumption $g \circ s = k^{\prime} \circ t$, we see that the diagram
\[
\begin{tikzcd}
P \ar[drr, bend left = 20]{}{\ell \circ s} \ar[ddr, swap, bend right = 20]{}{t} \ar[dr, shift left = 1]{}{k \circ \lambda} \ar[dr, swap, shift right = 1]{}{s}& & \\
 & Y \ar[r]{}{\ell} \ar[d, swap]{}{g} & Z \ar[d]{}{h} \\
 & v \ar[r, swap]{}{\ell^{\prime}} & U
\end{tikzcd}
\]
commutes with both presented lifts $s$ and $k \circ \lambda$. However, because Square $(2)$ is a  prepullback, we must have that $k \circ \lambda = s$ and so that the diagram
\[
\begin{tikzcd}
P \ar[drr, bend left = 20]{}{s} \ar[ddr, swap, bend right = 20]{}{t} \ar[dr]{}{\lambda} & & \\
 & X \ar[r]{}{k} \ar[d, swap]{}{f} & Y \ar[d]{}{g} \\
 & W \ar[r, swap]{}{k^{\prime}} & V
\end{tikzcd}
\]
commutes. Thus Square $(1)$ is a weak pullback.
\end{proof}

\begin{proposition}\label{Prop: Tsubmersions and stability properties}
In a tangent category $\Cscr$, let $f:X \to Y$ be a split $T$-submersion. Then:
\begin{enumerate}[{\em (i)}] 
    \item If $g:Y \to Z$ is a split $T$-sumbersion then $g \circ f:X \to Z$ is a split $T$-submersion.
    \item If $m \in \N$ then $T^mf$ is a split $T$-submersion.
    \item Assume $g:Z \to Y$ is an arbitrary map for which the $T$-pullback
    \[
    \begin{tikzcd}
    W \ar[r]{}{\pi_1} \ar[d, swap]{}{\pi_0} & Z \ar[d]{}{g} \\
    X \ar[r, swap]{}{f} & Y
    \end{tikzcd}
    \]
    exists. Then $\pi_1$ is a split $T$-submersion.
\end{enumerate}
\end{proposition}
\begin{proof} Starting with $(i)$, note that for all $m \in \N$ in the diagram
\[
\begin{tikzcd}
T^{m+1}X \ar[rr]{}{T^{m+1}f} \ar[d, swap]{}{(T^m \ast p)_X} & & T^{m+1}Y \ar[rr]{}{T^{m+1}g} \ar[d]{}{(T^m \ast p)_Y} & & T^{m+1}Z \ar[d]{}{(T^m \ast p)_Z} \\
T^mX \ar[rr, swap]{}{T^mf} & & T^mY \ar[rr,  swap]{}{T^mg} & & T^mZ
\end{tikzcd}
\]
appealing to Lemma \ref{Lemma: Folklore for weak pullbacks}.(i) gives the result because both the left-and-right-handed squares are weak pullbacks by assumption. For $(ii)$, observe that this is immediate (but tedious to verify) from the fact that each functor $T^m$ preserves the weak pullbacks defining what it means to be a $T$-submersion, using that we get various natural isomorphisms $T^{m+k} \ast p \cong T^{k} \ast \left(T^m \ast p\right) \cong T^k \ast \left(p \ast T^m\right)$ for all $k \in \N$ by appropriately whiskering the powers of the tangent bundle functors by various (whiskered) canonical flips, and using that being a weak pullback is stable under isomorphism. 

For $(iii)$, consider that by the final statement of Lemma \ref{Lemma: Folklore for weak pullbacks}, in the diagram
\[
\begin{tikzcd}
TW \ar[r, ""{name = UL}]{}{p_W} \ar[d, swap]{}{T\pi_1} & W \ar[d]{}{\pi_1} \ar[r, ""{name = UR}]{}{\pi_0} & X \ar[d]{}{f} \\
TZ \ar[r, swap, ""{name = DL}]{}{p_Z} & Z \ar[r, swap, ""{name = DR}]{}{g} & Y
\ar[from = UL, to = DL, symbol = {\rotatebox[origin=c]{90}{$(1)$}}]
\end{tikzcd}
\]
Square $(1)$ is a weak pullback if and only if the total square is a weak pullback. However, as the total square is equal to the commuting diagram
\[
\begin{tikzcd}
TW \ar[r]{}{T\pi_0} \ar[d, swap]{}{T\pi_1} & TX \ar[r]{}{p_X} \ar[d]{}{Tf} & X \ar[d]{}{f} \\
TZ \ar[r, swap]{}{Tg} & TY \ar[r, swap]{}{p_Y} & Y
\end{tikzcd}
\]
it suffices to prove that this rectangle is a weak pullback. To this end assume we have a commuting square
\[
\begin{tikzcd}
U \ar[r]{}{s} \ar[d, swap]{}{t} & TZ \ar[d]{}{p_Y \circ Tg} \\
X \ar[r, swap]{}{f} & Y
\end{tikzcd}
\]
and note that by reorganizing our composition we get a commuting diagram
\[
\begin{tikzcd}
U \ar[drr, bend left = 20]{}{Tg \circ s} \ar[ddr, swap, bend right = 20]{}{t} \ar[dr, dashed]{}{\exists\,\gamma} \\
 & TX \ar[r]{}{Tf} \ar[d, swap]{}{p_X} & TY \ar[d]{}{p_Y} \\
 & X \ar[r, swap]{}{f} & Y
\end{tikzcd}
\]
with the existence of $\gamma$ justified by the fact that the naturality square of $p$ at $f$ is a weak pullback. But then because $TW$ is a pullback of $Tf$ against $Tg$, we see that there is a unique morphism $\theta:U \to TW$ rendering the diagram
\[
\begin{tikzcd}
U \ar[drr, bend left = 20]{}{s} \ar[ddr, swap, bend right = 20]{}{\gamma} \ar[dr, dashed]{}{\exists!\theta} \\
 & TW \ar[r]{}{T\pi_1} \ar[d, swap]{}{T\pi_0} & TZ \ar[d]{}{Tg} \\
 & TX \ar[r, swap]{}{Tf} & TY
\end{tikzcd}
\]
commutative. This then implies that the diagram
\[
\begin{tikzcd}
U \ar[drr, bend left = 20]{}{s} \ar[ddr, swap, bend right = 20]{}{t} \ar[dr, dashed]{}{\exists\,\theta} \\
 & TW \ar[r]{}{T\pi_1} \ar[d]{}{p_X \circ T\pi_0} & TZ \ar[d]{}{p_Y \circ Tg} \\
 & X \ar[r, swap]{}{f} & Y
\end{tikzcd}
\]
commutes and so that the whole diagram is a weak pullback. Appealing again to the final statement of Lemma \ref{Lemma: Folklore for weak pullbacks} we get that the square
\[
\begin{tikzcd}
TW \ar[r]{}{T\pi_1} \ar[d, swap]{}{p_W} & TZ \ar[d]{}{p_Z} \\
W \ar[r, swap]{}{\pi_1} & Z
\end{tikzcd}
\]
is a weak pullback. Running this same argument for all $m \in \N$ and using that the powers of the tangent bundle functor $T^m$ preserve every structure weak pullback and pullback in sight allows us to conclude that $\pi_1$ is a $T$-submersion.
\end{proof}

As we can see in Definition \ref{Defn: TSubmersions} above, split $T$-submersions are defined in more general situations than when the horizontal descent $\theta_f$ is necessarily defined. This is not an issue in practice because conventional practice favours working with display $T$-submersions; cf.\! \cite[Proposition 2]{BenVectorBundles} and \cite[Proposition 2.28]{GeoffMarcelloTSubmersionPaper}, for instance. In order to relate split $T$-submersions to the horizontal descent $\theta_f$, we need to assume that the map $f:X \to Y$ is $p$-carrable. In this situation we find that a map is a split $T$-submersion if and only if its horizontal descent admits a section. 

\begin{proposition}\label{Prop: The lifting property characterization of TSmooth}
In a tangent category $\Cscr$, if $f: X \to Y$ is $p$-carrable morphism, then $f$ is a split $T$-submersion if and only if $\theta_f$ admits a section.
\end{proposition}
\begin{proof}
$\implies$: Assume that $f$ is a split $T$-submersion. Because the pullback projections from $f^{\ast}(TY)$ form a cone over the cospan $X \xrightarrow{f} Y \xleftarrow{p_Y} TY$, there is a map $s$ making the diagram
\[
\begin{tikzcd}
f^{\ast}(TY) \ar[dr, dashed]{}{\exists\,s} \ar[drr, bend left = 20]{}{\pr_1} \ar[ddr, swap, bend right = 20]{}{\pr_0} & & \\
 & TX \ar[r]{}{Tf} \ar[d, swap]{}{p_X} & TY \ar[d]{}{p_Y} \\
 & X \ar[r, swap]{}{f} & Y
\end{tikzcd}
\]
commute. Because $s$ is a cone morphism and $f^{\ast}(TY)$ is the terminal object in the category of cones, $\theta_f \circ s = \id$. So $\theta_f$ admits a section as desired. 

$\impliedby:$ Assume that $s:f^{\ast}(TY) \to TX$ is a section of $\theta_f$, i.e., $s$ is a map making
\[
\begin{tikzcd}
f^{\ast}(TY) \ar[dr, equals] \ar[r]{}{s} & TX \ar[d]{}{\theta_f} \\
 & f^{\ast}(TY)
\end{tikzcd}
\]
commute. Then to see that the naturality square for $p$ is a weak pullback define $t:Z \to TX$ by the formula $t := s \circ \gamma$. To see this defines a morphism of cones, consider that since we have the commuting diagrams
\[
\begin{tikzcd}
Z \ar[dr]{}{\gamma} \ar[drr, bend left = 20]{}{g} \ar[ddr, bend right = 20, swap]{}{h} \\
 & f^{\ast}(TY) \ar[r]{}{\pr_1} \ar[d, swap]{}{\pr_0} & TY \ar[d]{}{p_Y} \\
 & X \ar[r, swap]{}{f} & Y
\end{tikzcd}\quad
\begin{tikzcd}
TX \ar[dr]{}{\theta_f} \ar[drr, bend left = 20]{}{Tf} \ar[ddr, bend right = 20, swap]{}{p_X} \\
 & f^{\ast}(TY) \ar[r]{}{\pr_1} \ar[d, swap]{}{\pr_0} & TY \ar[d]{}{p_Y} \\
 & X \ar[r, swap]{}{f} & Y
\end{tikzcd}
\]
then
\[
p_X \circ t = p_X \circ s \circ \gamma = p_X \circ s \circ \gamma = \pr_0 \circ \gamma = h
\]
and
\[
Tf \circ t = Tf \circ s \circ \gamma = \pr_1 \circ \theta_f \circ s \circ \gamma = \pr_1 \circ \gamma = g,
\]
so $t$ is indeed a cone morphism. Finally we compute that
\[
\theta_f \circ t = \theta_f \circ s \circ \gamma = \gamma,
\]
so there is indeed a factorization of $\gamma$ through $TX$. In particular, this shows that the naturality square for $p$ at $f$ is a weak pullback. We now verify that for arbitrary $m \in \N$, the naturality square
\[
\begin{tikzcd}
    T^{m+1}X \ar[r]{}{T^{m+1}f} \ar[d, swap]{}{(T^{m} \ast p)_X} & T^{m+1}Y \ar[d]{}{(T^m \ast p)_Y} \\
    T^mX \ar[r, swap]{}{T^mf} & T^mY
\end{tikzcd}
\]
is a weak pullback. However, because this is given by applying $T^m$ to the naturality square for $f$ at $p$ and because $T^m$ preserves the pullback $f^{\ast}(TY)$, $T^m(s)$ is a section for $T^m(\theta_f)$ because split idempotents are absolute  coequalizers. Appealing to Corollary \ref{Cor: The other higher powers of T to thetaf in one line} and writing $\theta_{T^mf} = \tilde{C}_m \circ T^m\theta_f \circ C_m$, defining $s_{T^ms} := C_m^{-1} \circ T^ms \circ \tilde{C}_m^{-1}$ gives that
\[
\theta_{T^mf} \circ s_{T^mf} = \tilde{C}_m \circ T^m\theta_f \circ C_m \circ C_m^{-1} \circ T^ms \circ \tilde{C}_m^{-1} = \tilde{C}_m \circ T^m(\theta_f \circ s) \circ \tilde{C}_m^{-1} = \tilde{C}_m \circ \tilde{C}_m^{-1} = \id.
\]
Thus each such square is a $T$-weak pullback.
\end{proof}
We now give a sanity check which shows that split $T$-submersions are, in fact, $T$-submersions whenever their underlying map $f$ is $p$-carrable.
\begin{proposition}\label{Prop: Split Tsubmersion}
Let $f:X \to Y$ be a $p$-carrable morphism in a tangent category $\Cscr$. Then $f$ is a split $T$-submersion if and only if $\theta_f$ is a split $T$-epimorphism.
\end{proposition}
\begin{proof}
$\implies:$ Applying Proposition \ref{Prop: The lifting property characterization of TSmooth} shows that $f$ is a split $T$-submersion if and only if $\theta_f$ admits a section $s_f$. However, as retracts are absolute coequalizers, this implies that $\theta_f$ is a split $T$-epimorphism and hence a $T$-submersion. 

$\impliedby$: In the other direction, $\theta_f$ being a split $T$-epimorphism means that $T^m\theta_f$ is a split epimorphism for all $m \in \N$. However, this means that $\theta_{T^mf}$ is a split epimorphism for all $m \in \N$ by Corollary \ref{Cor: The other higher powers of T to thetaf in one line} and so that $\theta_{T^mf}$ is a retract and hence that the corresponding square
\[
\begin{tikzcd}
T^{m+1}X \ar[r]{}{T^{m+1}f} \ar[d, swap]{}{(p \ast T^m)_X} & T^{m+1}Y \ar[d]{}{(p \ast T^m)_Y} \\
T^mX \ar[r, swap]{}{T^mf} & T^mY
\end{tikzcd}
\]
is a $T$-weak pullback again by Proposition \ref{Prop: The lifting property characterization of TSmooth}.
\end{proof}

We now present the calculations of split $T$-submersions for our main examples. Note that because classifying split $T$-submersions in affine schemes is quite involved and subtle, it will involve a nontrivial digression from the abstract development of split $T$-submersions. As such, we have separated this into its own subsection below.

\begin{example}\label{Example: TSubmersion in SMan}
In $\SMan$, the split $T$-submersions are precisely the submersions whose horizontal descent admits a global section. To see this, note that as the horizontal descent $\theta_f$ exists in $\SMan$, by Proposition \ref{Prop: Split Tsubmersion} in order for a map $f:X \to Y$ be a split $T$-submersion, $f$ must be in particular a $T$-submersion for which $\theta_f$ is a retract. But this means that there is a morphism $s_f:X \times_Y TY \to TX$ for which
\[
\begin{tikzcd}
X \times_Y TY \ar[r]{}{s_f} \ar[dr, equals] & TX \ar[d]{}{\theta_f} \\
 & X \times_Y TY
\end{tikzcd}
\]
commutes. As submersions are precisely the maps for which $\theta_f$ admits local sections around every point $(x,\overrightarrow{v})$ in $X \times_Y TY$, the existence of $s_f$ means that upon making suitable compatibility choices, these local sections may be stictched together to form a global section of $\theta_f$. In particular, this states that $f$ is a split $T$-submersion if and only if $\theta_f$ has the property that there exists a family of local sections to $\theta_f$ for which the collection of the $s_f$ admit a smooth partition of unity subordinate to the induced open cover of $f^{\ast}TY$.
\end{example}
\begin{remark}
In light of Corollary \ref{Cor: Section Submersion: thetaf section iff linear section for adjoints} below, another way to compute the split $T$-submersions in $\SMan$ is to first remark that since $\SMan$ is a Rosick{\'y} tangent category, the map $\theta_f$ admits a section if and only if it admits a linear section. With this observation and Proposition \ref{Prop: Split Tsubmersion}, we see that $f:X \to Y$ is a split $T$-submersion if and only if the cotangent sequence
\[
\begin{tikzcd}
X \ar[r] & T_{X/Y} \ar[r] & TX \ar[r]{}{\theta_f} & f^{\ast}(TY) \ar[r] & X
\end{tikzcd}
\]
is right split in $\DBun(X) \simeq \mathbf{Vec}(X)$. Consequently, the split $T$-submersions are precisely those submersions whose horizontal descent admits a linear global section.
\end{remark}

\begin{example}\label{Example: REgular Epi submersion SMan Paracompact}
In the category $\SManPC$ of paracompact Hausdorff smooth manifolds, every $T$-submersion is a split $T$-submersion. This is because if $f:X \to Y$ is a morphism of paracompact smooth manifolds, then the paracompactness of both $X$ and $Y$ force $TX, TY$, and $f^{\ast}(TY)$ to be paracompact as well. As $TX$ and $f^{\ast}(TY)$ are paracompact, this means that any open cover of either space admits a locally finite smooth partition of unity subordinate to each cover and so allows the stitching of suitably compatible local sections of $\theta_f$ into global sections.
\end{example}

\begin{example}\label{Example: Tsubmersion in CAlgR}
Let $R$ be a commutative rig. In $\CAlg{R}$ the split $T$-submersions are precisely the surjections $f:A \to B$ with a linear section. To see this recall that given the naturality square
\[
\begin{tikzcd}
A[\epsilon] \ar[d, swap]{}{p_A} \ar[r]{}{f[\epsilon]} & B [\epsilon] \ar[d]{}{p_B} \\
A \ar[r, swap]{}{f} & B
\end{tikzcd}
\]
the pullback $A \times_B B[\epsilon] \cong A \ltimes B$ with the map $\theta_f:A[\epsilon] \to A \ltimes B$ given by $\theta_f(a+x\epsilon) = (a,f(x))$. Now observe that any map $\sigma:A \ltimes B \to A[\epsilon]$ takes the form $\sigma(r,s) = \sigma_0(r,s) + \sigma_1(r,s)\epsilon$ for some maps $\sigma_0, \sigma_1:A \ltimes B \to A$. Thus if there is a $\sigma$ rendering the diagram
\[
\begin{tikzcd}
A \ltimes B \ar[r]{}{\sigma} \ar[dr, equals] & A[\epsilon] \ar[d]{}{\theta_f} \\
 & A \ltimes B
\end{tikzcd}
\]
commutative then we have that $\id_{A \ltimes B} = \theta_f \circ \sigma$. This allows us to deduce:
\begin{prooftree}
    \AxiomC{$\id_{A \ltimes B} = \theta_f \circ \sigma$}
    \UnaryInfC{$\forall\,(a,b) \in A \ltimes B.\,(a,b) = \theta_f(\sigma_0(a,b)+\sigma_1(a,b)\epsilon) = (\sigma_0(a,b),f(\sigma_1(a,b)))$}
    \UnaryInfC{$a = \sigma_0(a,b)$ and $b = f(\sigma_1(a,b))$.}
\end{prooftree}
Thus it must be the case that for every $b$ there is an $x \in A$ (namely $x = \sigma_1(a,b)$) for which $f(x) = b$. This implies in particular that $f$ is surjective. We now show that $\sigma_1$ must be a derivation. Note that since $\sigma$ is a rig morphism, we deduce:
\begin{prooftree}
 \AxiomC{$\forall\,(a,b), (x,y) \in A \ltimes B.\, \sigma(a,b) + \sigma(x,y)  = \sigma((a,b) + (x,y)) = \sigma(a+x, b+y)$}
 \UnaryInfC{$\sigma_0(a+x, b+y) = \sigma_0(a,b) + \sigma_0(x,y)$ and $\sigma_1(a,b) + \sigma_1(x,y) = \sigma_1(a+x, b+y)$}
\end{prooftree}
Thus $\sigma_1((a,b)+(x,y)) = \sigma_1(a,b) + \sigma_1(x,y)$ and in particular $\sigma_1(0,0) = 0$, showing that is additive. Additionally because $\sigma$ is a rig morphism we deduce:
\begin{prooftree}
 \AxiomC{$\forall\,(a,b), (x,y) \in A \ltimes B.\, \sigma(a,b)\sigma(x,y)  = \sigma((a,b)(x,y))$}
 \UnaryInfC{$(\sigma_0(a,b)+\sigma_1(a,b)\epsilon)(\sigma_0(x,y)+\sigma_1(x,y)\epsilon)) = \sigma(ax,f(a)y + bf(x))$}
 \UnaryInfC{$\sigma_0(a,b)\sigma_1(x,y) +\sigma_1(a,b)\sigma_0(x,y) = \sigma_1(ax,f(a)y+bf(x))$}
\end{prooftree}
Finally, because rig morphisms preserve multiplicative units, we also deduce:
\begin{prooftree}
    \AxiomC{$\sigma(1,0) = 1_{A[\epsilon]}$}
    \UnaryInfC{$1 + 0 \epsilon = \sigma(1,0) = \sigma_0(1,0) + \sigma_1(1,0)\epsilon$}
    \UnaryInfC{$\sigma_1(1,0) = 0$}
\end{prooftree}
This in particular shows that $\sigma_1:A \ltimes B \to A$ is a derivation. Furthermore, consider the map $s:B \to A$ defined via the composition
\[
\begin{tikzcd}
B \ar[rr]{}{b \mapsto (0,b)} \ar[rr, swap]{}{\operatorname{incl}} \ar[drr, swap]{}{s} & & A \ltimes B \ar[d]{}{\sigma_1} \\
 & & A
\end{tikzcd}
\]
of $A$-modules; note that the map $b \mapsto (0,b)$ is a map of $A$-modules because for all $a \in A$ and for all $b \in B$
\[
a \bullet (0,b) = (a,0)(0,b) = (a\cdot 0, f(a)b + 0^2) = (0,a\bullet b).
\]
We compute that for all $b \in B$ on one hand
\begin{align*}
(\theta_f \circ s)(b) = \big(\theta_f \circ \sigma\big)(\operatorname{incl}(b)) = \operatorname{incl}(b) = (0,b)
\end{align*}
while on the other hand
\begin{align*}
(\theta_f \circ s)(b) = \theta_f\big(\sigma(\operatorname{incl}(b))\big) = \theta_f\big(\sigma_1(0,b)\epsilon\big) = \theta_f\big(0+\sigma_1(0,b)\epsilon\big) =  \bigg(0,f\big(\sigma_1(0,b)\big)\bigg) = \bigg(0,f\big(s(b)\big)\bigg).
\end{align*}
Thus we conclude, by comparing $B$-coordinates, that $f(s(b)) = b$ and so that $s:B \to A$ is a linear section of $A$.
\end{example}

\begin{example}\label{Example: TSubmersion CDC}
In a CDC, for a map $f: X \to Y$, because $\theta_f = \langle \pi_0, D[f]\rangle$, we get that any section $s:f^{\ast}(TY) \to TX$ has the form:
\[
\begin{tikzcd}
X \times Y \ar[drr, equals] \ar[rr]{}{\langle s_0, s_1 \rangle} & & X \times X \ar[d]{}{\langle \pi_0, D[f]\rangle} \\
& & X \times Y
\end{tikzcd}
\]
This tells us immediately that $s_0 = \id_X$. It also tells us that $D[f] \circ s_1 = \id_Y$ and so shows that it is necessary and sufficient for $D[f]$ to be a retract. In particular, if $f$ is linear and hence $\theta_f = \id_X \times f$, then $f$ is a split $T$-submersion if and only if $f$ is a retract.
\end{example}

\subsection{Split Submersions in Algebraic Geometry}\label{Subsection: TSubmersions in Algebraic Geometry}

Because knowing which maps of schemes $f:X \to Y$ make the cotangent sequence
\[
\begin{tikzcd}
X \ar[r] & T_{X/Y} \ar[r] & TX \ar[r] & f^{\ast}TY \ar[r] & X
\end{tikzcd}
\]
a \emph{split} exact sequence is of a high level of interest in algebrac geometry (it is necessary when checking formal smoothness, for instance, by the Jacobian Criterion; cf., for instance, \cite[Th{\'e}or{\`e}me 22.6.1]{EGA04}) for an affine version of the criterion) it is of interest to classify the split $T$-submersions in the categories of $S$-schemes and affine schemes $\CAlg{R}^{\op}$. This is of additional interest because the definition of formal smoothness of morphisms of schemes given in \cite[D{\'e}finition 17.1.1]{EGA44} suggests that formally smooth morphisms of schemes are, in effect, the submersions of schemes (or at least that is an often stated metaphor). However, we will see below that while formally smooth morphisms of schemes \emph{are} split $T$-submersions, it need not be the case that all split $T$-submersions of schemes are formally smooth (or even formally {\'e}tale).

In order to classify the split $T$-submersions in the categories $\CAlg{R}^{\op}$ (and more generally in $\Sch_{/S}$ for a base scheme $S$) we need to recall what it means for a morphism $f:A \to B$ of commutative rings to be formally smooth. We also need to recall a weaker and less-well-known notion of formal smoothness which commutative rings can satisfy developed in \cite[Section 19.9]{EGA04}. This is a notion which defines when a commutative algebra morphism $g:A \to B$ is formally smooth relative to a base ring $R$; it is a relatively weak definition and, despite the way the term ``relative to $(-)$'' is used in all of EGA and SGA, \emph{not} the same thing as $g:A \to B$ being formally smooth in $\CAlg{R}$.

\begin{definition}[{\cite[D{\'e}finition 19.3.1]{EGA04}}]\label{Defn: Formally Smooth}
Lef $f:A \to B$ be a morphism of commutative rings. We say that $f$ is \emph{formally smooth} if $f$ satisfies the weak left lifting property against nilpotent ideal quotient maps. That is, given any commutative square
\[
\begin{tikzcd}
A \ar[r]{}{\alpha} \ar[d, swap]{}{f} & C \ar[d]{}{\pi_{I}} \\
B \ar[r, swap]{}{\beta} & \frac{C}{I}
\end{tikzcd}
\]
with $I \trianglelefteq C$ a nilpotent ideal, there exists a ring map $\ell:B \to C$ making the diagram
\[
\begin{tikzcd}
A \ar[r]{}{\alpha} \ar[d, swap]{}{f} & C \ar[d]{}{\pi_{I}} \\
B \ar[ur, dashed]{}{\ell} \ar[r, swap]{}{\beta} & \frac{C}{I}
\end{tikzcd}
\]
commute. If instead the diagrams and lifts above occur in the category $\CAlg{R}$ for a commutative ring $R$, then $f$ is \emph{a formally smooth map of $R$-algebras}.
\end{definition}
An important way to view formally smooth maps is in terms of their characterization by K{\"a}hler differentials via the Jacobian Criterion for Formal Smoothness. This allows one to use the splitting of the relative conormal sequence to detect formal smoothness  --- it says that a map is formally smooth if and only if the map is formally nonsingular in the sense that locally/for a choice of local coordinates the Jacobian determinants of the transformation is invertible.

\begin{theorem}[Jacobian Criterion for Formal Smoothness; {\cite[Th{\'e}or{\`e}me 22.6.1]{EGA04}, \cite[Proposition 16.12]{Eisenbud}, \cite[Proposition 10.138.8]{stacks-project}}]\label{Thm: }
Let $f:A \to B$ be a morphism of commutative rings and let $h:A \to P$ be a formally smooth $A$-algebra with a surjection $p:P \to B$ making the diagram
\[
\begin{tikzcd}
 & P \ar[dr]{}{p} \\
A \ar[rr, swap]{}{f} \ar[ur]{}{h} & & B
\end{tikzcd}
\]
commute. Then $f:A \to B$ is formally smooth (in $\CAlg{R})$ if and only if the exact sequence
\[
\begin{tikzcd}
\frac{\operatorname{Ker}(p)}{\operatorname{Ker}(p)^2} \ar[r] & \Kah{P}{A} \otimes_{P} B \ar[r] & \Kah{B}{A} \ar[r] & 0
\end{tikzcd}
\]
is split exact in $\Mod{B}$.
\end{theorem}

As we will see below, formally smooth maps relative to a base $R$ are distinct in that they are characterized not by the splitting of the conormal sequence but instead by the splitting of the cotangent sequence.

\begin{definition}[{\cite[D{\'e}finition 19.9.1]{EGA04}}]\label{Defn: Formally Smooth relative to R}
Let $R \xrightarrow{f} A \xrightarrow{g} B$ be maps of commutative rings. We say that $g:A \to B$ is a \emph{formally smooth $A$-algebra relative to $R$} if and only if for any commutative diagram
\[
\begin{tikzcd}
A \ar[r]{}{\alpha} \ar[d, swap]{}{f} & C \ar[d]{}{\pi_{I}} \\
B \ar[r, swap]{}{\beta} & \frac{C}{I}
\end{tikzcd}
\]
in $\CAlg{A}$, if there is a lift $\ell_R$ making the diagram
\[
\begin{tikzcd}
R \ar[r]{}{\alpha \circ f} \ar[d, swap]{}{g \circ f} & C \ar[d]{}{\pi_{I}} \\
B \ar[r, swap]{}{\beta} \ar[ur, dashed]{}{\ell_R} & \frac{C}{I}
\end{tikzcd}
\]
commute in $\CAlg{R}$ then there is a lift $\ell_A$ making the diagram
\[
\begin{tikzcd}
A \ar[r]{}{\alpha} \ar[d, swap]{}{f} & C \ar[d]{}{\pi_{I}} \\
B \ar[r, swap]{}{\beta} \ar[ur, dashed]{}{\ell_A} & \frac{C}{I}
\end{tikzcd}
\]
commute.
\end{definition}
The way to view formally smooth maps relative to a base is in terms of their characterization by way of the cotangent sequence. However, we will use this below to give a cautionary example showing that formally smooth morphisms relative to a base $R$ need \emph{not} be formally smooth morphisms (even in $\CAlg{R}$). It is true, however, that formally smooth maps are formally smooth relative to $R$: this follows from a straightforward check directly from definitions by reusing any given lift $\ell$ provided by the formally smooth map.
\begin{theorem}[{\cite[Th{\'e}or{\`e}me 20.5.7.ii]{EGA04}}]\label{Thm: EGA 0 4 20 5 7}
Let $R \xrightarrow{f} A \xrightarrow{g} B$ be maps of commutative rings. The map $g:A \to B$ is formally smooth relative to $R$ if and only if the relative cotangent sequence
\[
\begin{tikzcd}
\Kah{A}{R} \otimes_A B \ar[r] & \Kah{B}{R} \ar[r] & \Kah{B}{A} \ar[r] & 0
\end{tikzcd}
\]
is split exact.
\end{theorem}
\begin{example}\label{Example: Classification of TSubmersions in CAlgRop}
Let $R$ be a commutative ring. By virtue of the equivalence $\DBun(\Spec R) \simeq \Mod{R}^{\op}$ asserted by \cite{GeoffJSDiffBunComAlg}, we deduce the following chains of equivalences:
\begin{prooftree}
    \AxiomC{$\theta_f:T_{B/R} \to f^{\ast}(T_{A/R})$ has a section in $\DBun(\Spec B)$}
    \UnaryInfC{$\theta_f:\Spec(\Sym_{B}(\Kah{B}{R})) \to \Spec(\Sym_A(\Kah{A}{R})) \times_{A} \Spec B$ has a section in $\DBun(\Spec B)$}
    \UnaryInfC{$\theta_f:\Spec(\Sym_{B}(\Kah{B}{R})) \to \Spec(\Sym_B(\Kah{A}{R} \otimes_{A} B))$ has a section in $\DBun(\Spec B)$}
    \UnaryInfC{$\Kah{A}{R} \otimes_A B \to \Kah{B}{R}$ admits a retract in $\Mod{B}$}
    \UnaryInfC{$0 \to \Kah{A}{R} \otimes_A B \to \Kah{B}{R} \to \Kah{A}{B} \to 0$ is split exact in $\Mod{B}$}
\end{prooftree}
As such, invoking \cite[Th{\'e}or{\`e}me 20.5.7.ii]{EGA04} gives that $g:A \to B$ is formally smooth relative to $R$ if and only if $\Spec g:\Spec B \to \Spec A$ is a split $T$-submersion in $\mathbf{AffSch}_{/R} \simeq \CAlg{R}^{\op}$.
\end{example}

\begin{example}\label{Example: Nonexample of formally smooth}
It is unfortunately not the case that split $T$-submersions in $\mathbf{CAlg}_{R}^{\op}$ coincide with formally smooth morphisms of rings. To see an explicit example of this, note that for any commutative ring $R$, $\Kah{R}{R} \cong 0$ so $\Sym_{R}(\Kah{R}{R}) \cong R$ and $T_{R/R} \cong R$. In particular, for any commutative $R$-algebra $A$, the diagram
\[
\begin{tikzcd}
\Spec A \ar[r]{}{!_A} \ar[d, equals] & T_{R/R} \ar[d] \\
\Spec A \ar[r, swap]{}{!_A} & \Spec R
\end{tikzcd}
\]
is a pullback diagram. But then $!_A$ is always a split $T$-submersion because the natural map $\theta_{!_A} = p_A:T_{A/R} \to \Spec A$ is a retract with linear section $0_A:\Spec A \to T_{A/R}$. As such, in particular, any non-smooth $R$-algebra $A$ is a split $T$-submersion over $R$. For a very explicit example, set $R =  \Fbb_2$ and $A = \Fbb_2[x]/(x^2)$ with the evident algebra map. Then because a morphism of rings is smooth if and only if it is formally smooth and finitely presented, it suffices to check that $\Fbb_2[x]/(x^2)$ is not a formally smooth $\Fbb_2$-algebra. However, as
\[
\operatorname{Jac}\left(\frac{\Fbb_2[x]}{(x^2)}\right) = \det\begin{pmatrix}
\frac{\mathrm{d}(x^2)}{\mathrm{d} x}
\end{pmatrix} = \det(0) = 0
\]
the algebra $\Fbb_2[x]/(x^2)$ is not smooth because it violates the Jacobian Criterion.
\end{example}

Due to some technical difficulties which may be avoided by Proposition \ref{Prop: Linear sections from nonlinear sections with adjoints and units}, we defer computing the split $T$-submersions in $\Sch_{/S}$ until Example \ref{Example: Submersions in Schemes} below. 

\begin{example}
If $R$ is a commutative rig instead of a commutative ring, it is still true from \cite{GeoffJSDiffBunComAlg} and Corollary \ref{Cor: Schemes and affine rig schemes are linearly submersive} below\footnote{The referenced corollary here allows us to show that in $\CAlg{R}^{\op}$, $\theta_f:TX \to f^{\ast}(TY)$ admits a section in $\Cscr_{/X}$ if and only if it admits a section in $\DBun(X)$. Note that it does not say that every section is linear, but instead that every section has a linear companion.} and the equivalence of categories $\DBun(A)^{\op} \simeq \Mod{A}$ (for all commutative $R$-algebras $A$) that for a map $f:B \to A$ in $\CAlg{R}^{\op}$, $\theta_f:T_{B/R} \to f^{\ast}(T_{A/R})$ admits a section if and only if the map of $B$-modules $\Kah{A}{R} \otimes_A B \to \Kah{B}{R}$ admits a retract $r:\Kah{B}{R} \to \Kah{A}{R} \otimes_A B$ in the category $\Mod{B}$.
\end{example}

\subsection{Linear Sections to Split Submersions}
Part of the reason that we have spent a great deal of time defining and studying the structure of $p$-carrable morphisms in this paper is that the horizontal descent $\theta_f:TX \to f^{\ast}(TY)$ allows us to use differntial bundles as a geometric tool with which to study how close or how far $f$ is from being a categorical local diffeomorphism. We have seen already that when studying $p$-carrable morphisms which are also split $T$-submersions, we have access to a section/retraction pair
\[
\begin{tikzcd}
f^{\ast}(TY) \ar[r]{}{s} \ar[dr, equals] & TX \ar[d]{}{\theta_f} \\
 & f^{\ast}(TY)
\end{tikzcd}
\]
However, there is one major structural question which underlies $p$-carrable split $T$-submersions at every corner: how linear is the section $s:f^{\ast}(TY) \to TX$? Certainly $s$ is a bundle map, as since
\[
\begin{tikzcd}
f^{\ast}(TY) \ar[r]{}{s} \ar[dr, equals] & TX \ar[d]{}{\theta_f} \\
 & f^{\ast}(TY)
\end{tikzcd}
\]
the diagram
\[
\begin{tikzcd}
 & TX \ar[dd, near start]{}{p_X} \ar[dr]{}{\theta_f} \\
f^{\ast}(TY) \ar[dr, swap]{}{\pi_0} \ar[ur]{}{s} \ar[rr, equals, crossing over] & & f^{\ast}(TY) \ar[dl]{}{\pi_0} \\
 & X
\end{tikzcd}
\]
also commutes. However, it need not be the case that $s$ itself is directly linear. What we instead ask is the following general question: ``given a $p$-carrable split $T$-submersion $f:X \to Y$, if we have corresponding section $s$ to $\theta_f$, can we build a \emph{linear} section of $\theta_f$ out of $s$?'' We will present some partial answers to this question below and indicate that, for the case of Rosick{\'y} tangent categories and for the opposite categories of commutative algebras $\CAlg{R}^{\op}$ over a commutative rig $R$, it is the case that we can find a ``linearization'' or ``linear component'' of $s$, of sorts.

The first cases we consider for building linear versions of given sections are those which mirror the local algebraic structure that the symmetric algebra $\Sym_{X}$ and underlying module $\operatorname{Und}_{X}$ induce, via the opposite equivalences $\DBun(\Spec A)^{\op} \simeq \Mod{A}$ of \cite{GeoffJSDiffBunComAlg}, on the corresponding differential bundle categories $\Sym_{X}^{\op}:\DBun(X)^{\op} \to \CAlg{X}^{\op}$ and $\operatorname{Und}_{X}^{\op}:\CAlg{X}^{\op} \to \DBun(X)^{\op}$. The point here is that there are three key ingredients to this situation:
\begin{enumerate}[{\em (i)}] 
    \item In $\CAlg{R}^{\op}$, every map is $p$-carrable.
    \item In $\CAlg{R}^{\op}$, $\Sym_{A}$ (the functor whose opposite forgets the differential bundle over $A$ structure) is left adjoint to the underlying module functor $\operatorname{Und}_A$ (the functor whose opposite takes a sliced object and promotes it to a differential bundle).
    \item The unit of adjunction $M \to \operatorname{Und}_A(\Sym_A(M))$ is a pointwise split monic in $\Mod{A}$, i.e., the left adjoint $\Sym_A$ is faithful.
\end{enumerate}
Abstracting these appropriately and taking opposites where indicated allows us to produce the following situation in which we know we can produce linear sections out of general sections.
\begin{proposition}\label{Prop: Linear sections from nonlinear sections with adjoints and units}
In a tangent category $\Cscr$, let $f: X \to Y$ be a map such that: 
\begin{enumerate}[{\em (i)}] 
    \item The forgetful functor $\operatorname{Forget}_X:\DBun(X) \to \Cscr_{/X}$ has a left adjoint $F_X:\Cscr_{/X} \to \DBun(X)$.
    \item The counit of adjunction $\epsilon$ for $F_X \dashv \operatorname{Forget}_X$ is a pointwise split epimorphism.
    \item The map $f:X \to Y$ is $p$-carrable.
\end{enumerate}
Then for any section of $\theta_f$, $s: f^{\ast}(TY) \to TX$, there admits a linear section $s_f$ of $\theta_f$ (so a section in $\DBun(X)$). In particular, if $f$ is a split $T$-submersion then there exists a linear section of $\theta_f$ witnessing that $f$ is a split $T$-submersion.
\end{proposition}
\begin{proof}
Begin by observing that since $\theta_f$ is a linear morphism of differential bundles over $X$, we can rewrite the diagram expressing that $s$ is a section $s$ as:
\[
\begin{tikzcd}
\operatorname{Forget}_X(f^{\ast}(TY)) \ar[r]{}{s} \ar[dr, equals] & \operatorname{Forget}_X(TX) \ar[d]{}{\operatorname{Forget}_X(\theta_f)} \\
 & \operatorname{Forget}_X(f^{\ast}(TY))
\end{tikzcd}
\]
Now apply the left adjoint $F_X$ to the diagram above to produce the commuting diagram:
\[
\begin{tikzcd}
F_X\left(\operatorname{Forget}_X(f^{\ast}(TY))\right) \ar[drr, equals] \ar[rr]{}{F_X(s)} & & F_X\left(\operatorname{Forget}_X(TX)\right) \ar[d]{}{F_X(\operatorname{Forget}_X(\theta_f)} \\
 & & F_X\left(\operatorname{Forget}_X(f^{\ast}(TY)\right)
\end{tikzcd}
\]
Let $\sigma_{f^{\ast}(TY)}:f^{\ast}(TY) \to F_X(\operatorname{Forget}_X(f^{\ast}(TY)))$ denote a given splitting of the counit
\[
\begin{tikzcd}
f^{\ast}(TY) \ar[rr]{}{\sigma_{f^{\ast}(TY)}} \ar[drr, equals] & & F_X\left(\operatorname{Forget}_X\big(f^{\ast}(TY)\big)\right) \ar[d]{}{\epsilon_{f^{\ast}(TY)}} \\
 & & f^{\ast}(TY)
\end{tikzcd}
\]
and define the linear map of differential bundles $s_f$ via the composition:
\[
\begin{tikzcd}
f^{\ast}(TY) \ar[d, swap]{}{s_f} \ar[rr]{}{\sigma_{f^{\ast}(TY)}} & & F_X(\operatorname{Forget}_X(f^{\ast}(TY))) \ar[d]{}{F_X(s)} \\
TX & & F_X(\operatorname{Forget}_X(TX)) \ar[ll]{}{\epsilon_{TX}}
\end{tikzcd}
\]
We then compute that
\begin{align*}
\theta_f \circ s_f &= \theta_f \circ \epsilon_{TX} \circ F_X(s) \circ \sigma_{f^{\ast}(TY)} = \epsilon_{f^{\ast}(TY)} \circ F_X\left(\operatorname{Forget}_X(\theta_f)\right) \circ F_X(s) \circ \sigma_{f^{\ast}(TY)} = \epsilon_{f^{\ast}(TY)} \circ \sigma_{f^{\ast}(TY)} \\
&= \id_{f^{\ast}(TY)} 
\end{align*}
so $s_f$ is indeed a section to $\theta_f$. Moreover, $s_f$ is linear exactly because it is defined as a composite of maps in $\DBun(X)$. The final claim of the proposition is immediate once we know that all powers $T^m$ preserve the weak pullback which makes $f$ a split $T$-submersion.
\end{proof}
\begin{corollary}\label{Cor: Section Submersion: thetaf section iff linear section for adjoints}
In a tangent category $\Cscr$, let $f: X \to Y$ be a $p$-carrable map such that: 
\begin{enumerate}[{\em (i)}] 
    \item The forgetful functor $\operatorname{Forget}_X:\DBun(X) \to \Cscr_{/X}$ has a left adjoint $F_X:\Cscr_{/X} \to \DBun(X)$.
    \item The counit of adjunction $\epsilon$ for $F_X \dashv \operatorname{Forget}_X$ is a pointwise split epimorphism.
\end{enumerate}
Then the horizontal descent map $\theta_f:TX \to f^{\ast}(TY)$ admits a section if and only if it admits a \emph{linear} section.
\end{corollary}

We now show that both our tangent categories of affine schemes and schemes satisfy the conditions of the corollary. 

\begin{corollary}\label{Cor: Schemes and affine rig schemes are linearly submersive}
The categories $\CAlg{R}^{\op}$ (for a commutative rig $R$) and $\Sch_{/S}$ (for a base scheme $S$) have the property that for any map $f:X \to Y$, the horizontal descent $\theta_f:TX \to f^{\ast}(TY)$ admits a section if and only if it admits a linear section.
\end{corollary}
\begin{proof}
In both cases the functor $\operatorname{Forget}_X:\DBun(X) \to \Cscr_{/X}$ arises either as 
\[
\operatorname{Forget}_X \cong \underline{\Spec}_X(\underline{\Sym}_{\Ocal_X}(-)):\QCoh(X)^{\op} \to \Sch_{/X}
\]
in the case of an $S$-scheme $X$ or as
\[
\operatorname{Forget}_X \cong \Sym_{X}^{\op}(-):\Mod{X}^{\op} \to \CAlg{X}^{\op} = (\CAlg{R}^{\op})_{/X}.
\]
when $X$ is an object of $\CAlg{R}^{\op}$. In both cases the functors
\[
\Sym_X:\Mod{X} \to \CAlg{X}, \qquad \underline{\Sym}_{\Ocal_X}:\QCoh(X) \to \QCoh(X,\CAlg{\Ocal_X})
\]
are left adjoints whose unit $\eta$ is pointwise split. Taking the opposite of the rig-theoretic case or the relative spectrum of the scheme-theoretic case implies that $\operatorname{Forget}_X$ is a right adjoint with a pointwise split counit in all cases.
\end{proof}
\begin{example}\label{Example: Submersions in Schemes}
Let $S$ be a base scheme. Then in $\Sch_{/S}$ a map $f:X \to Y$ is a split $T$-submersion if and only if the relative cotangent sequence
\[
\begin{tikzcd}
f^{\ast}\Kah{Y}{S} \ar[r]{}{v_f} & \Kah{X}{S} \ar[r] & \Kah{X}{Y} \ar[r] & 0
\end{tikzcd}
\]
is split exact in $\QCoh(X)$. This follows from the fact that $(\underline{\Spec}_X \circ \underline{\Sym}_{\Ocal_X})(v_f) = \theta_f$ and by Corollary \ref{Cor: Schemes and affine rig schemes are linearly submersive}, $\theta_f$ admits a section if and only if $v_f$ admits a section in $\QCoh(X)^{\op}$.
\end{example}

We now show that in a Rosick{\'y} tangent category, we can build linear sections from any section of the horizontal descent of a $p$-carrable map. We will need to do this in two steps: first we will show that in a Rosick{\'y} tangent category if $s:f^{\ast}(TY) \to TX$ is a section of the horizontal descent $\theta_f$, the map $s - s \circ f^{\ast}(0_Y) \circ \pi_0:f^{\ast}(TY) \to TX$ is a section  of $\theta_f$ which preserves the zero sections of the bundle map. Second we will have to show that if we have a zero-preserving section $s:f^{\ast}(TY) \to TX$ of the horizontal descent $\theta_f$ then there is a corresponding linear section $s_f$ of $\theta_f$.

\begin{lemma}\label{Lemma: From section to zero pres section}
Let $f:X \to Y$ be a $p$-carrable map in a Rosick{\'y} tangent category $\Cscr$. Then if $s:f^{\ast}(TY) \to TX$ is a section of the horizontal descent $\theta_f$, the map $s - s \circ f^{\ast}(0_Y) \circ \pi_0$ is a section of $\theta_f$ which preserves the zero section of $f^{\ast}(TY)$ and $TX$.
\end{lemma}
\begin{proof}
Begin by observing that by assumption the diagram
\[
\begin{tikzcd}
 & TX \ar[dd, near start]{}{p_X} \ar[dr]{}{\theta_f} \\
f^{\ast}(TY) \ar[dr, swap]{}{\pi_0} \ar[rr, equals] \ar[ur]{}{s} & & f^{\ast}(TY) \ar[dl]{}{\pi_0} \\
 & X
\end{tikzcd}
\]
commutes and so $s$ is a map in the slice category $\Cscr_{/X}$. We also compute that in $\Cscr_{/X}$
\[
p_X \circ (s - s \circ f^{\ast}(0_Y) \circ \pi_0) = p_X \circ s - p_X \circ s \circ f^{\ast}(0_Y) \circ \pi_0 = \pi_0 - \pi_0 \circ f^{\ast} \circ \pi_0 = \pi_0 - \pi_0 = \pi_0
\]
where the last equality holds because $\pi_0:f^{\ast}(TY) \to X$ is the map from $f^{\ast}(TY)$ to the terminal object in $\Cscr_{/X}$. We also compute that
\[
(s - s \circ f^{\ast}(0_Y) \circ \pi_0) \circ f^{\ast}(0_Y) = s\circ f^{\ast}(0_Y) - s \circ f^{\ast}(0_Y) \circ \pi_0 \circ f^{\ast}(0_Y) = s \circ f^{\ast}(0_Y) - s \circ f^{\ast}(0_Y) = 0_X
\]
which shows that the diagram
\[
\begin{tikzcd}
f^{\ast}(TY) \ar[rr]{}{s - s \circ f^{\ast}(0_Y) \circ \pi_0} & & TX \\
 & X \ar[ur, swap]{}{0_X} \ar[ul]{}{f^{\ast}(0_Y)}
\end{tikzcd}
\]
commutes. Note that the last equality uses the fact that $0_X$ is the unit of the bundle structure on $TX$ and the fact that the difference of the maps $\alpha - \alpha$ of morphisms $\alpha:X \to TX$ has the property that
\[
\begin{tikzcd}
X \ar[r, swap, shift right = 1]{}{\alpha - \alpha} \ar[r, shift left = 1]{}{0_X} & TX
\end{tikzcd}
\]
commutes. Finally we also compute that
\[
\theta_f \circ (s - s \circ f^{\ast}(0_Y) \circ \pi_0) = \theta_f \circ s - \theta_f \circ s \circ f^{\ast}(0_Y) \circ \pi_0 = \id - \theta_f \circ 0_X \circ \pi_0 = \id - f^{\ast}(0_Y) \circ \pi_0 = \id
\]
because $f^{\ast}(0_Y) \circ \pi_0:f^{\ast}(TY) \to f^{\ast}(TY)$ is the zero map on $f^{\ast}(TY)$. Thus $s - s \circ f^{\ast}(0_Y) \circ \pi_0$ is a section to the horizontal descent $\theta_f$ which preserves the zero section $f^{\ast}(0_Y)$ of $f^{\ast}(TY)$.
\end{proof}

\begin{proposition}\label{Prop: Rosicky tan cats all admit linear sections once the section preserves the zero}
Let $f:X \to Y$ be a $p$-carrable map in a Rosick{\'y} tangent category $\Cscr$. Suppose that the horizontal descent $\theta_f:TX \to f^{\ast}(TY)$ admits a section $s:f^{\ast}(TY) \to TX$ which preserves the zero section, i.e., $s \circ f^{\ast}(0_Y) = 0_X$. Then there exists a linear section $s_f:f^{\ast}(TY) \to TX$ of $\theta_f$.
\end{proposition}
\begin{proof}
We construct the section $s_f$ as follows. Recall that from the universal property of the vertical lift, we get that
\[
\begin{tikzcd}
TX \ar[r]{}{\ell_X} & T^2X \ar[rrr, bend right = 30, swap]{}{(T \ast p)_X} \ar[rrr,bend left = 30]{}{(p \ast T)_X} \ar[r]{}{(p \ast T)_X} & TX \ar[r]{}{p_X} & X \ar[r]{}{0_X} & TX
\end{tikzcd}
\]
is an equalizer diagram. We will use this universal property to show that we can induced a unique section $s_f$ from $f^{\ast}(TY) \to TX$ which equalizes the three listed maps. To this end we must first prove that
\[
\begin{tikzcd}
f^{\ast}(TY) \ar[rr]{}{Ts \circ \lambda} & & T^2X \ar[rrr, bend right = 30, swap]{}{(T \ast p)_X} \ar[rrr,bend left = 30]{}{(p \ast T)_X} \ar[r]{}{(p \ast T)_X} & TX \ar[r]{}{p_X} & X \ar[r]{}{0_X} & TX
\end{tikzcd}
\]
commutes where $\lambda = f^{\ast}(\ell_Y)$ is the lift of the differential bundle structure on $f^{\ast}(TY)$ over $X$. We first compute that by virtue of $\lambda = f^{\ast}(\ell_Y)$,
\[
(T \ast p)_X \circ Ts \circ \lambda = T(p_X \circ s) \circ \lambda = T\pi_0 \circ \lambda = 0_X \circ \pi_0
\]
and
\[
0_X \circ p_X \circ (T \ast p)_X \circ Ts \circ \lambda = 0_X \circ p_X \circ 0_X \circ \pi_0 = (0 \ast T)_X \circ \pi_0
\]
Finally, because $s \circ f^{\ast}(0_Y) = 0_X$, we also get that
\[
(p \ast T)_X \circ Ts \circ \lambda = s \circ p_X \circ \lambda = s \circ f^{\ast}(0_Y) \circ \pi_0 = 0_X \circ \pi_0.
\]
Applying the universal property of the equalizer  gives rise to a unique map $s_f$ rendering the diagram
\[
\begin{tikzcd}
TX \ar[r]{}{\ell_X} & T^2X \ar[rrr, bend right = 30, swap]{}{(T \ast p)_X} \ar[rrr,bend left = 30]{}{(p \ast T)_X} \ar[r]{}{(p \ast T)_X} & TX \ar[r]{}{p_X} & X \ar[r]{}{0_X} & TX \\
f^{\ast}(TY) \ar[u, dashed]{}{\exists! s_f} \ar[r, swap]{}{\lambda} & T(f^{\ast}(TY)) \ar[u, swap]{}{Ts}
\end{tikzcd}
\]
We will prove that $s_f$ is the section we desire.

Let us first prove that $s_f$ is a linear bundle morphism. First by construction we have the following commuting diagram:
\[
\begin{tikzcd}
f^{\ast}(TY) \ar[d, swap]{}{s_f} \ar[r]{}{\lambda} & T(f^{\ast}(TY)) \ar[d]{}{Ts} \ar[r]{}{p_{f^{\ast}(TY)}} & f^{\ast}(TY) \ar[d]{}{s} \\
TX \ar[r, swap]{}{\ell_X} & T^2X \ar[r, swap]{}{(p \ast T)_X} & TX
\end{tikzcd}
\]
We additionally have a commuting diagram
\[
\begin{tikzcd}
X \ar[d, equals] \ar[dr]{}{0_X} \ar[r]{}{f^{\ast}(0_Y)} & f^{\ast}(TY) \ar[d]{}{s} \\
X & TX \ar[l]{}{p_X}
\end{tikzcd}
\]
because $s \circ f^{\ast}(0_Y) = 0_X$. We also compute on one hand that $p_X = p_X \circ (p \ast T)_X \circ \ell_X$ and that $f^{\ast}(0_Y) \circ \pi_0 = p_{f^{\ast}(TY)} \circ \lambda$. Putting these together gives
\begin{align*}
p_X \circ s_f &= p_X \circ (p \ast T)_X \circ \ell_X \circ s_f = p_X \circ (p \ast T)_X \circ Ts \circ \lambda = p_X \circ s \circ p_{f^{\ast}(TY)} \circ \lambda \\
&= p_X \circ s \circ f^{\ast}(0_Y) \circ \pi_0 = p_X \circ 0_X \circ \pi_0 = \pi_0.
\end{align*}
Thus the diagram
\[
\begin{tikzcd}
f^{\ast}(TY) \ar[rr]{}{s_f} \ar[dr, swap]{}{\pi_0} & & TX \ar[dl]{}{p_X} \\
 & X
\end{tikzcd}
\]
commutes and so $s_f$ is a bundle map. To see that it is a linear bundle map we note that $\ell_X$ is $T$-monic by Proposition \ref{Prop: Vertical Lift is TMonic} and so, in particular, $T\ell_X = (T \ast \ell)_X$ is monic. We use this to deduce that
\begin{align*}
T\ell_X \circ Ts_f \circ \lambda &= T^2s \circ T\lambda \circ \lambda = T^2s \circ \ell_{f^{\ast}(TY)} \circ \lambda = (\ell \ast T)_X \circ Ts \circ \lambda = (\ell \ast T)_X \circ \ell_X \circ s_f \\
&= T\ell_X \circ \ell_X \circ s_f.
\end{align*}
Using that $T\ell_X$ is monic gives that $\ell_X \circ s_f = Ts_f \circ \lambda$ and hence that $s_f$ is a linear bundle map.

We now prove that $s_f$ is a section of $\theta_f$. To see this note that since $f^{\ast}(TY)$ is a pullback, it suffices to prove that $\pi_0 \circ \theta_f \circ s_f = \pi_0$ and that $\pi_1 \circ \theta_f \circ s_f = \pi_1$. For the first equation we derive that
\[
\pi_0 \circ \theta_f \circ s_f = p_X \circ s_f = \pi_0
\]
because $s_f$ is a linear bundle map. For the second note that $\pi_1 \circ \theta_f \circ s_f = Tf \circ s_f$ and that $\ell_Y \circ Tf \circ s_f = s_f \circ \ell_X \circ T^2f$. Since $\lambda:f^{\ast}(TY) \to T(f^{\ast}(TY))$ has the form $\lambda = \langle 0_X \circ \pi_0, \ell_Y \circ \pi\rangle$ by \cite[Bottom of Page 12]{GeoffRobinBundle}, the diagram
\[
\begin{tikzcd}
f^{\ast}(TY) \ar[r]{}{\pi_1} \ar[d, swap]{}{\lambda} & TY \ar[d]{}{\ell_Y} \\
T(f^{\ast}(TY)) \ar[r, swap]{}{T\pi_1} & T^2Y
\end{tikzcd}
\]
commutes. This allows us to deduce that on one hand $\ell_Y \circ \pi_1 = T\pi_1 \circ \lambda$ while on the other hand, since $Tf \circ s = \pi_1$ by virtue of $s$ being a section to $\theta_f$,
\begin{align*}
\ell_Y \circ \pi_1 \circ \theta_f \circ s_f &= \ell_Y \circ Tf \circ s_f = T^2f \circ \ell_X \circ s_f = T^2f \circ Ts \circ \lambda = T(Tf \circ s) \circ \lambda = T\pi_1 \circ \lambda.
\end{align*}
Thus $\ell_Y \circ \pi_1 \circ \theta_f \circ s_f = T\pi_1 \circ \lambda = \ell_1 \circ \pi_1$ and so, once again by $\ell_Y$ being $T$-monic by Proposition \ref{Prop: Vertical Lift is TMonic}, we have that $\pi_1 \circ \theta_f \circ s_f = \pi_1$. This in turn allows us to deduce that $\theta_f \circ s_f = \id_{f^{\ast}TY}$ and so proves that $s_f$ is a linear section of $\theta_f$.
\end{proof}

Putting these results together we can deduce the following: in a Rosick{\'y} tangent category every $p$-carrable split $T$-submersion admits a \emph{linear} section of the horizontal descent, not just a nonlinear section.

\begin{corollary}\label{Cor: For a Rosicky tan cat T submersion on p carrable if and only if linear section to descent}
Let $\Cscr$ be a Rosick{\'y} tangent category and let $f:X \to Y$ be a $p$-carrable morphism. Then the horizontal descent $\theta_f:TX \to f^{\ast}(TY)$ admits a section in $\Cscr_{/X}$ if and only if it admits a section in $\DBun(X)$. In particular, $f$ is a split $T$-submersion if and only if $\theta_f$ is a retract in $\DBun(X)$.
\end{corollary}
\begin{proof}
For the $\implies$ direction, assume that $\theta_f$ admits a section $s:f^{\ast}(TY) \to TX$ in $\Cscr_{/X}$. Then applying Lemma \ref{Lemma: From section to zero pres section} shows that there is a section $\hat{s}:f^{\ast}(TY) \to TX$ in $\Cscr_{/X}$ which preserves the zero section $f^{\ast}(0_Y)$ in the sense that $\hat{s} \circ f^{\ast}(0_Y) = 0_X$. Finally applying Proposition \ref{Prop: Rosicky tan cats all admit linear sections once the section preserves the zero} gives that there is a linear section $s_f$ of $\theta_f$, i.e., a section $s_f$ to $\theta_f$ in $\DBun(X)$. For the $\impliedby$ direction, this is immediate because linear sections are, in particular, sections upon forgetting from $\DBun(X)$ to $\Cscr_{/X}$. The final claim of the corollary follows from an immediate application of Proposition \ref{Prop: The lifting property characterization of TSmooth}. 
\end{proof}

\begin{corollary}\label{Cor: SMan has linear sections}
Each of the tangent categories $\SMan, \CAlg{R}$, $\CAlg{R}^{\op}$, and $\Sch_{/S}$ for $R$ a commutative ring have the property that $f:X \to Y$ is a split $T$-submersion if and only if $\theta_f$ admits a linear section.
\end{corollary}
\begin{proof}
Each of the above tangent categories are Rosick{\'y}, so simply apply Corollary \ref{Cor: For a Rosicky tan cat T submersion on p carrable if and only if linear section to descent} in each case.
\end{proof}


\section{Tangent {\'E}tale Morphisms}

We now discuss the final topic of the paper: that of $T$-{\'e}tale morphisms. These are the morphisms in a tangent category which play the role of local diffeomorphisms in differential geometry\footnote{It should be noted, however, that it is not the case that such $T$-{\'e}tale morphisms coincide with all notions of local homeomorphism one may want to consider in all tangent categories, but instead something different. We show below, for instance, that there are monic $T$-{\'e}tale maps in $\CAlg{R}^{\op}$ which can be non-open closed immersions and hence complementary to being a ``local homeomorphism'' of schemes.} In fact, in \cite[Theorem 3.45]{GeoffMarcelloTSubmersionPaper} the authors show that the monic $T$-{\'e}tale maps which are also display morphisms form a class of maps which are analogous to open embeddings of smooth submanifolds in differential geometry\footnote{We will see below, however, that this analogy fails at the level of schemes in the sense that some monic $T$-{\'e}tale morphisms fail to be open immersions of schemes (and in fact can even be \emph{closed} immersions of schemes).}. Because of the importance of open immersions and also because {\'e}tale morphisms give the local flavour of a category of geometric objects, we are thus obliged to study $T$-{\'e}tale morphisms in the same systematic way as we studied $T$-immersions, $T$-submersions, split $T$-submersions, and $T$-unramified morphisms.

\subsection{{\'E}tale Moprphisms and Examples}
Our focus in this subsection is to get to know $T$-{\'e}tale maps based on their structure, how they incarnate in our four main families of examples, and what permanence/stability properties they have. We are in particular interested in how they relate to, and may be used with, $T$-immersions, $T$-submersions, and split $T$-submersions. We expect that future applications of these ideas will require an intimate knowledge of what being {\'e}tale in this tangent-categorical context buys.

\begin{definition}[{\cite[Definition 35]{GeoffJSReverse}, \cite[Definition 2.27]{GeoffMarcelloTSubmersionPaper}}]\label{Defn: TEtale in a tangent category}
In a tangent category $\Cscr$, a map $f:X \to Y$ is \emph{$T$-{\'e}tale} when the naturality square for $p$ at $f$
\[
\begin{tikzcd}
TX \ar[r]{}{Tf} \ar[d, swap]{}{p_X} & TY \ar[d]{}{p_Y} \\
X \ar[r, swap]{}{f} & Y
\end{tikzcd}
\]
is a $T$-pullback.
\end{definition}

The following proposition is immediate from definition. They allow us to characterize $T$-{\'e}tale maps as maps which are both $T$-immersions and split $T$-submersions in any tangent category and, when a map is $p$-carrable, characterize being $T$-{\'e}tale in terms of the horizontal descent being an isomorphism.

\begin{proposition}\label{Prop: Classify of Tetales}
In a tangent category $\Cscr$:
\begin{enumerate}[{\em (i)}] 
    \item A map $f$ is $T$-{\'e}tale if and only if $f$ is both a strong $T$-immersion and a split $T$-submersion;
    \item A $p$-carrable map $f$ is $T$-{\'e}tale if and only if $\theta_f$ is an isomorphism. 
    \item Any $T$-{\'e}tale morphism is $p$-carrable with $f^{\ast}(TY) \cong TX$. 
    \item If $f$ is a $p$-carrable and $0$-carrable map then $f$ is $T$-{\'e}tale if and only if $f$ is a $T$-submersion and either a $T$-immersion or strong a $T$-immersion.
\end{enumerate}
\end{proposition}
\begin{proof}
For Statement $(i)$, this is immediate from the fact that pullbacks are pre-pullbacks which are weak pullbacks and vice-versa. For Statement $(ii)$ this immediate by uniqueness of pullbacks up to isomorphism. Statement $(iii)$ follows immediately from the fact that if
\[
\begin{tikzcd}
TX \ar[r]{}{Tf} \ar[d, swap]{}{p_X} & TY \ar[d]{}{p_Y} \\
X \ar[r, swap]{}{f} & Y
\end{tikzcd}
\]
is a pullback, $f^{\ast}(TY) \cong TX$ by definition. 

For Statement $(iv)$ we observe the following. First, note that because $f$ is a $T$-submersion, the morphism
\[
\theta_f:TX \to f^{\ast}(TY)
\]
is a regular epimorphism in $\DBun(X)$. Additionally, for any flavour of $T$-immersion, it follows that $\theta_f$ is monic in $\DBun(X)$. Finally, 
\end{proof}

Below we record the fact that at least for $p$-carrable morphisms $f$, asking for the naturality square of $p$ at $f$ to be a $T$-pullback is unnecessary. In such a case we need only ask for the naturality square of $p$ at $f$ to be a pullback in order to derive that $f$ is $T$-{\'e}tale.
\begin{proposition}
If $f:X \to Y$ is $T$-{\'e}tale and $p$-carrable, then
\[
\begin{tikzcd}
TX\ar[r]{}{Tf} \ar[d, swap]{}{p_X} & TY \ar[d]{}{p_Y} \\
X\ar[r, swap]{}{f} & Y
\end{tikzcd}
\]
is a $T$-pullback.
\end{proposition}
\begin{proof}
Because $\theta_f$ is an isomorphism in this case, applying $T^m$ gives the commuting diagram
\[
\begin{tikzcd}
T^{m+1}X \ar[dr]{}{T^m\theta_f} \ar[drr, bend left = 20]{}{T^{m+1}f} \ar[ddr, swap, bend right = 20]{}{(T^m \ast p)_X} \\
 & T^m(f^{\ast}TY) \ar[r]{}{T^m\pi_1} \ar[d, swap]{}{T^m\pi_0} & T^{m+1}Y \ar[d]{}{(T^m \ast p)_Y} \\
 & T^mX \ar[r, swap]{}{T^mf} & T^mY
\end{tikzcd}
\]
with $T^m\theta_f$ an isomorphism for all $m \in \N$. However, as there is a natural isomorphism $T^m(f^{\ast}TY) \cong T^m(f)^{\ast}(T^{m+1}Y)$, the result follows.
\end{proof}

We now present various examples of the $T$-{\'e}tale maps in our main examples. 

\begin{example}\label{Example: Section Etale: TEtale in SMan}
In $\SMan$, the $T$-{\'e}tale maps are precisely the local diffeomorphisms. This is because a map $f:X \to Y$ is a local diffeomorphism is exactly when $f$ is a $T$-immersion and a $T$-submersion, and hence $T$-{\'e}tale if and only if it is a local diffeomorphism. This gives a clean structural proof of the fact that in $\SMan$ local diffeomorphisms are exactly the smooth morphisms whose tangent bundle projection naturality square is a pullback.
\end{example}

\begin{example}\label{Example: Section Etale: TEtale in CAlg}
Let $R$ be a commutative rig. Then in $\CAlg{R}$, the $T$-{\'e}tale maps are precisely the isomorphisms of $R$-algebras. This is seen as follows: for any morphism $f:A \to B$ of commutative $R$-algebras, recall from Example \ref{Example: Relative tangent bundle in CAlg} that the horizontal descent $\theta_f:A[\epsilon] \to A \ltimes B$ is given by $a + x\epsilon \mapsto (a,f(x))$. Now $\theta_f$ is an isomorphism if and only if $f$ is an isomorphism in the first place.
\end{example}

\begin{example}\label{Example: Section Etale: TEtale in CAlgop}
Let $R$ be a commutative rig. In $\CAlg{R}^{\op}$, a map $f^{\op}:B \to A$ (so a map $f: A \to B$ in $\CAlg{R}$) is $T$-{\'e}tale if and only if the map $\Kah{A}{R} \otimes_A B \to \Kah{B}{R}$ is an isomorphism of $B$-modules (which follows from Examples \ref{Example: Timmersions in CAlgop} and \ref{Example: Classification of TSubmersions in CAlgRop}). This follows from the equivalence of categories $\DBun(A) \simeq \Mod{A}^{\op}$ of \cite{GeoffJSDiffBunComAlg} together with  the isomorphism of rigs $\Sym_{B}(\Kah{A}{R} \otimes_A B) \cong \Sym_A(\Kah{A}{R}) \otimes_A B$. Note that in particular, when $R$ is a commutative ring this says that the $T$-{\'e}tale maps in $\mathbf{AffSch}_{/R}$ are exactly the formally unramified maps of affine schemes which are also formally smooth relative to $R$. In particular, every formally unramified map $!_{A}: \Spec A \to \Spec R$ over the terminal object is $T$-{\'e}tale in $\mathbf{AffSch}_{/R}$ by Example \ref{Example: Nonexample of formally smooth}. Note that it certainly need not be the case that $\Spec A$ is formally smooth over $\Spec R$.
\end{example}

In the case of algebraic geometry, there is an unfortunate overlap in terminology between formally {\'e}tale maps of affine schemes and $T$-{\'e}tale maps of affine schemes. In particular, \emph{these are not the same concept --- being formally {\'e}tale is stronger.} We now show that it is the case that formally {\'e}tale maps are $T$-{\'e}tale but also provide an example of a morphism in $\mathbf{CAlg}_{R}^{\op}$ which is $T$-{\'e}tale but \emph{not} formally {\'e}tale.
\begin{example}\label{Example: Tetale in Schemes}
Let $S$ be a base scheme. In $\Sch_{/S}$, a map $f:X \to Y$ is $T$-{\'e}tale if and only if in the cotangent exact sequence
\[
\begin{tikzcd}
f^{\ast}(\Kah{Y}{S}) \ar[r]{}{u} & \Kah{X}{S} \ar[r] & \Kah{X}{Y} \ar[r] & 0
\end{tikzcd}
\]
the map $u$ is an isomorphism. Now recall that a map $f:X \to Y$ of $S$-schemes is formally {\'e}tale if and only if it is formally smooth and formally unramified. By \cite[Proposition 17.2.3.ii]{EGA44}, since $f$ is formally smooth the cotangent sequence is a short exact sequence in $\QCoh(X)$
\[
\begin{tikzcd}
0 \ar[r] & f^{\ast}(\Kah{Y}{S}) \ar[r] & \Kah{X}{S} \ar[r] & \Kah{X}{Y} \ar[r] & 0
\end{tikzcd}
\]
which is locally split. Furthermore, since $f$ is formally unramified, $\Kah{X}{Y} \cong 0$ and so the cotangent sequence is isomorphic to the exact sequence
\[
\begin{tikzcd}
0 \ar[r] & f^{\ast}(\Kah{Y}{S}) \ar[r] & \Kah{X}{S} \ar[r] & 0 \ar[r] & 0
\end{tikzcd}
\]
which shows that $f^{\ast}(\Kah{Y}{S}) \cong \Kah{X}{S}$. Since $f:X \to Y$ is  $T$-{\'e}tale if and only if $f^{\ast}(\Kah{Y}{S}) \cong \Kah{X}{S}$, it follows that formally {\'e}tale maps are $T$-{\'e}tale.
\end{example}

\begin{example}\label{Example: Tetale in Scheme not formal etale} Here is now an example of $T$-{\'e}tale morphism of schemes which is not formally {\'e}tale. Let $A = R = \mathbb{C}\llbracket t \rrbracket$ and let $B = \mathbb{C}\llbracket t \rrbracket/(t^2)$ with $f:A \to R$ the identity morphism and $g:A \to B$ the quotient map. Then $g$ is not formally {\'e}tale (it is not formally smooth) but $\Kah{A}{R} \otimes_A B \cong 0 \cong \Kah{B}{R}$. Thus $g:A \to B$ is $T$-{\'e}tale but not formally {\'e}tale.
\end{example}
\begin{remark}
Example \ref{Example: Tetale in Scheme not formal etale} above shows that the closed immersion $\Spec \mathbb{C}\llbracket t \rrbracket/(t^2) \to \Spec \mathbb{C}\llbracket t \rrbracket$ is a monic $T$-{\'e}tale morphism in $\Sch_{/\mathbb{C}\llbracket t \rrbracket}$ which is as far away from an open immersion as possible; the only open subset of $\Spec \mathbb{C}\llbracket t\rrbracket$ which contains the image of $\Spec \mathbb{C}\llbracket t\rrbracket/(t^2)$ is the entire scheme $\Spec \mathbb{C}\llbracket t\rrbracket$. Furthermore, this is even a display monic $T$-{\'e}tale morphism, as each category $\Sch_{/S}$ is finitely complete and each tangent functor is a right adjoint.
\end{remark}

\begin{example}\label{EXample: Section Etale: TEtale in CDC}
In a CDC, a map $f:X \to Y$ is $T$-{\'e}tale if and only if $D[f]$ is an isomorphism. This is because by Example \ref{Example: Immersions in CDC} $f$ is a $T$-immersion if and only if $D[f]$ is monic while by Example \ref{Example: TSubmersion CDC} $f$ is a $T$-submersion if and only if $D[f]$ is a retract.
\end{example}

\subsection{The Structure of {\'E}tale Morphisms}
We conclude this paper with the general theory of $T$-{\'e}tale morphisms in a tangent category. Our next task is to prove some stability and permanence properties of $T$-{\'e}tale maps and some of their variations. Of particular interest is Proposition \ref{Prop: Etale reflects immersive/submersive stuff postcompositionally} below, as it shows that when we know a map $g:Y \to Z$ is $T$-{\'e}tale then the any composite $g \circ f$ is $T$-{\'e}tale (respectively a split $T$-submersion or $T$-immersion) if and only if the same is true for $f$. After establishing this, we will focus on monic $T$-{\'e}tale morphisms before concluding with a study of when tangent bundle functors reflect {\'e}tale morphisms.

\begin{proposition}\label{Prop: Tetale stable under composition and base change}
In a tangent category $\Cscr$, let $f:X \to Y$ and $g:Y \to Z$ be $T$-{\'e}tale morphisms and assume that we have a $T$-pullback square:
\[
\begin{tikzcd}
X \times_Y W \ar[r]{}{\pr_0} \ar[d, swap]{}{\pr_1} & X \ar[d]{}{f} \\
W \ar[r, swap]{}{h} & Y
\end{tikzcd}
\]
Then:
\begin{enumerate}[{\em (i)}] 
    \item The composite $g \circ f$ is $T$-{\'e}tale.
    \item The morphism $\pr_1$ is $T$-{\'e}tale.
\end{enumerate}
\end{proposition}
\begin{proof}
For $(i)$ this is the Pullback Lemma in action. For $(ii)$ we observe that this is exactly \cite[Lemma 36]{GeoffJSReverse}.
\end{proof}

\begin{proposition}\label{Prop: Etale reflects immersive/submersive stuff postcompositionally}
In a tangent category $\Cscr$, let $f:X \to Y$ and $g:Y \to Z$ be morphisms in a tangent category for which $g$ is $T$-{\'e}tale. Then:
\begin{enumerate}[{\em (i)}] 
    \item $f$ is a strong $T$-immersion if and only if $g \circ f$ is a strong $T$-immersion.
    \item $f$ is a split $T$-submersion if and only if $g \circ f$ is a split $T$-submersion.
    \item $f$ is $T$-{\'e}tale if and only if $g \circ f$ is $T$-{\'e}tale.
\end{enumerate}
\end{proposition}
\begin{proof}
Fix an $m \in \N$ and consider the diagram:
\[
\begin{tikzcd}
T^{m+1}X \ar[d, swap]{}{(p \ast T^m)_X} \ar[r]{}{T^{m+1}f} & T^{m+1}Y \ar[d]{}{(p \ast T^m)_Y} \ar[r]{}{T^{m+1}g} & T^{m+1}Z \ar[d]{}{(p \ast T^m)_Z} \\
T^mX \ar[r, swap]{}{T^mf} & T^mY \ar[r, swap]{}{T^mg} & TY
\end{tikzcd}
\]
Note that the right-handed square in the rectangle is a pullback by virtue of $g$ being $T$-{\'e}tale. We must show:
\begin{enumerate}
    \item for $(i)$ that the left-handed square is a prepullback if and only if the whole rectangle is;
    \item for $(ii)$ that the left-handed square is a weak pullback if and only if the whole rectangle is;
    \item for $(iii)$ that the left-handed square is a pullback if and only if the whole rectangle is.
\end{enumerate}  
However, in all cases we can reduce proving each item to applying the functor $T^m$ to the diagram and then: for $(i)$ appealing to Lemma \ref{Lemma: The prepullback stacking lemma} for each $m \in \N$; for $(ii)$ by appealing to Lemma \ref{Lemma: Folklore for weak pullbacks} for each $m \in \N$; and for $(iii)$, this is simply the Pullback Lemma applied at each $m \in \N$. 
\end{proof}

As before, if we know also that all maps in sight are $p$-carrable, then we can rephrase the $T$-{\'e}tale permanence properties in terms in order to reflect $T$-submersions.
\begin{proposition}
Let $\Cscr$ be a tangent category and let $f:X \to Y, g:Y \to Z$ be morphisms for which $g \circ f, g,$ and $f$ are all $p$-carrable and $0$-carrable. Assume also that $g$ is $T$-{\'e}tale. Then $f$ is a $T$-submersion if and only if $g \circ f$ is a $T$-submersion.
\end{proposition}
\begin{proof}
Begin by observing that asking for $g \circ f$ to be a $T$-submersion means that we require $\theta_{g \circ f}$ to be a $T$-coequalizer. However, as
\[
\theta_{g \circ f} = \gamma \circ (\id_X \times \theta_g) \circ \theta_f
\]
by Lemma \ref{Lemma: Thetaf and composition}, and both $\gamma$ and $(\id_X \times \theta_g)$ are isomorphisms ($\theta_g$ by virtue of $g$ being $T$-{\'e}tale and  $\gamma$ by Lemma \ref{Lemma: Thetaf and composition}), we see that $\theta_{g \circ f}$ is a $T$-coequalizer if and only if $\theta_f$ is.
\end{proof}

We now give a short classification of the monic $T$-{\'e}tale maps. In particular, we show that monic $T$-{\'e}tale maps are exactly the $T$-monic split $T$-submersions in any tangent category.

\begin{proposition}\label{Prop: Classify of monic Tetales}
In a tangent category $\Cscr$, a map $f: X \to Y$ is monic $T$-{\'e}tale if and only if $f$ is both $T$-monic and a split $T$-submersion.
\end{proposition}
\begin{proof}
$\implies:$ If $f$ is a monic $T$-{\'e}tale morphism then for all $m \in \N$ the diagram
\[
\begin{tikzcd}
T^{m+1}X \ar[r]{}{T^{m+1}f} \ar[d, swap]{}{(T^m \ast p)_X} & T^{m+1}Y \ar[d]{}{(T^m \ast p)_Y} \\
T^mX \ar[r, swap]{}{T^mf} & T^mY
\end{tikzcd}
\]
is a pullback. Since $T$-pullbacks are in particular $T$-weak pullbacks, $f$ is a split $T$-submersion. Finally, a routine induction on $m$ gives that $T^mf$ is monic for all $m \in \N$ (the trick is that $T^0f = f$ is monic by assumption and for each $k \in \N$, $T^{k+1}f$ is the pullback of $T^kf$ against $(T^k \ast p)_Y$). Thus $f$ is also $T$-monic.

$\impliedby:$ If $f:X \to Y$ is $T$-monic and a split $T$-submersion then for all $m \in \N$ the diagram
\[
\begin{tikzcd}
T^{m+1}X \ar[r]{}{T^{m+1}f} \ar[d, swap]{}{(T^m \ast p)_X} & T^{m+1}Y \ar[d]{}{(T^m \ast p)_Y} \\
T^mX \ar[r, swap]{}{T^mf} & T^mY
\end{tikzcd}
\]
is a weak pullback with each map $T^kf$ monic for all $k \in \N$. But then since $T^{m+1}f$ is monic, any diagram of the form
\[
\begin{tikzcd}
Z \ar[drr, bend left = 20]{}{h} \ar[dr, shift left = 1]{}{\psi} \ar[dr, shift right = 1, swap]{}{\varphi} \ar[ddr, swap, bend right = 40]{}{h} \\
 & T^{m+1}X \ar[r]{}{T^{m+1}f} \ar[d, swap]{}{(T^m \ast p)_X} & T^{m+1}Y \ar[d]{}{(T^m \ast p)_Y} \\
 & T^mX \ar[r, swap]{}{T^mf} & T^mY
\end{tikzcd}
\]
has $\varphi = \psi$ by virtue of $T^{m+1}f$ monic. Thus
\[
\begin{tikzcd}
T^{m+1}X \ar[r]{}{T^{m+1}f} \ar[d, swap]{}{(T^m \ast p)_X} & T^{m+1}Y \ar[d]{}{(T^m \ast p)_Y} \\
T^mX \ar[r, swap]{}{T^mf} & T^mY
\end{tikzcd}
\]
is a weak pullback which is a prepullback and hence a pullback. Applying $T^m$ to the naturality diagram at $p$ thus is a pullback for all $m \in \N$. But arguing this for $m = 0$ gives that the naturality square for $p$ at $f$ is a $T$-pullback and hence that $f$ is a monic $T$-{\'e}tale morphism.
\end{proof}

\begin{corollary}\label{Cor: Monic Tetale stable under comp and base change}
In a tangent category $\Cscr$, let $f:X \to Y$ and $g:Y \to Z$ be monic $T$-{\'e}tale morphisms and assume that we have a pullback square
\[
\begin{tikzcd}
X \times_Y W \ar[r]{}{\pr_0} \ar[d, swap]{}{\pr_1} & X \ar[d]{}{f} \\
W \ar[r, swap]{}{h} & Y
\end{tikzcd}
\]
which is preserved by the tangent bundle functor $T$. Then:
\begin{enumerate}[{\em (i)}] 
    \item The composite $g \circ f$ is monic $T$-{\'e}tale.
    \item The morphism $\pr_1$ is monic $T$-{\'e}tale.
\end{enumerate}
\end{corollary}
\begin{proof}
Because monic $T$-{\'e}tale maps are exactly the $T$-monic split $T$-submersion maps by Proposition \ref{Prop: Classify of monic Tetales}, the result follows from the corresponding base change and composition stability results for $T$-monic and $T$-{\'e}tale morphisms.
\end{proof}

\begin{corollary}\label{Cor: T monic etale for p carrable}
In a tangent category $\Cscr$, if a map $f:X \to Y$ is $p$-carrable then $f$ is monic $T$-{\'e}tale if and only if $f$ is $T$-monic and a $T$-submersion. In particular, $f$ is a monic $T$-immersion and a $T$-submersion.
\end{corollary}
\begin{proof}
As $f$ is $T$-monic means that $T^mf$ and $\theta_{T^mf}$ are all monic for all $m \in \N$, and as $f$ being a $T$-submersion means that $\theta_{T^mf}$ is a $T$-cokernel and hence a regular $T$-epimorphism for all $m \in \N$, the result follows by virtue of isomorphisms being exactly the monic regular epimorphisms.
\end{proof}

Our final task is to prove a proposition which has as a consequence that when all maps in a tangent category are $p$-carrable, then all powers of the tangent bundle functor preserve and reflect $T$-{\'e}tale maps. 


\begin{proposition}\label{Prop: tangent bundle functors preserve Tetales}
Let $\Cscr$ be a tangent category such that all morphisms are $p$-carrable. Then for all $m \in \N$, $T^m$ preserves and reflects $T$-{\'e}tale morphisms, that is, $T^{m}f$ is $T$-{\'e}tale if and only if $f$ is $T$-{\'e}tale.
\end{proposition}
\begin{proof} So fix an $m \in \N$ and let $f$ be an arbitrary {\'e}tale map. Applying Corollary \ref{Cor: The other higher powers of T to thetaf in one line} gives rise to isomorphisms $C$ and $\tilde{C}$ for which the identity $\theta_{T^{m}f} = \tilde{C} \circ T^{m}\theta_{f} \circ C$ holds. Because $\tilde{C}$ and $C$ are isomorphisms, $T^{m}\theta_f$ is an isomorphism if and only if $\theta_{T^{m}f}$ is an isomorphism. For the $\implies$ direction: if $T^{m}\theta_f$ is an isomorphism, then applying Corollary \ref{Cor: tangent bundle functors are isomorphism reflecting} gives that $\theta_f$ is an isomorphism as well. For the $\impliedby$ direction: if $\theta_f$ is an isomorphism then so too is $T^{m}(\theta_f)$ by virtue of $T^{m}$ being a functor. But then $\theta_{T^{m}(f)}$ is an isomorphism as well and so $T^{m}f$ is $T$-{\'e}tale.
\end{proof}

\bibliographystyle{amsalpha}      
\bibliography{BundlesAndStuffBib}

\end{document}